\documentclass{amsart}
\usepackage{latexsym}
\usepackage{amsfonts}
\usepackage{amssymb}

\usepackage[all]{xy}

\renewcommand{\geq}{\geqslant}
\renewcommand{\leq}{\leqslant}
\renewcommand{\ngeq}{\ngeqslant}
\renewcommand{\nleq}{\nleqslant}

\newcommand{\sub}[1]{_{\textup{\tiny{#1}}}}

\newcommand{\converges}{\mathord{\downarrow}}
\newcommand{\diverges}{\mathord{\uparrow}}

\newcommand{\sd}{\bigtriangleup}
\newcommand{\sa}{\mathbin{\raisebox{2pt}{$\bigtriangledown$}}}
\renewcommand{\phi}{\varphi}
\renewcommand{\epsilon}{\varepsilon}

\newtheorem{theorem}{Theorem}[section]
\newtheorem{lemma}[theorem]{Lemma}

\newtheorem{corollary}[theorem]{Corollary}

\newtheorem{proposition}[theorem]{Proposition}

\numberwithin{equation}{section}

\theoremstyle{definition}
\newtheorem{definition}[theorem]{Definition}
\newtheorem{question}[theorem]{Question}

\newtheorem{observation}[theorem]{Observation}

\newcommand{\bd}{\mathbf{d}}
\newcommand{\be}{\mathbf{e}}
\newcommand{\gamd}{\gamma_{\mathbf{d}}}
\newcommand{\game}{\gamma_{\mathbf{e}}}
\newcommand{\cd}{\overline{\mathbf{d}}}
\newcommand{\ce}{\overline{\mathbf{e}}}

\newcommand{\urho}{\underline{\rho}}
\newcommand{\orho}{\overline{\rho}}

\newcommand{\restr}{\upharpoonright}
\newcommand{\uhr}{\upharpoonright}

\newcommand{\mcS}{\mathcal{S}}

\newcommand{\mcM}{\mathcal{M}}
\newcommand{\mcB}{\mathcal{B}}
\newcommand{\mcU}{\mathcal{U}}
\newcommand{\mcC}{\mathcal{C}}
\newcommand{\mcP}{\mathcal{P}}
\newcommand{\mcA}{\mathcal{A}}

\DeclareMathOperator{\cm}{c.m.\,}
\DeclareMathOperator{\aev}{a.e.\,}

\author[D. R. Hirschfeldt]{Denis R. Hirschfeldt}

\address{Department of Mathematics\\University of Chicago}

\email{drh@math.uchicago.edu}

\author[C. G. Jockusch, Jr.]{Carl G. Jockusch, Jr.}

\address{Department of Mathematics\\University of Illinois at
Urbana-Cham\-paign}

\email{cjockusc@gmail.com}

\author[P. E. Schupp]{Paul E. Schupp}

\address{Department of Mathematics\\University of Illinois at
Urbana-Cham\-paign}

\email{schupp@illinois.edu}

\thanks{Jockusch and Schupp gratefully acknowledge the support of the
  Simons Foundation Collaboration Grant for
  Math\-e\-ma\-ti\-cians. Hirschfeldt was partially supported by
  grants DMS-1600543 and DMS-1854279 from the National Science
  Foundation. We thank George Barmpalias, Doug Cenzer, Noam Greenberg,
  Aleko Kechris, Joe Miller, Chris Porter, Richard Shore, Anush
  Tserunyan, and Liang Yu for useful comments and responses to
  queries.}



\title[The density metric, distances between
degrees, and reverse math] {Coarse computability, the density metric,
Hausdorff distances between Turing degrees, perfect trees, and
reverse mathematics}

\begin{document}

\begin{abstract}
For $A \subseteq \omega$, the \emph{coarse similarity class} of $A$,
denoted by $[A]$, is the set of all $B \subseteq \omega$ such that the
symmetric difference of $A$ and $B$ has asymptotic density $0$. There
is a natural metric $\delta$ on the space $\mathcal{S}$ of coarse
similarity classes defined by letting $\delta([A],[B])$ be the upper
density of the symmetric difference of $A$ and $B$. We study the
metric space of coarse similarity classes under this metric, and show
in particular that between any two distinct points in this space there
are continuum many geodesic paths. We also study subspaces of the form
$\{[A] : A \in \mathcal U\}$ where $\mathcal U$ is closed under Turing
equivalence, and show that there is a tight connection between
topological properties of such a space and computability-theoretic
properties of $\mathcal U$.

We then define a distance between Turing degrees based on Hausdorff
distance in the metric space $(\mathcal{S},\delta)$. We adapt a proof
of Monin to show that the Hausdorff distances between Turing degrees
that occur are exactly $0$, $1/2$, and $1$, and study which of these
values occur most frequently in the senses of Lebesgue measure and
Baire category. We define a degree $\mathbf{a}$ to be
\emph{attractive} if the class of all degrees at distance $1/2$ from
$\mathbf{a}$ has measure $1$, and \emph{dispersive} otherwise. In
particular, we study the distribution of attractive and dispersive
degrees. We also study some properties of the metric space of Turing
degrees under this Hausdorff distance, in particular the question of
which countable metric spaces are isometrically embeddable in it,
giving a graph-theoretic sufficient condition for embeddability.

Motivated by a couple of issues arising in the above work, we also
study the computability-theoretic and reverse-mathematical aspects of
a Ramsey-theoretic theorem due to Mycielski, which in particular
implies that there is a perfect set whose elements are mutually
$1$-random, as well as a perfect set whose elements are mutually
$1$-generic.

Finally, we study the completeness of $(\mathcal S,\delta)$ from the
perspectives of computability theory and reverse mathematics.
\end{abstract}

\maketitle

\clearpage

\tableofcontents

\section{Introduction}
\label{intro}

We call two sets $A, B \subseteq \omega$ \emph{coarsely similar} if
their symmetric difference $A \sd B$ has asymptotic density $0$, i.e.,
$\lim_n \frac{|(A \sd B) \cap [0,n)|}{n}=0$. This is an equivalence
relation on subsets of $\omega$ that arises naturally in studying
coarse computability (see e.g.\ \cite{JSsurv}), as well as in other
areas of computability theory, for example in the work of Greenberg,
Miller, Shen, and Westrick \cite{GMSW}, who showed that a real has
effective Hausdorff dimension $1$ if and only if it is coarsely
similar to a $1$-random real (i.e., one that is random in the sense of
Martin-L\"of). A set $A$ is called \emph{coarsely computable} if it is
coarsely similar to some computable set. Let $\mathcal{S}$ be the set
of equivalence classes of the above equivalence relation.  We consider
a natural metric $\delta$ on $\mathcal{S}$ given by $\delta([A], [B])
= \overline{\rho}(A \sd B)$, where $[A]$ is the equivalence class of
$A$, and $\overline{\rho}$ is upper density, and explore several ways
in which it can interact with Turing reducibility.

We first study the metric space $(\mathcal{S}, \delta)$.  This space
is non-separable and non-compact but is complete and contractible. We
show that any two distinct points in $\mathcal{S}$ are joined by
continuum many geodesic paths. We also study the topological
properties of subspaces of this metric space generated by collections
of sets that are closed under Turing equivalence. We establish close
connections between topological properties of these subspaces and
computability-theoretic properties of their generating collections of
sets.

The second main topic of this paper is the application of the metric
$\delta$ to the study of Turing degrees. Since $\delta$ is bounded,
the notion of Hausdorff distance between subsets of $\mathcal S$
yields a metric on the closed subsets of $(\mathcal{S}, \delta)$. We
define the \emph{closure} $\cd$ of a Turing degree $\mathbf{d}$ to
be the closure of $\{[A] : A \in \mathbf{d}\}$ in $(\mathcal{S},
\delta)$. We show that the Hausdorff distance between closures of
Turing degrees yields a metric $H$ on the set of Turing degrees
$\mathcal D$.  We show how to calculate Hausdorff distances between
closures of Turing degrees using a relativized version of the function
$\Gamma$ (see \cite{ACDJL}), which is based on a generalized version
of coarse computability. We adapt a proof of B. Monin \cite{M} to
show that the distances that occur are exactly $0$, $1/2$, and $1$,
and determine which distances between Turing degrees occur most
frequently in terms of Lebesgue measure and Baire category. We define
a degree $\mathbf{a}$ to be \emph{attractive} if the class of all
degrees at distance $1/2$ from $\mathbf{a}$ has measure $1$, and
\emph{dispersive} otherwise. In particular, we study the distribution
of attractive and dispersive degrees.

We also study which countable metric spaces are isometrically
embeddable in $(\mathcal D,H)$. If $\mathcal M$ is a countable metric
space with all distances equal to $0$, $1/2$, or $1$, let $G_{\mathcal
M}$ be the graph whose vertices are the points of $\mathcal M$, such
that two points are joined by an edge if and only if the distance
between them is $1$. We show that if $G_{\mathcal M}$ is a
comparability graph (i.e., there is a partial order $\prec$ on the set
of vertices of $G_{\mathcal M}$ such that there is an edge between
distinct vertices $x$ and $y$ if and only if they are
$\prec$-comparable), then $\mathcal M$ is isometrically embeddable in
$(\mathcal D, H)$. We also show that the complement of the graph
$G_{(\mathcal D,H)}$ (where pairs of degrees at distance $1/2$ are
joined by an edge) is a connected graph with diameter at least $3$ and
at most $4$.

The interplay between randomness and genericity (i.e., between being
typical in the sense of measure and being typical in the sense of
category) is a recurring theme in this paper. As we will see, having
distance $1/2$ in the Hausdorff metric is related to mutual
randomness, while having distance $1$ is related to relative
genericity. Thus, for instance, both our discussion of attractive and
dispersive degrees in Section \ref{freqsec} and that of isometric
embeddings into $(\mathcal D,H)$ in Section \ref{isosec} rely heavily
on randomness and genericity. For example, one of the results we
obtain in the former section is that if $A$ is weakly $2$-generic and
$B$ is $2$-random, then $B$ computes a set that is weakly $1$-generic
relative to $A$, which, as we will see, implies that $H(A,B)=1$.

At the end of Section \ref{propsec}, we discuss a Ramsey-theoretic
theorem due to Mycielski \cite{My}, which in particular implies that
there is a perfect set whose elements are mutually random (say in the
sense of Martin-L\"of randomness), as well as a perfect set whose
elements are mutually generic. This theorem has several interesting
computability-theoretic and reverse-mathematical aspects, which we
discuss in Section \ref{mycsec}. In particular, we show that there is
a $\emptyset'$-computable perfect tree such that the join of any
nonempty finite collection of pairwise distinct paths is $1$-random.

In Section \ref{compsec}, we discuss the computability-theoretic and
reverse-mathematical strength of the fact that $(\mathcal S,\delta)$
is complete. We finish with a section containing several open
questions.

After the first part of Section \ref{propsec}, we assume familiarity
with computability theory, as in \cite{Snew}, for instance. We will
also use basic concepts and results from the theory of algorithmic
randomness, and refer to \cite{DH} for details. This book also
includes sections on genericity (Section 2.24) and on interactions
between randomness and genericity (Section 8.21). Particularly useful
is van Lambalgen's Theorem, which implies that if $A$ and $B$ are
$1$-random and $A$ is $1$-random relative to $B$, then $B$ is
$1$-random relative to $A$. This fact is one of the reasons that we
work with $1$-randomness below, even though for some results, weaker
notions of algorithmic randomness might suffice. We will also use the
fact that the analog of van Lambalgen's Theorem for $1$-genericity in
place of $1$-randomness holds, as shown by Yu \cite{Yuvl}. In Sections
\ref{mycsec} and \ref{compsec} we will also assume some knowledge of
reverse mathematics, as in \cite{Simpson}, for instance.

We need to note several basic definitions. First we recall our
definition of when two sets are ``essentially the same''.  Since we
are working with subsets of $\omega$, we use classical asymptotic
density from number theory.  For $A \subseteq \omega$ and $n \in
\omega$, let $A \upharpoonright n = \{m < n : m \in A\}$ (which we
also identify with the binary string $\sigma$ of length $n$ such that
$\sigma(m)=1$ if and only if $m \in A$).

\begin{definition}
If $A \subseteq \omega$, then, for $ n \geq 1$, the \emph{density of
$A$ at $n$} is
\[
\rho_n (A) = \frac{ |A \upharpoonright n | }{n}.
\]
The \emph{asymptotic density}, $\rho(A)$, of $A$ is $\lim_n
\rho_n (A)$ if this limit exists.

While the limit for density does not exist in general, the
\emph{upper density} $ \overline{\rho}(A) = \limsup_n \rho_n (A)$ and
the \emph{lower density} $ \underline{\rho}(A) = \liminf_n \rho_n(A)$
always exist.
\end{definition}

Although easy, the following lemma is basic. Here we write $\neg A$
for the complement of $A$ since we will use overlines for closures.

\begin{lemma} \label{lowerupper}
If $A \subseteq \omega$ then $ \underline{\rho} (A) = 1 -
\overline{\rho}( \neg A) $.
\end{lemma}

\begin{proof}
Note  that $\rho_n(A) = 1 - \rho_n(\neg A)$ for all $n \geq 1$.  The
lemma follows by taking the lim inf of both sides of this equation.
\end{proof}

We identify sets and their characteristic functions.  The symmetric
difference $ A \sd B = \{ n: A(n) \ne B(n) \} $ is the subset of
$\omega$ where $A$ and $B$ disagree. There does not seem to be a
standard notation for the complement of $ A \sd B $, which is $ \{ n:
A(n) = B(n)\} $, the \emph{symmetric agreement} of $A$ and $B$.  We
find it useful to use $ A \sa B $ to denote $\{n : A(n) = B(n)\}$.

\begin{definition}
\label{def:COARSESIM}
If $A, B \subseteq \omega$, then $A$ and $B$ are \emph{coarsely
similar}, written $A \thicksim\sub{c} B$, if the density of the
symmetric difference of $A$ and $B$ is $0$, that is, $ \rho(A \sd B) =
0$.  Equivalently, $\rho(A \sa B) = 1$.  Given $A$, any set $B$ such
that $B \thicksim\sub{c} A$ is called a \emph{coarse description} of
$A$.
\end{definition}

It is easy to check that coarse similarity is indeed an equivalence
relation on $\mcP(\omega)$.  We write $[A]$ for the coarse similarity
class of the set $A$.  Let $\mcS$ denote the set of all coarse
similarity classes.  There is a natural pseudo-metric on
$\mcP(\omega)$.

\begin{definition}
\label{def:METRIC}
If $A, B \subseteq \omega$,  let $\delta(A,B) = \overline{\rho}(A \sd
B)$.
\end{definition}

A Venn diagram argument shows that $\delta$ satisfies the triangle
inequality and is therefore a pseudo-metric on subsets of $\omega$.
Since $\delta(A,B) = 0$ exactly when $A$ and $B$ are coarsely similar,
$\delta$ is actually a metric on the space $\mcS$ of coarse similarity
classes, and we now work in the metric space $(\mathcal{S}, \delta)$.

Recall that we apply Lebesgue measure in computability theory by
regarding the set $A \subseteq \omega$ as corresponding to the binary
expansion defined by its characteristic function. For any set $A$, the
Strong Law of Large Numbers implies that $\mcM = \{ B : \delta(B,A) =
1/2 \}$ has measure $1$ in $\mcP(\omega)$. It follows that the
complement of $\mcM$ has measure $0$, which implies that any coarse
similarity class has measure $0$.

This metric has probably been rediscovered many times. It has recently
been used on subsets of $\mathbb{Z}$ to study cellular automata on the
line as dynamical systems. See \cite{BFK} and \cite{HM}. The automata
theory literature ascribes this metric to Besicovitch and cites his
well-known book \cite{Bes} on almost periodic functions as a
reference. At least in his book, however, Besicovitch does not
consider arbitrary subsets of $\omega$ or any metric on them. We will
refer to this metric as the \emph{density metric}.

Of course, $\mcP(\omega)$ can be made into an abelian group of exponent
$2$ by defining $A + B = A \sd B$.  The operation $+$ is well-defined
as a map (also denoted $+$) from $\mcS^2$ to $\mcS$ and it is clear
that $+$ is continuous in the metric on $\mcS$. Therefore
$\mathcal{S}$ has the structure of a topological group by defining
$[A] + [B] = [A \sd B]$. The important property of this topological
group is the following.

\begin{observation}
\label{isomobs}
Let $C \subseteq \omega$. Define the translation map $\tau_C: \mcS \to
\mcS$ by $\tau_C([A]) = [A] +[C]$. Since the operation $+$ is $\sd$,
\[
(A  + C)  +  (B + C) = A + B
\]
is the statement that
\[
(A \sd C) \sd (B \sd C) = A \sd B,
\]
which holds. Thus $\tau_C$ is an isometry, and $\mathcal{S}$ acts on itself 
by isometries.  
\end{observation}

Considering this topological group and its actions is of interest to
computability theory. For instance, Kuyper and Miller \cite{KM}
studied set stabilizers for these actions for the classes of
$1$-random and $1$-generic sets.

\section{Properties of the metric space $(\mcS, \delta)$}
\label{propsec}

In this section, we explore the properties of the metric space $(\mcS,
\delta)$. No knowledge of computability theory is needed to read most
of this section, although our goal later will be to explore
connections between $(\mcS, \delta)$ and computability theory.
  
We first show that the space $\mcS$ is non-separable and non-compact
in a very strong sense.

The following sequence of intervals is basic to studying coarse
computability.

\begin{definition}
\label{idef}
Let $I_n = [n!, (n+1)!)$, and let $\mathcal I(A) = \bigcup_{n \in A}
I_n$.
\end{definition}

\begin{theorem}
\label{compthm}
\begin{enumerate}
    
\item If $\mcA = \{ [A_i] \}$ is any countable subset of $\mcS$, then
there is a class $[B]$ such that $\delta([B], [A_i]) = 1$ for all $i$.  

\item There is a subset $\mcU$ of $\mcS$ of size continuum such that
the members of $\mcU$ are pairwise at distance $1$ from each other.

\end{enumerate}
\end{theorem}

\begin{proof}
(1) Let $\langle i,m \rangle$ denote the pairing function from $\omega
\times \omega \to \omega$.  Define $B$ as follows.  If $n = \langle
i,m \rangle$ then $B$ agrees with the complement $\neg A_i$ of
$A_i$ on $I_n$.  So $\rho_{(n+1)! - 1}(B \sa A_i) \leq
1/(n+1)$. Thus $\urho(B \sa A_i) = 0 = 1 - \orho(B \sd A_i)$, and
$\delta([B],[A_i]) = 1$.

(2) Let $\mathcal{C}$ be a collection of continuum many infinite
subsets of $\omega$ such that any two distinct sets in $\mathcal{C}$
have finite intersection.  For example, identify the nodes of the
infinite perfect binary tree $T$ with natural numbers, and let
$\mathcal{C}$ be the set of paths through $T$. Then let $\mathcal{U}
= \{[\mathcal I(A)] : A \in \mathcal{C}\}$. Argue as in the proof of
(1) that any two distinct elements of $\mathcal{U}$ are at distance
$1$ from each other.  Clearly $\mathcal{U}$ has size continuum.
\end{proof}

The first part of the following corollary was shown by Blanchard,
Formenti, and K\r{u}rka \cite{BFK}.

\begin{corollary} 
The space $(\mcS,\delta)$ is not compact. Indeed, open covers need
not have countable subcovers, so $(\mcS,\delta)$ is not Lindel\"{o}f.
\end{corollary}

\begin{proof}
Consider any cover of $\mcS$ by open balls, all of whose radii are
less than $1$.  Let $\mcC = \{\mathcal B_i\}$ be any countable subset
of the cover.  By the previous theorem, there is a point $[P]$ at
distance $1$ from all the centers of the $\mathcal B_i$ and thus $[P]
\notin \bigcup_{\mathcal B_i \in \mcC} \mathcal B_i$.
\end{proof}

We now show that $(\mcS, \delta)$ has several ``good'' properties,
namely it is complete, contractible to a point, and geodesic in a
strong sense.  Completeness and pathwise connectedness (using
$\mathbb{Z}$) were shown by completely different arguments by
Blanchard, Formenti, and K\r{u}rka \cite{BFK}, in a paper on automata
theory (see also \cite{HM}).  We give ``computability-theoretic''
proofs.

\begin{definition}
\label{jddef}
Let $J_k$ be the interval $[2^k - 1, 2^{k+1} - 1)$.  For any set $C$,
let $d_k(C)$ be the density of $C$ on $J_k$, that is, $d_k(C) =
\frac{|C \cap J_k|}{2^k}$.
\end{definition}

The following lemma, Lemma 5.10 of \cite{HJMS}, relates
$\overline{\rho}(C)$ to  $\overline{d}(C) = \limsup_k d_k(C)$. We
include the proof for the sake of self-containment.

\begin{lemma}[Hirschfeldt, Jockusch, McNicholl, and Schupp
\cite{HJMS}]
\label{factor2} 
For every set $C$, 
\[
\frac{\overline{d}(C)}{2} \leq \overline{\rho}(C) \leq 2
\overline{d}(C).
\]
\end{lemma}

\begin{proof} 
For all $k$,
\[
d_k(C) = \frac{|C \cap J_k|}{2^k} \leq \frac{|C \upharpoonright
2^{k+1}|}{2^k} = 2 \rho_{2^{k+1}} (C).
\] 
Dividing both sides of this inequality by $2$ and then taking the lim
sup of both sides yields the fact that $\frac{\overline{d}(C)}{2} \leq
\overline{\rho}(C)$.

To prove that $\overline{\rho}(C) \leq 2 \overline{d}(C)$, assume that
$k - 1 \in J_n$, so $2^n \leq k < 2^{n+1}$.  Then
\begin{multline*}
\rho_k(C) = \frac{|C \upharpoonright k|}{k} \leq 
\frac{|C \upharpoonright (2^{n+1} - 1)|}{2^n} = 
\frac{\sum_{0 \leq i \leq n}|C \upharpoonright J_i|}{2^n} \\ =
\frac{\sum_{0 \leq i \leq n} 2^i d_i(C)}{2^n} < 2 \max_{i \leq n}
d_i(C).
\end{multline*}

Let $\epsilon > 0$ be given.  Then $d_i(C) < \overline{d}(C) +
\epsilon$ for all sufficiently large $i$.  Hence there is a finite set
$F$ such that $d_i(C \setminus F) < \overline{d}(C \setminus F) +
\epsilon$ for \emph{all} $i$.  Then, by the above inequality applied
to $C\setminus F$, we have $\rho_k(C\setminus F) < 2
(\overline{d}(C\setminus F) + \epsilon)$ for all $k$, so
$\overline{\rho}(C\setminus F) \leq 2\overline{d}(C\setminus F)$.  As
$\overline{\rho}$ and $\overline{d}$ are invariant under finite
changes of their arguments and $\epsilon > 0$ is arbitrary, it follows
that $\overline{\rho}(C) \leq 2 \overline{d}(C)$.
\end{proof}

\begin{theorem}[Blanchard, Formenti, and K\r{u}rka \cite{BFK}]
\label{complthm}
The space $(\mathcal{S},\delta)$ is complete.
\end{theorem}

\begin{proof} 
Let $[C_0], [C_1], \dots$ be a Cauchy sequence of similarity classes
with respect to the density metric $\delta$.  By passing to a
subsequence, we can assume that if $m<n$ then $\delta(C_m, C_n) <
2^{-m-1}$, so that by the above lemma, $\overline{d}(C_m \sd C_n) <
2^{-m}$. Then there is a sequence $0=k_0<k_1<\cdots$ such that for all
$m<n$, if $i \geq k_n$ then $d_i(C_m \sd C_n) < 2^{-m}$. Let $C$ be
the unique set such that $C$ and $C_n$ agree on the interval $J_i$ for
each $i \in [k_n,k_{n+1})$.

Fix $m$.  For every $n > m$ and $i \in [k_n,k_{n+1})$, we have $d_i
(C_m \sd C) = d_i (C_m \sd C_n)<2^{-m}$, since $C$ and $C_n$ agree
on the interval $J_i$.  Hence $\overline{d}(C_m \sd C) = \limsup_i
d_i (C_m \sd C) \leq 2^{-m}$.  So by the above lemma we have
\[
\delta(C_m, C) = \overline{\rho}(C_m \sd C) \leq 2 \overline{d}(C_m
\sd C) \leq 2^{-m + 1}.
\]
Hence $\lim_m \delta(C_m, C) = 0$ and the sequence $\{[C_m]\}$
converges to $[C]$.
\end{proof}

The completeness of $(\mathcal S,\delta)$ raises the question of how
difficult it is to obtain the limit of a Cauchy sequence from the
sequence, in the senses of computability theory and reverse
mathematics. As this issue is not directly related to what we will
pursue in the next few sections, we leave it to Section \ref{compsec}
below.

The first part of the following theorem was shown by Blanchard,
Formenti, and K\r{u}rka \cite{BFK}.

\begin{theorem}
\label{conthm}
The space $(\mcS,\delta)$ is pathwise connected. Indeed, it is
contractible.
\end{theorem}

\begin{proof}
First consider how we would make a path from $[\emptyset]$ to
$[\omega]$.

For every real $r \in [0, 1]$ in the unit interval, we will define a
set $C_r$.  We will have $C_s \subseteq C_r$ if $s \leq r$, and $C_r$
will have density $r$.  From this property it follows that if $0 \leq
s \leq r \leq 1$ then
\[
\delta(C_r, C_s) = \orho(C_r \sd C_s) = \orho(C_r \setminus C_s) =
\rho(C_r) - \rho(C_s) = r- s.
\]
Define $p:[0,1] \to \mcS$ by $p(r) = [C_r]$.  Then $p$ preserves the
metric and hence is continuous. Furthermore, we will have $C_0 =
\emptyset$ and $C_1 = \omega$.

Here is the construction of $C_r$.  Partition $\omega$ into
consecutive intervals $L_1, L_2, \dots$ with $|L_i| = i$ for all $i$.
Given $r \in [0,1]$, let $C_r$ be such that $C_r \cap L_i = [m_i , m_i
+ \lfloor ri \rfloor)$, where $m_i$ is the least element of
$L_i$. That is, $C_r \cap L_i$ is the longest initial segment of $L_i$
whose density within $L_i$ does not exceed $r$.
     
We now verify that $C_r$ has the properties claimed above.  It is
obvious that $C_s \subseteq C_r$ if $s \leq r$.  It remains to be
shown that $\rho(C_r) = r$ for all $r \in [0, 1]$.  Fix $r$ and let
$d_i$ be the density of $C_r$ on $L_i$, that is, $d_i = |C_r \cap L_i|
/ i$.  Since $|C_r \cap L_i| = \lfloor ri \rfloor$, we have $r -
\frac{1}{i} \leq d_i \leq r$.  It follows that $\lim_i d_i = r$.

For every $i$, we have $|C_r \upharpoonright m_i| = \sum_{j < i} |C_r
\cap L_j| = \sum_{j < i} j d_j$.  Hence $\rho_{m_i}(C_r) = \sum_{j <
i} j d_j / m_i$ is the weighted average of $d_1, d_2, \dots, d_{i -
1}$, where $d_j$ has weight $j$.  It follows that $\lim_i
\rho_{m_i} (C_r) = r $. 
              
If $n \in L_i$ then
\[
\frac{{m_i} \rho_{m_i} (C_r)}{m_i + i} \quad \leq \quad \rho_{n+1}
(C_r) \quad \leq \quad \frac{m_i \rho_{m_i} (C_r) + i}{m_i}.
\]              
Now let $n$ approach infinity, so that $i$ approaches infinity also.
Since $\lim_i \rho_{m_ i} (C_r) = r$ and $\lim_i i / m_i = 0$, we
have $\rho(C_r) = \lim_n \rho_{n+1}(C_r) = r$.  This completes the
proof that there is a path from $[\emptyset]$ to $[\omega]$.
  
For arbitrary $A \subseteq \omega$ define $A_r = A \cap C_r$ where
$C_r$ is as above.  To obtain a path from $[\emptyset]$ to $[A]$,
define $p_A:[0,1] \to \mcS$ by $p_A(r) = A_r$.  Then $p_A(0) =
[\emptyset]$ and $p_A(1) = [A]$.  To show that $p_A$ is continuous, it
suffices to show that $\delta( p(s),p(r) ) \leq |r - s|$ for all $s,r
\in [0,1]$.  But this holds since
\begin{multline*}
\delta( p(s),p(r) ) = \orho( (A \cap C_s) \sd  (A \cap C_r) ) =
\orho( A \cap (C_s \sd C_r) ) \\ \leq \orho( C_s \sd C_r) =
\delta(C_s, C_r) = |r - s|.
\end{multline*}

For a path between two arbitrary classes $[A]$ and $[B]$, use the
reverse of the path from $[\emptyset]$ to $[A]$ followed by the path
from $[\emptyset]$ to $[B]$.

To check contractibility, define $\Phi: \mcS \times [0,1] \to \mcS$ by
$\Phi([A], r) = [A_r]$.  By definition, for all $[A]$, we have
$\Phi([A], 1) = [A]$ and $\Phi([A], 0) = [\emptyset]$.  To verify that
$\Phi$ is continuous check that
\[ 
A_r \sd B_r \subseteq (A \sd B) \cup (C_r \sd C_s).
\]
\end{proof}

The path just constructed from $[\emptyset]$ to an arbitrary point
$[A]$ is called the \emph{uniform path} between $[\emptyset]$ and
$[A]$.  In general, these uniform paths need not be geodesics,
although the path between $[\emptyset]$ and $[\omega]$ is a geodesic.
It turns out that $(\mcS, \delta)$ is actually a geodesic metric
space.

\begin{theorem} 
\label{geodesic}
The space $(\mcS, \delta)$ is a geodesic metric space.
\end{theorem}

\begin{proof}
As before, we first want to construct a geodesic path between
$[\emptyset]$ and an arbitrary point $[A]$.  The idea is to relativize
the previous construction by working within the set $A$. We suppose
that $[A] \ne [\emptyset]$.  Let $L_0^A,L_1^A, \cdots$ partition $A$
into successive disjoint intervals with $|L_n^A| = n$ and $\max(L_n^A)
< \min(L_{n+1}^A)$ for all $n$.

Define $f:[0,1] \to \mcS$ as follows:
\[ 
f(s) = \left[\bigcup_n F_n^{s,A}\right],
\]
where $F_n^{s,A}$ consists of the first $\lfloor sn \rfloor$ many
elements of $L_n^A$.

It is clear that $f(0) = [\emptyset]$ and $f(1) = [A]$. We want to
show that the length of the constructed path is indeed
$\delta([\emptyset],[A])$. To do that, we show that if $0 \leq s \leq
t \leq 1$ then
\begin{equation}
\label{geoeq} 
\delta([f(s)],[f(t)]) = (t-s)\delta([\emptyset],[A]).
\end{equation}

Let $k_n^A = \max{L_n^A} + 1$. The idea is to show first that
\[ 
\frac{|f(t) \upharpoonright k_n^A| - |f(s) \upharpoonright k_n^A|}{k_n^A}
\approx \frac{t|A\upharpoonright k_n^A| - s|A\upharpoonright k_n^A|}{k_n^A},
\]
and the error in this approximation approaches $0$ as $n \to \infty$.

By the definition of $k_n^A$, we have $A \upharpoonright k_n^A =
\bigcup_{i \leq n} (A \upharpoonright L_i^A)$.  Since the $L_i^A$ are
pairwise disjoint we have
\[ 
|A\upharpoonright k_n^A|  = \sum_{i \leq n} |A\upharpoonright L_i^A| =
\sum_{i \leq n} i = \frac{n(n+1)}{2}. 
\]
Then
\[ 
|f(t) \upharpoonright k_n^A|  = \sum_{i \leq n} |f(t) \upharpoonright
L_i^A| = \sum_{i \leq n} \lfloor it \rfloor. 
\]
So
\[ 
\sum_{i \leq n} ( it - 1) \leq  |f(t)\upharpoonright k_n^A| \leq \sum_{i
\leq n}  it,
\]
yielding
\[ 
t\left[ \frac{n(n+1)}{2}\right] - n - 1 \leq  |f(t)\upharpoonright
k_n^A| \leq t\left [\frac{n(n+1)}{2}\right].
\]
Since $k_n^A = \max L_n^A \geq \frac{n(n+1)}{2}$, if we divide by $k_n^A$
we obtain 
\[
\lim_{n \to \infty} \left( \frac{t|A \upharpoonright k_n^A|}{k_n^A} -
\frac{|f(t) \upharpoonright k_n^A|}{k_n^A} \right) = 0, 
\]
and the same holds with $t$ replaced by $s$.

To complete the proof of Equation (\ref{geoeq}), we need to show that
\begin{enumerate} 

\item $\limsup_n  \rho_{k_n^A}(f(t) \sd f(s))= \orho(f(t) \sd f(s))$ and

\item $\limsup_n  \rho_{k_n^A}(A) = \orho(A)$.

\end{enumerate}

Clearly, $\leq$ holds in both lines.  We prove $\geq$ for part (2),
and the proof for part (1) is essentially the same.  Let $x_n$ be such
that $k_n^A \leq x_n \leq k_{n+1}^A$. By definition,
\[
\rho_{x_n}(A) = \frac{|A \upharpoonright x_n|}{x_n} \leq \frac{|A
  \upharpoonright x_n| + (n+1)}{x_n}.
\]
Since $x_n \geq k_n^A \geq \frac{n(n+1)}{2}$ we have $ \rho_{x_n}(A)
\leq \rho_{k_n^A}(A) + \text{o}(n)$ and the equation follows.

We have established that there is a geodesic from $[\emptyset]$ to an
arbitrary point $[C]$. Given arbitrary $[A]$ and $[B]$ there is a
geodesic from $[\emptyset]$ to $[A + B]$.  Applying the isometry
$\tau_A$ from Observation \ref{isomobs} to this geodesic yields a
geodesic from $[\emptyset] + [A] = [A]$ to $[A + B] + [A] = [B]$.
\end{proof}

\begin{observation}
\label{hilbert}
It is interesting to note that $(\mathcal S,\delta)$ is a rather rich
space. Consider the sets $C_r$ in the proof of Theorem
\ref{conthm}. The argument above with $A=\omega$ shows that the
subspace $\{C_r : r \in [0,1]\}$ of $(\mathcal S,\delta)$ is
homeomorphic to $[0,1]$. Let $R_k = \{m : \ 2^k \mid m \,\mathbin{\&}\, 2^{(k +
1)} \nmid m\}$. Let $i_0^k<i_1^k<\cdots$ be the elements of
$R_k$. For sets $X_0,X_1,\ldots$, let $\bigoplus_k^{\mathcal R} X_k =
\{i^k_n : n \in X_k\}$. It is not difficult to adapt this argument to
show that the subspace $\{\bigoplus^{\mathcal R}_i C_{r_i} :
r_0,r_1,\ldots \in [0,1]\}$ of $(\mathcal S,\delta)$ is homeomorphic
to the Hilbert cube $[0,1]^\omega$. Thus, for instance, every second
countable normal space embeds into $(\mathcal S,\delta)$.
\end{observation}

\begin{theorem}
\label{contgeothm}
Let $[A]$ and $[B]$ be distinct points in $\mathcal{S}$. Then there
are continuum many geodesics from $[A]$ to $[B]$.
\end{theorem}

\begin{proof}
It is enough to show that
\[
M = \left\{[Y] : \delta([A], [Y]) = \delta([Y], [B]) = \frac{1}{2}
\delta([A], [B])\right\}
\]
is uncountable, because for each $[Y] \in M$ there is a geodesic path
from $[A]$ to $[Y]$ and a geodesic path from $[Y]$ to $B$, by Theorem
\ref{geodesic}. Joining these two paths together yields a geodesic
path from $[A]$ to $[B]$ with midpoint $[Y]$, and obviously paths with
distinct midpoints are distinct.

Since $[A] \neq [B]$, the symmetric difference $A \sd B$ is infinite,
so let $d_0, d_1, \dots$ list the elements of $A \sd B$ in increasing
order.  We now define a function $F : 2^\omega \to 2^\omega$, which
will have the property that $F(X) \in M$ whenever $\rho(X) = 1/2$.
The idea is that $F(X)$ copies the common value of $A$ and $B$ at
points where $A$ and $B$ agree, and at other points $d_i$, it copies
$A$ if $i \in X$ and otherwise copies $B$. Thus we define
\[
F(X)(n)=
\begin{cases}
A(n) & \text{if } A(n) = B(n) \text{ or } (n = d_i \,\mathbin{\&}\, i \in
X)\\ B(n) & \text{if } n = d_i \,\mathbin{\&}\, i \notin X.
\end{cases}
\]

\begin{lemma}
If $\rho(X) = \frac{1}{2}$, then $[F(X)] \in M$.
\end{lemma}

\begin{proof}
Recall that we denote the complement of $X$ by $\neg X$. Suppose that
$d_k < n \leq d_{k+1}$. By the definition of $F$,
\[
|(A \sd F(X)) \upharpoonright n|  =  |\neg X \upharpoonright
k|.
\]
Dividing both sides by $n$ yields
\[
\frac{|(A \sd F(X)) \upharpoonright n|}{n}  =  \frac{|\neg X
  \upharpoonright k|} {k}\cdot \frac{k}{n}.
\]
Then, by the definition of $\rho$, it follows that
\[
\rho_n (A \sd F(X)) = \rho_k (\neg X) \cdot \rho_n (A \sd B).
\]
Taking the lim sup of both sides as $n$ goes to infinity, and hence
also $k$ goes to infinity, we obtain
\[
\overline{\rho}(A \sd F(X)) = \frac{1}{2} \overline{\rho}(A \sd
B).
\]
Hence, by the definition of $\delta$,
\[\delta([A], [F(X)]) = \frac{1}{2} \delta([A],[B]).
\]
A similar argument shows that $\delta([B], [F(X)]) = \frac{1}{2} \delta([A]
, [B])$, completing the proof of the lemma.
\end{proof}

\begin{lemma}   
If $\underline{\rho}(X_0 \sd X_1) > 0$, then $[F(X_0] \neq [F(X_1)]$.
\end{lemma}

\begin{proof} 
One can argue as in the proof of the previous theorem that if $d_k < n
\leq d_{k+1}$ then
\[
\rho_n (F(X_0) \sd F(X_1)) = \rho_k (X_0 \sd X_1) \cdot \rho_n (A \sd
B).
\]
Choose a real $\epsilon > 0$ and a number $n_0$ such that $(\forall n
\geq n_0)[ \rho_n (A \sd B) \geq \epsilon]$.  Then for all $n \geq
n_0$ we have
\[
\rho_n(F(X) \sd F(X_1)) \geq \epsilon \cdot \rho_n(A \sd B).
\]
Taking the lim sup of both sides we obtain
\[
\delta([F(X_0)], [F(X_1)]) \geq \epsilon \delta([A], [B]) > 0.
\]
It follows that $[F(X_0)] \neq [F(X_1)]$.
\end{proof}

To complete the proof of the theorem, it suffices to show that there
is a family $\mathcal C \subseteq 2^\omega$ of size continuum such that
every element of $\mathcal{C}$ has density $1/2$, and the symmetric
difference of any two distinct elements of $\mathcal C$ has positive
lower density, since then by the lemmas $\{[F(X)] : X \in
\mathcal{C}\}$ is a subset of $M$ of size continuum.

Let $C_r$ be as in the proof of Theorem \ref{conthm}, let $X_r = C_r
\oplus \neg C_r$, and let $\mathcal C = \{X_r : r \in [0,1]\}$. Then
$\mathcal C$ has size continuum, and every element $X_r$ of $\mathcal
C$ has density $1/2$, since for each $n$, exactly one of $2n$ and
$2n+1$ is in $X_r$. Let $r \neq s$. As shown in the proof of Theorem
\ref{conthm}, $\rho(C_r)=r$ and $\rho(C_s)=s$, so
$\underline{\rho}(C_r \sd C_s)>0$. Furthermore, for each $n$, we have
$\rho_{2n}(X_r \sd X_s) = \rho_n (C_r \sd C_s)$, so
$\underline{\rho}(X_r \sd X_s)>0$.
\end{proof}

The properties of the family $\mathcal C$ in the proof of Theorem
\ref{contgeothm} suggest the question of whether we can build a family
of size continuum such that each element has density $1/2$, and the
symmetric difference of any two distinct elements also has density
$1/2$. Such a family would behave as we expect pairwise mutually
random sets to behave, and indeed it suffices to build a family of
size continuum of pairwise mutually $1$-random sets. (Of course,
$1$-randomness is considerably more than what is needed
here. Computable randomness would suffice, as would even the fairly
weak notion of Church stochasticity (see \cite[Definition
7.4.1]{DH}).)

The existence of such a family of $1$-randoms is a natural question,
but it appears not to have been directly addressed in the literature,
except in an unpublished manuscript of Miller and Yu
\cite{MilYu}. They gave a direct construction of such a family, but
noted [personal communication] that its existence also follows from a
general Ramsey-theoretic theorem due to Jan Mycielski \cite[Theorem
  1]{My}. (A simpler version of this theorem, which is sufficient for
this purpose, is cited as Theorem 10.3.15 in \cite{DH}.)  We are
grateful to Alekos Kechris for bringing \cite{My} to our attention.
We also thank Anush Tserunyan very much for formulating the result in
a form that is much easier for us to apply and showing us how it can
be proved using extensions of some exercises in Kechris's book
\cite{K}. (This occurred before we were even aware of the possibility
of applying \cite{My} here.) We will give Tserunyan's formulation and
sketch her proof in this section, and will discuss a different proof
in Section \ref{mycsec}.

For any set $A$, let $(A)^n$ denote the set of all ordered $n$-tuples
of distinct elements of $A$, and let $A^n$ denote the set of all
ordered $n$-tuples of elements of $A$.  We consider the usual Lebesgue
measure on $[0, 1]^n$, where $[0, 1]$ is the unit interval.

\begin{theorem}[Mycielski \cite{My}]
\label{Ramsey}  
For any sequence of relations $\{R_k\}_{k \in \omega}$ with $R_k
\subseteq [0, 1]^{n_k}$ such that $R_k$ is of measure $1$ for all $k$,
there is a homeomorphic copy $C \subseteq [0, 1]$ of the Cantor space
$2^\omega$ such that $(C)^{n_k} \subseteq R_k$ for all $k \in \omega$.
\end{theorem}

\begin{proof}[Proof sketch\ \emph{(Tserunyan {[private
communication]})}.]
We first turn this measure-the\-o\-re\-tic statement into a
topological one by considering the (Lebesgue) density topology on $[0,
1]$, as defined in \cite[Exercise (17.47.ii)]{K}.  Note that the
density topology on $[0, 1]$ can be similarly defined on $[0, 1]^n$ by
using open balls in $[0, 1]^n$ in place of open intervals in $(0, 1)$
in the definition of $\varphi(A)$ in \cite[Exercise (17.47.i)]{K}.
The remaining parts of \cite[Exercise (17.47)]{K} continue to hold for
$[0, 1]^n$ in place of $[0,1]$.  The subsets of $[0, 1]^n$ of Lebesgue
measure $0$ are exactly the meager sets in the density topology by
\cite[Exercise (17.47.iii)]{K}, extended to $[0, 1]^n$, and this
topological space is Choquet by \cite[Exercise (17.47.vi)]{K} and is
clearly perfect, since nonempty open sets have positive Lebesgue
measure.  Thus, the Ramsey Theorem for perfect Choquet spaces with a
weaker metric \cite[Exercise (19.5)]{K} applies for each fixed
$n_k$.  The latter result is easily extended to cover a sequence of
relations of varying arity instead of a single relation, so the result
follows.
\end{proof}

The above theorem also holds for $2^\omega$ (equipped with its usual
fair coin toss measure) in place of $[0, 1]$ because the binary
expansion map from $2^\omega$ to $[0, 1]$ is measure-preserving and
almost a homeomorphism.  More precisely, this map induces a
homeomorphism from a co-countable subset of $2^\omega$ (the set of
binary sequences with infinitely many $1$'s) to a co-countable subset
of $[0, 1]$ (the set of non-dyadic reals in $[0, 1]$). 

The relation that holds of $(X,Y)$ if and only if $X$ and $Y$ are
mutually $1$-random has measure $1$, so we have the following.

\begin{corollary}
\label{perfrandcor}
There is a nonempty perfect set whose elements are pairwise mutually
$1$-random.
\end{corollary}

Of course, there is nothing special about $1$-randomness
here. Mycielski's Theorem applies just as well to other notions of
algorithmic randomness.

The following corollary was first proved by a direct construction,
which we will return to in Section \ref{mycsec}.

\begin{corollary}
\label{perfhalfcor}
There is a perfect $\mathcal{C} \subseteq 2^\omega$ such that every
element of $\mathcal{C}$ has density $1/2$ and the symmetric
difference of any two distinct elements of $\mathcal{C}$ also has
density $1/2$.
\end{corollary}

We can strengthen Corollary \ref{perfrandcor} by using the fact that
each of the relations $\{(X_0,\ldots,X_n) : \textrm{each $X_i$ is
  $1$-random relative to } \bigoplus_{j \leq n,\ j \neq i} X_j\}$ has
measure $1$.

\begin{corollary}
\label{perfmutcor}
There is a nonempty perfect set $\mathcal P$ such that if $\mathcal F
\subset \mathcal P$ is a nonempty finite set, then the join of the
elements of $\mathcal F$ is $1$-random.
\end{corollary}

It is not possible to extend this corollary to countable sets
$\mathcal F$, because if $\mathcal P$ is perfect and $X \in \mathcal
P$, then there are initial segments $\sigma_0 \prec \sigma_1 \prec
\cdots$ of $X$ such that for each $i$ there is an $X_i \in \mathcal P$
for which $\sigma_i$ is the longest common initial segment of $X_i$
and $X$, and then $X \leq\sub{T} \bigoplus_i X_i$.

It is not possible for $\mathcal P$ to have positive measure,
even in Corollary \ref{perfrandcor}, because such a $\mathcal P$ would
be an antichain with respect to Turing reducibility, and Yu \cite{Yu}
has shown that such an antichain cannot have positive measure.

In Section \ref{mycsec}, we will give a proof of Corollary
\ref{perfmutcor} as an effectivization of a proof of Mycielski's
Theorem along the lines of the one outlined in \cite{Tod}, and will
also discuss the computability-theoretic and reverse-mathematical
content of the above corollaries. It is worth noting that the
relativized form of Corollary \ref{perfmutcor} is in fact equivalent
to Mycielski's Theorem, in the sense that the latter can be proved
easily from it. We will give the argument in Section \ref{mycsec}, a
version of which was first given by Miller and Yu \cite{MilYu}.

Mycielski also proved an analog of Theorem \ref{Ramsey} for category,
which implies that there is a nonempty perfect set $\mathcal P
\subseteq 2^\omega$ such that if $\mathcal F \subset \mathcal P$ is a
nonempty finite set, then the join of the elements of $\mathcal F$ is
$1$-generic, and hence, in particular, any two distinct elements of
$\mathcal P$ are mutually $1$-generic.

\begin{theorem}[Mycielski \cite{My}]
\label{Ramseyforcat}
For any sequence of relations $\{R_k\}_{k \in \omega}$ with $R_k
\subseteq [0, 1]^{n_k}$ such that $R_k$ is comeager for all $k$, there
is a homeomorphic copy $C \subseteq [0, 1]$ of the Cantor space
$2^\omega$ such that $(C)^{n_k} \subseteq R_k$ for all $k \in \omega$.
\end{theorem}

As with Theorem \ref{Ramsey}, this result also holds for $2^\omega$,
for the same reason. We will prove it (in the $2^\omega$ version) as
Theorem \ref{myccatthm} below, and discuss its computability-theoretic
and reverse-mathematical content in Section \ref{mycsec3}.

\section{Turing invariant subspaces of $(\mathcal S,\delta)$}
\label{invsec}

In this section, we begin to examine some interactions between
$(\mathcal S,\delta)$ and computability theory by considering subspaces
of this metric space arising from sets of Turing degrees. In
particular, we discuss conditions under which such subspaces can have
some of the properties of $(\mathcal S,\delta)$ discussed in the
previous section. Say that a subset of $\mathcal P(\omega)$ is
\emph{Turing invariant} if it closed under Turing equivalence (i.e.,
it is a union of Turing degrees). A particular example of a Turing
invariant set is a \emph{Turing ideal}, i.e., a subset of $\mathcal
P(\omega)$ that is closed downward under Turing reducibility and
closed under finite joins. Say that $U \subseteq \mathcal
P(\omega)$ \emph{generates} $\mathcal U \subseteq \mathcal S$ if
$\mathcal U=\{[A] : A \in U\}$. We will be interested in subspaces of
$(\mathcal S,\delta)$ that are generated by Turing invariant sets.

We first note that it does not matter whether we consider sets closed
under Turing equivalence or sets that are downward closed under Turing
reducibility.

\begin{lemma}
\label{downclosed}
If $U$ is Turing invariant then
\[ \{[A] : A \in U\} = \{[B] : (\exists A \in U) [B \leq\sub{T} A]\}.
\]
\end{lemma}

\begin{proof}
If $A \in U$ and $B \leq\sub{T} A$, then we can take a coinfinite
computable set $C$ of density $1$, list the elements of its complement
as $d_0<d_1<\cdots$, and define $D = (B \cap C) \cup \{d_n : n \in
A\}$. Then $D \equiv\sub{T} A$, so $D \in U$, and $[D]=[B]$.
\end{proof}

We now note a couple of other interesting facts about subsets of
$\mathcal S$ generated by Turing invariant sets. Recall that $J_k =
[2^k - 1, 2^{k+1} - 1)$.

\begin{definition}
\label{jdef}
Let $\mathcal J(A) = \bigcup_{k \in A} J_k$.
\end{definition}

\begin{lemma}
\label{jlem}
If $\delta(\mathcal J(A),C)<1/4$ then $A \leq\sub{T} C$.
\end{lemma}

\begin{proof}
If $\delta(\mathcal J(A),C)<1/4$ then, by Lemma \ref{factor2}, for all
but finitely many $n$, more than half of the bits in $C
\upharpoonright J_n$ are equal to $A(n)$.
\end{proof}

Recall that $\mathcal S$ has the structure of a topological
group under the operation defined by $[A] + [B] = [A \sd B]$.

\begin{proposition}
Let $U$ be downward closed under Turing reducibility. Then $\{[A] : A
\in U\}$ is a subgroup of $\mathcal S$ if and only if $U$ is a Turing
ideal.
\end{proposition}

\begin{proof}
Suppose that $U$ is a Turing ideal. If $B,C \in U$ then $B \sd C
\leq\sub{T} B \oplus C$ is also in $U$. Since each element of
$\mathcal S$ is its own inverse, it follows that $\{[A] : A \in U\}$
is a subgroup of $\mathcal S$.

Now suppose that $\mathcal U=\{[A] : A \in U\}$ is a subgroup of
$\mathcal S$ and let $B,C \in U$. Then $[\mathcal J(B) \oplus
\emptyset],[\emptyset \oplus \mathcal J(C)] \in \mathcal U$, so
$[\mathcal J(B) \oplus \mathcal J(C)] = [(\mathcal J(B) \oplus
\emptyset) \sd (\emptyset \oplus \mathcal J(C))] \in \mathcal U$. Thus
$U$ contains a coarse description of $\mathcal J(B) \oplus \mathcal
J(C)$. This description computes coarse descriptions of $\mathcal
J(B)$ and $\mathcal J(C)$, and hence computes both $B$ and $C$. It
follows that $B \oplus C \in U$.
\end{proof}

We say that a collection $U$ of sets is \emph{cofinal in the Turing
degrees} if for every $B$, there is an $A \in U$ such that $B
\leq\sub{T} A$. The following fact follows from Lemmas
\ref{downclosed} and \ref{jlem}.

\begin{proposition}
\label{cofcor}
Let $U$ be Turing invariant. Then $U$ is cofinal in the Turing degrees
if and only if $\{[A] : A \in U\}=\mathcal S$.
\end{proposition}

The space $(\mathcal S,\delta)$ has many compact subspaces. (For
example, the subspace consisting of all the $[C_r]$ defined in the
proof of Theorem \ref{conthm} is homeomorphic to the interval
$[0,1]$.) However, none of them can be generated by a Turing invariant
set.

\begin{theorem}
If $\mathcal U \subseteq \mathcal S$ is nonempty and generated by a
Turing invariant set, then $\mathcal U$ is not compact.
\end{theorem}

\begin{proof}
Let $U$ be a Turing invariant set that generates $\mathcal U$. The
construction in part (1) of the proof of Theorem \ref{compthm} shows
that for sets $A_0,\ldots,A_n$, there is a set $A_{n+1} \leq\sub{T}
A_0 \oplus \cdots \oplus A_n$ such that $\delta([A_{n+1}],[A_i])=1$
for all $i \leq n$. So we can take any $A_0 \in U$ and build an
infinite sequence $A_0,A_1,\ldots \leq\sub{T} A_0$ such that
$\delta([A_i],[A_j]))=1$ for all $i \neq j$. Each $[A_i]$ is in
$\mathcal U$.

Now take any open cover of $\mathcal U$ by open balls, all of whose
radii are less than $1/2$. None of these balls can contain both
$[A_i]$ and $[A_j]$ for $i \neq j$, so this cover has no finite
subcover.
\end{proof}

We now turn to completeness. Since $(\mathcal S,\delta)$ is itself
complete, a subspace is complete if and only if it is closed. The
space $(\mathcal S,\delta)$ has many kinds of closed subspaces (for
example, the closure $\cd$ of a degree $\bd$ in Definition
\ref{closdef} below), but if a nonempty $\mathcal U \subsetneq
\mathcal S$ is generated by a Turing invariant set, then it is
difficult for $\mathcal U$ to be closed, as we now show.

The following sets defined in \cite{JS1} are very useful in
studying asymptotic density and computability.

\begin{definition} 
\label{rdef}
Let $R_k = \{m :  \ 2^k \mid m \,\mathbin{\&}\, 2^{(k + 1)} \nmid m\}$, and let
$\mathcal R(A) = \bigcup_{k \in A} R_k$.
\end{definition}

\begin{lemma}[{Asher Kach [personal communication]}]
\label{Rlem}
The set $\mathcal{R}(A)$ is a limit of computable sets for every $A
\subseteq \omega$.
\end{lemma}

\begin{proof}
Let $A_k=A \cap [0,k]$.  Each $\mathcal{R}(A_k)$ is computable, and
$\mathcal{R}(A) \sd \mathcal{R}(A_k) \subseteq \bigcup_{j \geq k}
R_j$. Since the density of the latter set goes to $0$ as $k$ goes to
infinity, so does $\delta(\mathcal R(A),\mathcal R(A_k))$.
\end{proof}

The following lemma is a relativized version of one direction of
Theorem 2.19 of \cite{JS1}.

\begin{lemma}
\label{Rlem2}
If $C$ is a coarse description of $\mathcal{R}(A)$ then $A \leq\sub{T}
C'$.
\end{lemma}

\begin{proof}
For each $k$, the density of $C$ within $R_k$ must be $1$ if $k \in
A$, and $0$ if $k \notin A$. Let
\[
\rho^k_n(C) = \frac{|(C \cap R_k) \upharpoonright n|}{|R_k
\upharpoonright n|}.
\] 
Given $k$, we can use $C'$ to search for an $m$ such that either
$\rho^k_n(C) > \frac{1}{2}$ for all $n \geq m$, or $\rho^k_n(C) <
\frac{1}{2}$ for all $n \geq m$. In the first case $k \in A$, while in
the latter $k \notin A$.
\end{proof}

Recall the notation $\bigoplus_k^{\mathcal R} X_k$ from Observation
\ref{hilbert}. If $C$ is a coarse description of
$\bigoplus_k^{\mathcal R} X_k$ then $\{n : i^k_n \in C\}$ is a coarse
description of $X_k$ for each $k$, since $R_k$ has positive density,
so we have the following fact.

\begin{lemma}
\label{rpluslem}
Every coarse description of $\bigoplus_k^{\mathcal R} X_k$ computes a
coarse description of $X_k$ for each $k$.
\end{lemma}

We can now give necessary conditions for a subspace of $(\mathcal
S,\delta)$ generated by a Turing invariant set to be closed, and hence
complete.

\begin{theorem}
\label{clthm}
If $\mathcal U \subseteq \mathcal S$ is generated by a Turing
invariant set $U$ and is closed, then the following hold.
\begin{enumerate}

\item For every $X$ and every $A \in U$, there is a $C \in U$ such
that $A \leq\sub{T} C$ and $X \leq\sub{T} C'$.

\item Every countable Turing ideal contained in $U$ has an upper bound
in $U$.

\end{enumerate}
\end{theorem}

\begin{proof}
(1) Fix $X$ and $A \in U$. We have $[\mathcal J(A) \oplus B] \in
\mathcal U$ for all computable $B$, so an easy adaptation of the proof
of Lemma \ref{Rlem} shows that $[\mathcal J(A) \oplus \mathcal R(X)]
\in \mathcal U$, i.e., there is a $C \in U$ that is a coarse description
of $\mathcal J(A) \oplus \mathcal R(X)$. This $C$ computes coarse
descriptions of both $\mathcal J(A)$ and $\mathcal R(X)$, so by Lemmas
\ref{jlem} and \ref{Rlem2}, $A \leq\sub{T} C$ and $X \leq\sub{T} C'$.

(2) Suppose that $\mathcal I$ is a countable Turing ideal contained in
$U$. Let $A_0,A_1,\ldots$ be the elements of $\mathcal I$. Let
$i_0^k<i_1^k<\cdots$ be the elements of $R_k$. Let $Z_n =
\bigoplus_k^{\mathcal R} X_k$ where $X_k=\mathcal J(A_k)$ for $k<n$
and $X_k=\emptyset$ for $k \geq n$, and let $Z = \bigoplus_k^{\mathcal
R} \mathcal J(A_k)$. Then each $Z_n$ is computable from finitely
many elements of $\mathcal I$, and hence is in $\mathcal I$. It is
also clear that $[Z]$ is the limit of the $[Z_n]$ in $(\mathcal
S,\delta)$, so $[Z] \in \mathcal U$. That is, there is a $Y \in U$
that is a coarse description of $Z$. By Lemma \ref{rpluslem}, $Y$
computes a coarse description of $\mathcal J(A_k)$ for each $k$, and
hence $Y$ computes each $A_k$, i.e., $Y$ is an upper bound for
$\mathcal I$.
\end{proof}

\begin{corollary}
\label{clocontcor}
If  $\mathcal U \subseteq \mathcal S$ is nonempty, generated by a
Turing invariant set, and closed, then it has size continuum.
\end{corollary}

\begin{proof}
Let $U$ be a Turing invariant set generating $\mathcal U$. By the
theorem, the set of jumps of elements of $U$ is cofinal in the Turing
degrees, and hence has size continuum, so $U$ itself has size
continuum. Thus the set of degrees of elements of $U$ has size
continuum. If $A$ and $B$ have different degrees, then $[\mathcal
J(A)] \neq [\mathcal J(B)]$, by Lemma \ref{jlem}, so $\mathcal U$
 has size continuum.
\end{proof}

It is of course natural to ask whether the conditions in the above
theorem can hold for any nonempty $\mathcal U \subsetneq \mathcal S$
generated by a Turing invariant set. We will discuss this question at
the end of this section.  We do not know whether these necessary
conditions are also sufficient, but we will show that if we strengthen
condition (2) from countable Turing ideals to all countable subsets of
$U$, then we do obtain sufficient conditions for a subspace of
$(\mathcal S,\delta)$ generated by a Turing invariant set to be
closed. Of course, in that case, the downward closure of $U$ under
Turing reducibility is a Turing ideal, so we might as well state our
result for Turing ideals, which also allows us to simplify condition
(1). We say that a collection $U$ of sets is \emph{jump-cofinal in the
Turing degrees} if for every $X$, there is a $C \in U$ such that $X
\leq\sub{T} C'$. It is easy to see that if $U$ is a nonempty Turing
ideal, then $U$ satisfies condition (1) if and only if it is
jump-cofinal in the Turing degrees.

We will use the relativized form of the following result, which
follows from a theorem due to J. Miller (see \cite{HJMS}).

\begin{lemma}[Hirschfeldt, Jockusch, McNicholl, and Schupp
{\cite[Corollary 5.11]{HJMS}}]
\label{hjmslem}
Suppose there is a $\emptyset'$-computable function $f$ such that, for
all $e$, we have that $\Phi_{f(e)}$ is total and $\{0,1\}$-valued, and
$\overline{\rho}(B \bigtriangleup \Phi_{f(e)}) \leq 2^{-e}$.  Then $B$
is coarsely computable.
\end{lemma}

\begin{theorem}
\label{clidthm}
Let $\mathcal U \subseteq \mathcal S$ be nonempty and generated by a
Turing ideal $U$. Then $\mathcal U$ is closed if and only if $U$ is
jump-cofinal in the Turing degrees and every countable subset of $U$
has an upper bound in $U$.
\end{theorem}

\begin{proof}
The ``only if'' direction follows from Theorem \ref{clthm} and the
fact that every countable subset of $U$ is contained in a countable
subideal of $U$. For the ``if'' direction, let $B_0,B_1,\ldots \in U$
be such that $[B_0],[B_1],\ldots$ is a Cauchy sequence, and let $[B]$
be its limit. By passing to a subsequence, we can assume that
$\delta([B_m],[B_n]) < 2^{-m}$ for all $m<n$. Let $D \in U$ be such
that $B_0,B_1,\ldots \leq\sub{T} D$ (which exists because the subideal
of $U$ generated by the $B_i$ has an upper bound in $U$). Let $e_i$ be
such that $\Phi_{e_i}^D=B_i$ for all $i$. There is a $C \in U$ such
that $D \leq\sub{T} C$ and the function $i \mapsto e_i$ is computable
in $C'$. Then there is a function $f \leq\sub{T} C'$ such that
$\Phi_{f(i)}^C=B_i$ for all $i$. By the relativized form of Lemma
\ref{hjmslem}, $B$ is coarsely $C$-computable, so $[B] \in \mathcal
U$.
\end{proof}

Again, we will discuss the question of whether there are any $\mathcal
U \subsetneq \mathcal S$ that satisfy the conditions in this theorem
at the end of this section.

We now turn to the properties in Theorems \ref{conthm} and
\ref{geodesic}, for which we have the following exact
characterization, which again involves condition (1) above. A real is
\emph{right-c.e.}\ if the set of rationals greater than it is c.e.,
and \emph{left-c.e.}\ if the set of rationals less than it is c.e. If
a real is both right-c.e.\ and left-c.e., then it is computable.

\begin{lemma}
\label{rightlem}
For any sets $A$ and $B$, we have that $\delta(A,B)$ is
right-c.e.\ relative to $(A \oplus B)'$.
\end{lemma}

\begin{proof}
We can assume $\delta(A,B)$ is irrational, since every rational is a
computable real. Then for a rational $q$, we have $q > \delta(A,B)$ if
and only if $(\exists m)(\forall n>m)[\rho_n(A \sd B) < q]$, which is
a c.e.\ condition relative to $(A \oplus B)'$.
\end{proof}

\begin{theorem}
\label{conthm2}
Let $\mathcal U \subseteq \mathcal S$ be nonempty and generated by a
Turing invariant set $U$. The following are equivalent.
\begin{enumerate}

\item $\mathcal U$ is connected.

\item $\mathcal U$ is pathwise connected.

\item $\mathcal U$ is contractible.

\item $\mathcal U$ is geodesic.

\item For every real $r$ and every $A \in U$, there is a $B \in U$
such that $A \leq\sub{T} B$ and $r$ is right-c.e.\ relative to $B'$.

\item For every real $r$ and every $A \in U$, there is a $C \in U$
such that $A \leq\sub{T} C$ and $r \leq\sub{T} C'$.

\end{enumerate}
\end{theorem}

\begin{proof}
By Lemma \ref{downclosed}, we can assume that $U$ is downward
closed under Turing reducibility.

It is well-known that (3) and (4) both imply (2), which in turn
implies (1).

Suppose that (5) holds and fix a real $r$ and a set $A \in U$. Let $B
\in U$ be such that $A \leq\sub{T} B$ and $r$ is right-c.e.\ relative
to $B'$. Now let $C \in U$ be such that $B \leq\sub{T} C$ and $1-r$ is
right-c.e.\ relative to $C'$, which implies that $r$ is left-c.e.\
relative to $C'$. Since $B' \leq\sub{T} C'$, we have that $r$ is both
right-c.e.\ and left-c.e.\ relative to $C'$, so $r
\leq\sub{T} C'$. Thus (5) implies (6). Clearly (6) implies (5), so (5)
and (6) are equivalent.

Now suppose that (5) fails. We show that $\mathcal U$ is not
connected. Let $r$ and $A \in U$ be such that if $B \in U$ and $A
\leq\sub{T} B$, then $r$ is not right-c.e.\ relative to $B'$. If $A$
is computable, then Lemma \ref{rightlem} implies that there is no $B
\in U$ with $\delta(A,B)=r$, which allows us to separate $\mathcal U$
into the disjoint nonempty open sets $\{[Y] \in \mathcal U :
\delta(Y,A)<r\}$ and $\{[Y] \in \mathcal U : \delta(Y,A) > r\}$. Since
we cannot assume that $A$ is computable, we replace $A$ by $\mathcal
J(A)$ (which is in $\mathcal U$), and replace $r$ by a real $s$ that
differs from $r$ by a rational and is small enough to ensure that
every $B$ such that $\delta(B,\mathcal J(A))<s$ computes $A$.

In detail: Let $s<1/4$ differ from $r$ by a rational. Note that $s$
cannot be right-c.e.\ relative to $B'$ for any $B \geq\sub{T} A$ in
$U$. Let $\mathcal D$ consist of all $[Y] \in \mathcal U$ such that
$\delta(Y,\mathcal J(A)) < s$, and let $\mathcal E$ consist of all
$[Y] \in \mathcal U$ such that $\delta(Y,\mathcal J(A)) > s$. Then
$\mathcal D$ and $\mathcal E$ are open in $\mathcal U$, disjoint, and
nonempty. (We have $[\mathcal J(A)] \in \mathcal D$, and $U$ contains
sets of density $0$ and $1$, at least one of which is in $\mathcal
E$.)

If $B \in U$ and $\delta(B,\mathcal J(A)) = s$, then
$\delta(B,\mathcal J(A)) < 1/4$, so that by Lemma \ref{jlem}, $A
\leq\sub{T} B$, and hence $s$ is not right-c.e.\ relative to $B'$. But
$ B \oplus \mathcal J(A) \equiv\sub{T} B \oplus A \equiv\sub{T} B$, so
this fact contradicts Lemma \ref{rightlem}. Thus there is no such $B$,
and hence $\mathcal U = \mathcal D \cup \mathcal E$ is not connected.

Finally, suppose that (6) holds. We first show that $\mathcal U$ is
contractible. It is enough to show that for $A \in U$ and each of the
sets $A_r$ in the proof of Theorem 2.7, $[A_r] \in \mathcal U$. We use
the notation of that proof. The basic idea is that $C_r$ has density
very close to $r$ within each $L_i$, but we can instead take a
sequence $q_0,q_1,\ldots$ of rationals with limit $r$, and build a set
that is defined like $C_r$ but has density very close to $q_i$ within
each $L_i$. We can then show that this set is coarsely equivalent to
$C_r$, and thus its intersection with $A$ is coarsely equivalent to
$A_r$.

In detail: Given $r$, let $C \in U$ be such that $A \leq\sub{T} C$ and
there is a $C$-computable sequence of rationals $q_0,q_1,\ldots$ with
limit $r$. Define $D_r$ by letting $D_r \cap L_i = [m_i,m_i+\lfloor
q_ii \rfloor]$. Let $B_r = A \cap D_r$. Then $B_r \leq\sub{T} A
\oplus C \equiv\sub{T} C$, so $B_r \in U$. Furthermore,
$\delta(A_r,B_r) = \overline{\rho}(A_r \sd B_r) \leq
\overline{\rho}(C_r \sd D_r) = \delta(C_r,D_r)$. Fix $\epsilon>0$. Let
$s<t$ be such that $r \in (s,t)$ and $t-s<\epsilon$. Then $C_r \sd D_r
\subseteq^* C_s \sd C_t$, so $\delta(C_r,D_r) \leq
\delta(C_s,C_t)=t-s<\epsilon$. (Here $\subseteq^*$ is containment up to
finitely many elements.) Since $\epsilon$ is arbitrary,
$\delta(A_r,B_r) \leq \delta(C_r,D_r)=0$, that is, $[A_r]=[B_r] \in
\mathcal U$.

Essentially the same argument, working with $L_i^A$ instead of $L_i$,
shows that $\mathcal U$ is geodesic.
\end{proof}

By Proposition \ref{cofcor}, the following examples of sets satisfying
the conditions of Theorem \ref{conthm2} are nontrivial.

\begin{corollary}
\label{friedcor}
Let $X$ be noncomputable. Then $\{[A] : X \nleq\sub{T} A\}$, thought
of as a subspace of $(\mathcal S,\delta)$, is pathwise connected (and
hence connected), and is in fact contractible and geodesic.
\end{corollary}

\begin{proof}
By the theorem, it is enough to show that if $X \nleq\sub{T} A$ then
for any real $r$ there is a $C \geq\sub{T} A$ such that $X
\nleq\sub{T} C$ and $r \leq\sub{T} C'$. The Friedberg Jump Inversion
Theorem can be combined with cone avoidance, as noted in
\cite[Exercise 4.18]{Ler}. Relativizing this fact to $A$ produces the
desired $C$.
\end{proof}

Combining Theorems \ref{clthm} and \ref{conthm2} yields the following
consequence.

\begin{corollary}
Let $\mathcal U \subseteq \mathcal S$ be nonempty, generated by a
Turing invariant set, and closed. Then it is pathwise connected (and
hence connected), and is in fact contractible and geodesic.
\end{corollary}

The following corollary has the same proof as Corollary
\ref{clocontcor}.

\begin{corollary}
If  $\mathcal U \subseteq \mathcal S$ is nonempty, generated by a
Turing invariant set, and connected, then it has size continuum.
\end{corollary}

As an example of the negative application of the above theorems, we
can take $U$ to be the set of hyperimmune-free degrees. This is an
uncountable set, downward closed under Turing reducibility, and every
degree above $\mathbf{0''}$ is the double jump of a hyperimmune-free
degree, so there might seem to be some hope that $\mathcal U = \{[A] :
A \in U\}$ is connected. However no hyperimmune-free degree can be
high, so by Theorems \ref{clthm} and \ref{conthm2}, $\mathcal U$ is
neither connected nor closed.

If $U$ is a Turing ideal, then conditions 5 and 6 in Theorem
\ref{conthm2} can be replaced by the following simpler equivalent
ones:
\begin{itemize}

\item[(5$'$)] For every real $r$ there is a $B \in U$ such that $r$ is
right-c.e.\ relative to $B'$.

\item[(6$'$)] For every real $r$ there is a $C \in U$ such that
$r \leq\sub{T} C'$, i.e., $U$ is jump-cofinal in the Turing degrees.

\end{itemize}
We now show that it is possible for a Turing ideal to satisfy these
conditions without being cofinal in the Turing degrees, using the
following definition.

\begin{definition}
\label{ptdef}
A \emph{perfect tree} is a map $T : 2^{<\omega} \rightarrow
2^{<\omega}$ such that for each $\sigma$, the strings $T(\sigma^\frown
0)$ and $T(\sigma^\frown 1)$ are incompatible and both properly extend
$T(\sigma)$. Let $T(A) = \bigcup_{\sigma \prec A} T(\sigma)$. We say
that $X$ is a \emph{path through $T$} if there is 
an $A$ such that $X = T(A)$.
\end{definition}

\begin{proposition}
\label{treeexthm}
There exists a $\emptyset'$-computable perfect tree $T$ such that if
$\mathcal F$ is a nonempty finite collection of paths through $T$,
then $\emptyset' \nleq\sub{T} \bigoplus_{Y \in \mathcal F} Y$.
\end{proposition}

\begin{proof}
This proposition can be proved directly, and Yu Liang [personal
communication] has noted that it also follows easily from the
result of Binns and Simpson \cite{BS} that there is a nonempty
$\Pi^0_1$ class $\mathcal C$ such that if $\mathcal F$ is a nonempty
finite collection of elements of $\mathcal C$ and $X$ is an element of
$\mathcal C \setminus \mathcal F$, then $X \nleq\sub{T} \bigoplus_{Y
\in \mathcal F} Y$.

The proposition also follows from results in Section \ref{mycsec}. In
Theorem \ref{perfeff}, we will show that there is a
$\emptyset'$-computable perfect tree $T$ such for any nonempty finite
collection $\mathcal F$ of paths through $T$, the set $\bigoplus_{Y
\in \mathcal F} Y$ is $1$-random. We will do the same for
$1$-genericity in place of $1$-randomness in the proof of Theorem
\ref{myccatthm}, as noted in Theorem \ref{ancthm}. Either of these
trees can be used to prove the theorem. Let us do it using
$1$-randomness.

Let $T$ be as above, and suppose there is a nonempty finite collection
$\mathcal F$ of paths through $T$ such that $\emptyset' \leq\sub{T}
\bigoplus_{Y \in \mathcal F} Y$. If $Z$ is computable then $T(Z)$ is
$\emptyset'$-computable, so there is a $\emptyset'$-computable path
$X$ through $T$ such that $X \notin \mathcal F$. Then, by van
Lambalgen's Theorem, $X$ is $1$-random relative to $\bigoplus_{Y \in
\mathcal F} Y$, and hence relative to $\emptyset'$, contradicting
the fact that $X \leq\sub{T} \emptyset'$.
\end{proof}

\begin{proposition}
Let $T$ be as in Proposition \ref{treeexthm} and let $U$ be the
downward closure under Turing reducibility of the class of all
$\bigoplus_{Y \in \mathcal F} Y$ such that $\mathcal F$ is a nonempty
finite collection of paths through $T$. Then $U$ is a Turing ideal and
is not cofinal in the Turing degrees. Furthermore, $U$ is jump-cofinal
in the Turing degrees, and hence $\{[A] : A \in U\}$, thought of as a
subspace of $(\mathcal S,\delta)$, is pathwise connected (and hence
connected), and is in fact contractible and geodesic.
\end{proposition}

\begin{proof}
That $U$ is a Turing ideal and is not cofinal in the Turing degrees is
clear from its definition. To see that $U$ is jump-cofinal in the
Turing degrees, fix a set $Z$. We can compute $Z$ from $T \oplus T(Z)
\leq\sub{T} \emptyset' \oplus T(Z) \leq\sub{T} T(Z)'$, and $T(Z) \in
U$.
\end{proof}

In general, we cannot expect an exact computability-theoretic
criterion for connectedness for arbitrary subspaces of $(\mathcal
S,\delta)$. For example, if we let $\mathcal U$ consist of $[C_r]$ for
all the sets $C_r$ in the proof of Theorem 2.7, then $\mathcal U$ is
pathwise connected, but if we add a single point $[A]$ to $\mathcal U$ for
a set $A$ that does not have density, then the resulting set is no
longer connected. Adding $[A]$ has no effect as far as Turing degrees
go, however, since the $C_r$ already have all possible
degrees. Nevertheless, one direction of Theorem \ref{conthm2} can be
adapted as follows.

\begin{theorem}
Let the subspace $\mathcal U \subseteq \mathcal S$ generated by $U$ be
connected and have at least two points. Then for every real $r$ and
every $A \in U$, there is a $C \in U$ such that $r \leq\sub{T} (A
\oplus C)'$.
\end{theorem}

\begin{proof}
Assume that $\mathcal U$ has at least two points but the condition in
the second sentence of the theorem fails.  The same argument as in the
proof of Theorem \ref{conthm2} shows that there exist a real $r$ and
an $A \in U$ such that for all $B \in U$, we have that $r$ is not
right-c.e.\ relative to $(A \oplus B)'$. Now repeat the argument that
not (5) implies not (1) in the proof of Theorem \ref{conthm2}, but
with $A$ itself in place of $\mathcal J(A)$. The only other
significant difference is that we now choose $s$ small enough so that
there is a point in $\mathcal U$ at distance greater than $s$ from
$[A]$, to ensure that $\mathcal E$ is nonempty.
\end{proof}

Let us return to the question of the existence of sets $\mathcal U
\subsetneq \mathcal S$ satisfying the conditions in Theorem
\ref{clthm}, or the stronger ones in Theorem \ref{clidthm},
i.e.\ nonempty closed sets $\mathcal U \subsetneq \mathcal S$
generated by Turing invariant sets (or ideals). We do not know
whether such sets exist in an absolute sense, but they do exist if we
assume the Continuum Hypothesis (CH).

\begin{theorem}[Richard Shore {[personal communication]}]
Assuming CH, there is a Turing ideal $U \subsetneq 2^\omega$ that is
jump-cofinal and such that every countable subset of $U$ has an upper
bound in $U$.
\end{theorem}

\begin{proof}
Assuming CH, let $(X_\alpha)_{\alpha < \omega_1}$ list
$2^\omega$. Let $D$ be noncomputable. We define a sequence
$(C_\alpha)_{\alpha < \omega_1}$ with $C_\alpha \leq\sub{T} C_\beta$
for $\alpha < \beta$, such that $X_\alpha \leq\sub{T} C_{\alpha+1}'$
for all $\alpha < \omega_1$ but $D \nleq\sub{T} C_\alpha$ for all
$\alpha < \omega_1$. Then $\{Y : (\exists \alpha<\omega_1)[Y
\leq\sub{T} C_\alpha]\}$ is our desired ideal.

Let $C_0=\emptyset$. Given $C_\alpha$ such that $D \nleq\sub{T}
C_\alpha$, as in the proof of Corollary \ref{friedcor}, relativizing
the combination of the Friedberg Jump Inversion Theorem with
cone-avoidance, we see that there is a $Z \geq\sub{T} C_\alpha$ such
that $X_\alpha \leq\sub{T} Z'$ and $D \nleq\sub{T} Z$. Let
$C_{\alpha+1}=Z$. For a limit ordinal $\gamma < \omega_1$, given
$C_\alpha \ngeq\sub{T} D$ for all $\alpha<\gamma$, let $Z_0,Z_1$ be an
exact pair for the ideal $\mathcal I = \{Y : (\exists \alpha<\gamma)[Y
\leq\sub{T} C_\alpha]\}$, as constructed by Spector
\cite{Spec}. That is, $Y \in \mathcal I$ if and only if $Y \leq\sub{T}
X_0$ and $Y \leq\sub{T} X_1$. Since $D \notin \mathcal I$, there is an
$i<2$ such that $D \nleq\sub{T} X_i$. Let $C_\gamma = X_i$.
\end{proof}

\section{Computability theory  and Hausdorff distance}

As mentioned in the introduction, a \emph{coarse description} of $A$
is a set $C$ such that $\delta(A,C)=0$, and a set is \emph{coarsely
computable} if it has a computable coarse description. Even if $A$ is
not coarsely computable, one can measure how closely $A$ can be
approximated by computable sets.  Let $r$ be a real number such that
$0 \leq r \leq 1$. A set $A$ is \emph{coarsely computable at density
$r$} if there is a computable set $C$ such that the symmetric
agreement between $A$ and $C$ has lower density at least $r$, that is,
$\urho(A \sa C) \geq r$. A set $C$ such that $\urho(A \sa C) \geq r$
is called an \emph{$r$-description} of $A$. Then, as in \cite{HJMS},
define the \emph{coarse computability bound} $\gamma(A)$ of $A$ by
\[
\gamma(A) = \sup\{r : A \text{ is coarsely computable at density }
r\}.
\]
We can relativize the definition of coarsely computable sets in
\cite{JS1} and the coarse computability bound $\gamma$ in \cite{HJMS}
to any Turing degree $\bd$.

\begin{definition}
The set $A$ is \emph{coarsely $\bd$-computable} if it has a coarse
description computable from $\bd$.

The set $A$ is  \emph{coarsely $\bd$-computable at density $r$} if
there is an $r$-description $B$ of $A$ such that $B \leq\sub{T} \bd$.

The \emph{coarse $\bd$-computability bound} of a set $A$ is
\[ 
\gamma_{\bd}(A) = \sup\{r : {A \text{ is coarsely }
\mathbf{d}\text{-computable at density }  r}\}.
\] 
If $D$ is a set whose degree is $\bd$, we also write $\gamma_D(A)$ for
$\gamma_{\bd}(A)$.
\end{definition}

Note that these definitions depend only on similarity classes, so we
can consider them to be defined on such classes, and in particular let
$\gamma_{\bd}([A])=\gamma_{\bd}(A)$ for any $[A] \in \mathcal
S$. Observe that $\gamma_{\bd}([A]) = 1$ if and only if $[A]$ is a
limit of coarse similarity classes of $\bd$-computable sets.

Our goal is to use the topology of $\mcS$ to investigate coarse
computability and Turing degrees. For any degree $\mathbf{d}$,
there are sets $A$ with $\gamma_{\mathbf{d}}(A) = 0$.  (For example,
if $A$ is weakly $1$-generic relative to $\mathbf{d}$, by the
relativization of a result in \cite{HJMS} that we will revisit in
Theorem \ref{genthm}.)  The following result follows from relativizing
the theorem for $\gamma$ given in \cite[Theorem 3.4]{HJMS}.  The proof
in that paper is a slightly messy computability-theoretic
construction, but the result is obvious in the present context.

\begin{theorem}
For a degree $\mathbf{d}$, if $0 \leq r \leq 1$ then there is a
set $B$ with $\gamma_{\mathbf{d}}(B) = r$.
\end{theorem}

\begin{proof} 
Let $\alpha$ be a path from $[\emptyset]$ to a set $[A]$ with
$\gamma_{\mathbf{d}}(A) = 0$.  The function $\gamma_{\mathbf{d}}$ is
continuous along $\alpha$, so there is a point $[B]$ on $\alpha$ with
$\gamma_{\mathbf{d}}(B) = r$ by the Intermediate Value Theorem.
\end{proof}

If $[A]$ is coarsely $\bd$-computable, then $\gamma_{\bd}([A]) = 1$,
but it follows from Lemmas \ref{Rlem} and \ref{Rlem2} that the
converse fails for all degrees $\bd$.

\begin{definition} 
The \emph{core}, $\kappa(\bd)$, of a degree $\bd$ is the
collection of all coarse similarity classes $[A]$ such that $A$ is
computable from $\bd$. So $[A] \in \kappa(\bd)$ if and only if there
is a set $D \leq\sub{T} \bd$ such that $A \sim\sub{c} D$.
\end{definition}

By Lemma \ref{downclosed}, $\kappa(\bd)$ is also the collection
of all coarse similarity classes $[A]$ such that $A \in \bd$.  It is
clear that cores are countable since $\bd$ computes only countably
many sets. 

\begin{lemma}
If $\bd$ and $\be$ are degrees, then $\bd \leq \be$ if and only
if $\kappa(\bd) \subseteq \kappa(\be)$. Thus $\bd=\be$ if and only if
$\kappa(\bd)=\kappa(\be)$.
\end{lemma}

\begin{proof}
Suppose that $\kappa(\bd) \subseteq \kappa(\be)$ and let $D \in
\bd$. Then $[\mathcal J(D)] \in \kappa(\bd)$, where $\mathcal J(D)$ is
as in Definition \ref{jdef}, so there is an $\be$-computable coarse
description of $\mathcal J(D)$. By Lemma \ref{jlem}, $\bd \leq
\be$. The other direction is obvious.
\end{proof}

\begin{definition}
\label{closdef}
The \emph{closure} ${\cd}$ of the degree $\bd$ is the closure
of $\kappa(\bd)$ in the metric space $(\mcS,\delta)$. 
\end{definition}

So $\cd$ is exactly the set of those classes that are limits of points
of $\kappa(\bd)$ in the sense of the metric topology induced by
$\delta$. Thus $\cd = \{[A] : \gamma_{\mathbf d}(A) = 1\}$.

\begin{observation}
Recall that $\mcS$ has a group structure as described at the end of
Section \ref{intro}. Note that $\kappa(\bd)$ is a subgroup of $\mcS$
since if we can compute $A$ and $B$ from $\bd$, then we can compute
their symmetric difference from $\bd$. Therefore, $\cd$ is also a
subgroup since the closure of a subgroup of a topological group is
again a subgroup. Indeed, since $\cd$ is closed in $(\mathcal
S,\delta)$ and $\kappa(\bd)$ is countable, $\cd$ is also a Polish
space, and hence is a Polish group.
\end{observation}

In order to prove that closures determine degrees, that is, $\bd \leq
\be$ if and only if $\cd \subseteq \ce$, we use a relativized version
of the sets $\mathcal R(C)$, defined using the notation in Observation
\ref{hilbert}. Recall that we denote the complement of $A$ by $\neg
A$.

\begin{definition}
For sets $A$ and $C$, let $\mathcal R^A(C) = \bigoplus^{\mathcal R}_n
X_n$, where $X_n = A$ if $n \in C$, and $X_n = \neg A$ if $n \notin
C$. In particular, $\mathcal R^\omega(C)=\mathcal R(C)$ for all $C$.
\end{definition}

We note some properties of this definition.

\begin{lemma}
\begin{enumerate}
         
\item[(i)] For all $A$ and $C$, the set $\mathcal{R}^A(C)$ is a
limit of $A$-computable sets.
                 
\item[(ii)] For all sets $A$, $C$, and $G$, if
$\gamma_G(\mathcal{R}^A(C)) = 1$ then $\gamma_G(A) = 1$.

\item[(iii)] For all sets $A$, $C_1$, and $C_2$, if $C_1 \neq C_2$,
then $[\mathcal{R}^A(C_1)] \neq [\mathcal{R}^A(C_2)]$.

\end{enumerate}
\end{lemma}

\begin{proof}
The proof of part (i) is essentially the same as for $A = \omega$.    
         
For part (ii), note first that if $\gamma_G (X \oplus Y) = 1$ then
$\gamma_G(X) = \gamma_G(Y) = 1$.  This fact is pointed out for $G =
\emptyset$ in Lemma 2.4 of \cite{AHJ}, and the relativization to
arbitrary $G$ is routine.  Now assume that $\gamma_G (\mathcal{R}^A
(C) ) = 1$.  First consider the case where $0 \in C$.  Then
$\mathcal{R}^A(C) = Y \oplus A$ for some $Y$, since $A$ is coded into
the odds in $\mathcal{R}^A(C)$.  It follows that $\gamma_G(A) = 1$.
If $0 \notin C$, then $\mathcal{R}^A(C) = Y \oplus \neg A$ for some
$Y$, so $\gamma_G(\neg A) = 1$, and again it follows that $\gamma_G(A)
= 1$.
 
Part (iii) follows from the fact that every $R_n = \{m : \ 2^k \mid m
\,\mathbin{\&}\, 2^{(k + 1)} \nmid m\}$ has positive density, and if $C_1$ and
$C_2$ differ at $n$, then $\mathcal{R}^A(C_1)$ and
$\mathcal{R}^A(C_2)$ differ at every point of $R_n$.
\end{proof}

The following theorem shows that closures of distinct degrees
are indeed distinct.

\begin{theorem}
\label{T:different}  
If $\mathbf{d}$ and $\mathbf{e}$ are degrees and $\mathbf{d}
\nleq \be$, then there are continuum many similarity classes that
belong to the closure of $\mathbf{d}$ but not to the closure of
$\mathbf{e}$.
\end{theorem}

\begin{proof}
Suppose that $\mathbf{d \nleq e}$.  We first construct a single set
$F$ whose similarity class belongs to the closure of $\mathbf{d}$ but
not the closure of $\mathbf{e}$.  Let $D$ and $E$ be sets of degree
$\mathbf{d}$ and $\mathbf{e}$ respectively.  Let $A = \mathcal J(D)$
(as defined in Definition \ref{jdef}), which is Turing equivalent to
$D$.  Choose any set $C$ and let $F = \mathcal{R}^A(C)$.  We must show
that $[F]$ belongs to the closure of $\mathbf{d}$ but not the closure
of $\mathbf{e}$.

To show that $[F]$ belongs to the closure of $\mathbf{d}$, apply part
(i) of the above lemma, using the fact that $A$ is Turing equivalent
to $D$.
 
To show that $[F]$ is not in the closure of $\mathbf{e}$ it suffices
to prove that $\gamma_{\be}(F) < 1$.  Suppose for a contradiction that
$\gamma_E(F) = 1$, so $\gamma_E (\mathcal{R}^A(C)) = 1$.  It follows
from part (ii) of the lemma that $\gamma_E(A) = 1$, so
$\gamma_E(\mathcal J(D)) = 1$.  But then $D$ would be computable from
$E$ by Lemma \ref{jlem}, contradicting the hypothesis that $\mathbf{d
\nleq e}$.
  
To complete the proof of the theorem, just note that there are
continuum many choices for $C$, and apply part (iii) of the lemma.
\end{proof}

We have noted that closures of degrees are subgroups. Any group
$G$ of exponent $2$ has a well-defined dimension, $\dim(G)$, as a
vector space over the field of two elements.
 
\begin{observation}
If $\mathbf{d}$ and  $\mathbf{e}$ are degrees then 
\[ 
\dim \left( \frac{ \overline{\mathbf{e}} }{ \overline{\mathbf{d}} \cap
\overline{\mathbf{e}} } \right)
\]
is either $0$ (exactly when $\be \leq \bd$) or the cardinality of the
continuum.
\end{observation}

We can consider the distance from a single point to a subset of
$\mathcal S$ in the usual way.

\begin{definition}
\label{SETDISTANCE}
If  $[A] \in \mathcal{S}$ and $\mcB \subseteq \mathcal{S}$, then 
\[
\delta([A], \mcB) = \inf\{\delta([A],[B]) : [B] \in \mathcal{B}\}.
\]
\end{definition}

There is also a natural definition of the ``computational distance''
between a set and the closure $\cd$ of a degree $\bd$, namely:
\begin{definition}
$c([A], \cd) = 1 - \gamd(A)$.
\end{definition}

Lemma \ref{lowerupper} shows that
\[
\gamma_{\bd}([A]) = 1 - \delta([A], \cd),
\]
so $c([A], \cd) = \delta([A], \cd)$, so the computational distance
equals the metric distance. This fact shows that the density metric is the
correct metric for our situation.

If $\mathcal{A}, \mcB$ are subsets of a metric space, the
\emph{Hausdorff distance} between them is, roughly speaking, the
greatest distance from a point in either set to the other set.  More
precisely, for the metric space $(\mathcal{S}, \delta)$ the definition
is as follows: If $\mathcal{A}, \mcB \subseteq \mathcal{S}$ then the
Hausdorff distance between them is given by
\begin{equation}
\label{H1}
H(\mathcal{A},\mcB) = \max\{\sup_{x \in \mathcal{A}}
\delta(x,\mcB),\sup_{y \in \mcB} \delta(y,\mathcal{A})\}.
\end{equation}
The Hausdorff distance is always a pseudo-metric on the subsets of a
metric space, and is a metric on its closed bounded subsets. Since
$\delta$ is bounded, it is a metric on the closed subsets of
$(\mathcal{S},\delta)$.

In any metric space, the Hausdorff distance between two subsets of the
space is the same as the Hausdorff distance between their closures.
Thus, for any degrees $\mathbf{d}, \mathbf{e}$, the Hausdorff
distance between their cores is the same as the Hausdorff distance
between their closures, and it seems that this distance is a
reasonable measure of the distance between the degrees. We will
therefore use the following definition.

\begin{definition}   
The Hausdorff distance, $H(\mathbf{d}, \mathbf{e})$, between the
degrees $\bd$ and $\be$ is the Hausdorff distance $H(\cd, \ce)$
between their closures in the space $(\mcS,\delta)$. If $D \in \bd$
and $E \in \be$, we also write $H(D,E)$ for $H(\mathbf{d},
\mathbf{e})$.
\end{definition}

Using Equation \eqref{H1} we have
\begin{equation}\label{H2}
H(\mathbf{d}, \mathbf{e}) = \max\{\sup_{[A] \in \cd} \{1 -
\game([A]) \}, \sup_{[B] \in \ce} \{ 1 - \gamd([B]) \} \}.
\end{equation}
Let $\mathcal{D}$ denote the set of all Turing degrees.  In order to
calculate these Hausdorff distances between degrees, we use a
relativized version of the function $\Gamma: \mathcal{D} \to [0, 1]$
defined in \cite{ACDJL}.  We first review the definition of $\Gamma$
and some of its known properties.

For $\mathbf{a} \in \mathcal{D}$, define
\[
\Gamma(\mathbf{a})  = \inf\{\gamma(A) : A \leq\sub{T} \mathbf{a}\}.
\]
Obviously, $\Gamma(\mathbf{0}) = 1$.  We define $\Gamma$ on sets by
$\Gamma(A) = \Gamma(\mathbf{a})$, where $\mathbf{a}$ is the degree of
$A$. Let $\mathcal I(A)$ be as in Definition \ref{idef}. It was proved
in \cite{HJMS} that for all $A$, if $\gamma(\mathcal I(A)) > 1/2$ then
$A$ is computable.  It follows that if $\Gamma(\mathbf{a}) > 1/2$ then
$\mathbf{a} = \mathbf{0}$, so $\Gamma(\mathbf{a}) = 1$.  It was also
proved in \cite{HJMS} that if the degree $\mathbf{a}$ is hyperimmune
or PA, then $\Gamma(\mathbf{a}) = 0$.  Furthermore, it was proved in
\cite{ACDJL} that if the degree $\mathbf{a}$ is either nonzero and
computably traceable or both hyperimmune-free and $1$-random, then
$\Gamma(\mathbf{a}) = 1/2$.  These results showed that the range $R$
of $\Gamma$ satisfies $\{0, \ 1/2, \ 1\} \subseteq R \subseteq [0,
1/2] \cup \{1\}$, and in fact there are continuum many degrees
$\mathbf{a_1}$ such that $\Gamma(\mathbf{a_1}) = 1/2$ and also there
are continuum many degrees $\mathbf{a_2}$ such that
$\Gamma(\mathbf{a_2}) = 0$.  Several years later, these results were
capped off  in \cite{M} by B. Monin, who proved that, for all degrees
$\mathbf{a}$, if $\Gamma(\mathbf{a}) < 1/2$ then $\Gamma(\mathbf{a}) =
0$, thus establishing that $R = \{0, \ 1/2, \ 1\}$.

We can use $\Gamma$ to calculate the Hausdorff distance between
$\mathbf{0}$ and an arbitrary degree $\mathbf{a}$.  Namely, it is easy
to see that $H(\mathbf{0}, \mathbf{a}) = 1 - \Gamma(\mathbf{a})$.  In
order to find Hausdorff distances between arbitrary pairs of degrees
we need to relativize $\Gamma$.  This is done in the obvious way, by
defining
\begin{equation}
\label{G}
\Gamma_{\mathbf c}(\mathbf{a})= \inf \{\gamma_{\mathbf c} (A) : A
\leq\sub{T} \mathbf{a}\}.
\end{equation}
If $C \in \mathbf{c}$ and $A \in \mathbf{a}$ then we also write
$\Gamma_C(A)$ for $\Gamma_{\mathbf{c}}(\mathbf{a})$.

Many of the above results on $\Gamma$ relativize routinely to
$\Gamma_\mathbf{c}$ for an arbitrary degree $\mathbf{c}$ as noted in
the following proposition.

\begin{proposition}
\label{rel}
Let $\mathbf{a}, \mathbf{c}$ be degrees.
\begin{enumerate}

\item $\Gamma_{\mathbf{c}} (\mathbf{a}) = 1$ if and only if $\mathbf{a}
\leq \mathbf{c}$.

\item If $\Gamma_{\mathbf{c}} (\mathbf{a}) > 1/2$, then $\mathbf{a}
\leq \mathbf{c}$, so $\Gamma_{\mathbf{c}} (\mathbf{a}) = 1$.
        
\item There exist continuum many degrees $\mathbf{a_1} > \mathbf{c}$
such that $\Gamma_{\mathbf{c}}(\mathbf{a_1}) = 1/2$.
        
\item There exist continuum many degrees $\mathbf{a_2} > \mathbf{c}$
such that $\Gamma_{\mathbf{c}}(\mathbf{a_2}) = 0$.

\end{enumerate}
\end{proposition}       

We omit the routine proof. The result of Monin also relativizes, but
more care is needed for that, and we will deal with it later in
Theorem \ref{Monin}. The next proposition shows how the relativized
version of $\Gamma$ can be used to compute Hausdorff distances between
degrees.

\begin{proposition}
\label{Hformula}
For any two degrees $\mathbf{d}$ and $\mathbf{e}$, we have
\[
H(\mathbf{d}, \mathbf{e}) = 1 - \min\{\Gamma_{\mathbf{d}} (\mathbf{e}), 
\Gamma_{\mathbf{e}} (\mathbf{d})\}.
\]
\end{proposition}

\begin{proof}
The proposition follows directly from Equations \eqref{H2} and
\eqref{G}.  To express the right-hand side of \eqref{H2} in terms of
$H$, first note by \eqref{G} that
\[
\sup_{[A] \in  \cd} \{1 - \game([A])\} = 1 - \inf_{[A] \in
 \cd}\game([A]) = 1 - \Gamma_{\mathbf{e}}(\mathbf{d}).
\]
Of course, the same result holds if $\mathbf{d}$ and $\mathbf{e}$ are
interchanged.

Then by equation \eqref{H2},
\[
H(\mathbf{d}, \mathbf{e}) = \max\{1 - \Gamma_{\mathbf{e}}(\mathbf{d}),
1 - \Gamma_{\mathbf{d}}(\mathbf{e})\} = 1 -
\min\{\Gamma_{\mathbf{e}}(\mathbf{d}),
\Gamma_{\mathbf{d}}(\mathbf{e})\}.
\]
\end{proof}

The next corollary follows immediately from Propositions \ref{rel} and
\ref{Hformula}.

\begin{corollary}
If $\mathbf{a} \leq \mathbf{b}$, then $H(\mathbf{a}, \mathbf{b} ) =
1 - \Gamma_ \mathbf{a} (\mathbf{b})$.
\end{corollary}

Note that $H$ respects the ordering of degrees in the
sense that if $\mathbf{a}, \mathbf{b}, \mathbf{c},\mathbf{d}$
are degrees and $\mathbf{a} \leq \mathbf{b} \leq \mathbf{c} \leq
\mathbf{d}$, then $H(\mathbf{b}, \mathbf{c}) \leq H(\mathbf{a},
\mathbf{d})$. To prove that this is the case, observe that
\[
H(\mathbf{b}, \mathbf{c}) = 1 - \Gamma_{\mathbf{b}} (\mathbf{c}) \leq
1 - \Gamma_{\mathbf{a}}(\mathbf{d}) = H(\mathbf{a}, \mathbf{d}),
\]  
since $\Gamma_{\mathbf{b}} (\mathbf{c}) \geq \Gamma_{\mathbf{a}}
(\mathbf{c}) \geq \Gamma_{\mathbf{a}} (\mathbf{d})$.

The next corollary also follows immediately from Propositions
\ref{rel} and \ref{Hformula}.
  
\begin{corollary}
\label{Hvalues} 
Let $\mathbf{d}$ and $\mathbf{e}$ be degrees.
\begin{enumerate}

\item $H(\mathbf{d}, \mathbf{e}) = 0$ if and only if $\mathbf{d} =
\mathbf{e}$.

\item  $H(\mathbf{d}, \mathbf{e}) = 1/2$ if and only if one of the
following holds:
\begin{gather*}
\Gamma_\mathbf{d}(\mathbf{e}) = \Gamma_\mathbf{e}(\mathbf{d}) = 1/2\\
\mathbf{d \leq e} \, \mathbin{\&} \, \Gamma_\mathbf{d}(\mathbf{e}) =  1/2\\
\mathbf{e \leq d} \, \mathbin{\&} \, \Gamma_\mathbf{e}(\mathbf{d}) = 1/2.
\end{gather*}

\item $H(\mathbf{d}, \mathbf{e}) = 1$ if and only if either
$\Gamma_\mathbf{d}(\mathbf{e}) = 0$ or $\Gamma_\mathbf{e}(\mathbf{d})
= 0$.

\end{enumerate}
\end{corollary}

The following corollary follows immediately from the above corollary
together with Proposition \ref{rel}.

\begin{corollary}
\label{ballsize}
\begin{enumerate}

\item For every degree $\mathbf{c}$ there are continuum many
degrees $\mathbf{d} > \mathbf{c}$ such that $H(\mathbf{c}, \mathbf{d})
= 1/2$. 

\item For every degree $\mathbf{c}$ there are continuum many
degrees $\mathbf{d} > \mathbf{c}$ such that $H(\mathbf{c}, \mathbf{d})
= 1$.

\end{enumerate}
\end{corollary}

\begin{corollary}  
$(\mathcal{D},H)$  is a metric space.
\end{corollary}

\begin{proof}
As mentioned above, $H$ is a pseudo-metric on the subsets of any
metric space. To show that it is a metric on $\mathcal{D}$, assume that
$H(\mathbf{d}, \mathbf{e}) = 0$. It then follows from part (1) of the
previous proposition that $\mathbf{d} = \mathbf{e}$.
\end{proof}

The following result is a relativized form of Monin's theorem in
\cite{M} that if $\Gamma(A) < 1/2$ then $\Gamma(A) = 0$. However, some
care is needed to prove it, as explained below.

\begin{theorem}[Monin \cite{M}, relativized] 
\label{Monin}
For all sets $A$ and $D$, if  $\Gamma_D(A)  < 1/2$ then $\Gamma_D(A) = 0$.
\end{theorem}

\begin{proof}
The proof of this result is a straightforward modification of Monin's proof
for the case where $D$ is computable, but it is somewhat lengthy
because Monin's proof involves so many steps.  First, if we state
Monin's theorem in terms of $\gamma$ rather than $\Gamma$, it says the
following:
\[
(\forall \epsilon > 0) (\forall X) [ \text{if } \gamma(X) < 1/2 \text{
    then }(\exists Y \leq\sub{T} X) [\gamma(Y) < \epsilon]].
\]
Any proof of Monin's theorem should easily relativize to show the
following:
\[
(\forall D)(\forall \epsilon > 0) (\forall X) [\text {if } \gamma_D
  (X) < 1/2 \text{ then } (\exists Y \leq\sub{T} X \oplus D) [\gamma_D
    (Y) < \epsilon]].
\]
However, this relativization is not good enough for us.  We need to
ensure that $Y \leq\sub{T} X$, not merely that $Y \leq\sub{T} X \oplus
D$.  In fact, this stronger result comes right out of Monin's proof:
The set $Y$ is constructed from $X$ by a series of intermediate steps
in his argument.  If one relativizes this argument to $D$, one never
uses a $D$-oracle to construct a set or function, although the
constructed objects have certain properties pertaining to
$D$-computable functions.

We follow the proof of Theorem 3.8 of Monin's paper \cite{M} and use
his notation.  We assume the reader is familiar with Monin's proof and
has access to his paper.  We start with a set $X$, an oracle $D$ such
that $\gamma_D (X) < 1/2$, and a real $\epsilon > 0$.  Our goal is to
construct a set $B \leq\sub{T} X$ with $\gamma_D(B) < \epsilon$. We do
so via the following steps, where all references are to \cite{M}.
\begin{enumerate}

\item Define the notion of infinitely often equal as in Definition
  3.3, but with ``computable'' replaced by ``$D$-computable''.

\item From $X$ compute a sequence of strings $\{\sigma_n\}_{n \in
\omega}$ as in Corollary 3.9.  This corollary uses Theorem 3.8. The
proof of that theorem allows us to choose this sequence of strings to
be $X$-computable, rather than merely $(X \oplus D)$-computable.

\item  Using Theorem 2.4, construct for each $n$ a set
$C_n$ of strings as described in the proof of Theorem 3.11.
The sets $C_n$ are uniformly computable, and relativization to $D$
has no effect on this step.  Let $C_n$ be effectively listed as
$\tau_0^n, \tau_1^n ,\dots$.

\item  As in the proof of Theorem 3.11, define $T_n = \{i :
\delta(\sigma_n, \tau^n_i) \leq 1/2 - \epsilon'\}$, where
$\epsilon'$ is such that $0<\epsilon'<\epsilon$ and $\delta$ is
normalized Hamming distance as defined in Definition 2.1. Then, as
shown in the proof of Theorem 3.11, there is an $L$ such that $|T_n|
\leq L$ for all $n$.  Also. the sets $T_n$ are uniformly computable
from the sequence $\{\sigma_n\}_{n \in \omega}$ in (2), again
without using a $D$-oracle.

\item  Show that for every suitably bounded $D$-computable function
$g$, we have $g(n) \in T_n$ for infinitely many $ n$, as in the
proof of Theorem 3.11.  (This step does not involve a construction,
so we do not worry that it involves $D$-computable functions.)

\item Again as in the proof of Theorem 3.11, show that there is a
sequence of sets $\{T'_n\}_{n \in \omega}$, uniformly computable
from the sequence $\{T_n\}_{n \in \omega}$ (not using $D$), with
each $T'_n$ of size at most $L$ and $\max T'_n \leq 2^{L 2^n}$ for
all $n$, such that for every $D$-computable function $g$
bounded by $2^{L2^n}$, we have $g(n) \in T'_n$ for infinitely many
$n$.  

\item Continuing to follow the proof of Theorem 3.11, show that there
is a function $h$ computable from the sequence $\{T'_n\}_{n \in
\omega}$ (not using $D$) such that for every $D$-computable
function $g$ bounded by $2^{2^n}$ there exist infinitely many $n$
for which $h(n) = g(n)$.

\item    Use the proof of Theorem 3.6 to show that there
is a set $B \leq\sub{T} h$ (not $h \oplus D)$ such that $\gamma_D(B) <
\epsilon$. We have $B \leq\sub{T} X$ since
\[
B \leq\sub{T} h \leq\sub{T} \{T'_n\}_{n \in \omega} \leq\sub{T}
\{T_n\}_{n \in \omega} \leq\sub{T} \{\sigma_n\}_{n \in \omega}   \leq\sub{T}   X.
\]

\end{enumerate}
\end{proof}

The following corollary follows at once from the theorem and
Proposition \ref{rel}.

\begin{corollary}
For every degree $\mathbf{c}$ the range of $\Gamma_{\mathbf c}$ is
$\{0, 1/2, 1\}$.
\end{corollary}

This corollary and Proposition \ref{Hformula} yield the following.

\begin{corollary}
\label{3Hvalues}
The possible Hausdorff distances between degrees are exactly $0$, $1/2$,
and $1$.
\end{corollary}

Although we do not pursue this idea further here, it is worth noting
that $H$ can be extended from degrees to downward-closed sets of
degrees in a straightforward way. Let $\mathcal E$ be the collection
of all downward-closed sets of degrees. Then $(\mathcal E,H)$ is still
a $0,1/2,1$-valued metric space, and we can think of $(\mathcal D,H)$
as a subspace by identifying a degree with its lower cone. An
intermediate space that could also be worth studying is that of Turing
ideals.

\section{Hausdorff distance, Lebesgue measure, and Baire category}
\label{freqsec}

It is natural to ask which Hausdorff distance occurs ``most
frequently'' between pairs of degrees.  The answer depends on whether
the question is formalized using Lebesgue measure or Baire
category. Indeed, as we will see in this and the next section, the
interplay between typicality in the sense of measure (as captured
computability-theoretically in notions such as $1$-randomness) and
typicality in the sense of category (as captured
computability-theoretically in notions such as (weak) $1$-genericity)
seems central to understanding the structure of $(\mathcal
D,H)$. Randomness leads to constructions of degrees at distance $1/2$,
while genericity leads to constructions of degrees at distance $1$.
We use $(\cm A)$ to abbreviate ``for comeager many $A$'' (in the usual
topology on $2^\omega$) and $(\aev A)$ to abbreviate ``for almost
every $A$'' (in the usual coin-toss measure on $2^\omega$).

We first consider Baire category, in the usual topology on
$2^\omega$. Recall that we write $\Gamma_A$ for $\Gamma_{\mathbf{a}}$
where $\mathbf{a}$ is the degree of $A$.  In \cite{HJMS} it is shown
that if $B$ is weakly $1$-generic then $\gamma(B)=0$. The proof of
this result relativizes to establish the following fact.

\begin{theorem}[Hirschfeldt, Jockusch, McNicholl, and Schupp
{\cite[proof of Theorem 2.2, relativized]{HJMS}}]
\label{genthm}
If $B$ is weakly $1$-generic relative to $A$, then $\gamma_A(B)=0$, so
$\Gamma_A(B)=0$, and hence $H(A,B)=1$. Therefore, for all $C
\geq\sub{T} B$, we have $\Gamma_A(C) = 0$, and hence $H(A, C) = 1$.
\end{theorem}

\begin{corollary}
\label{H3}
$(\forall A)(\cm B)[H(A,B)=1]$.
\end{corollary}

It also follows that there is an uncountable family $\mathcal{C}$ of
degrees such that any two distinct degrees $\mathbf{a}, \mathbf{b} \in
\mathcal{C}$ satisfy $H(\mathbf{a}, \mathbf{b}) = 1$: Consider a
family $\mathcal{C}$ with the property that any two distinct degrees
$\mathbf{a}, \mathbf{b}$ in $\mathcal{C}$ satisfy $H(\mathbf{a},
\mathbf{b}) = 1$.  If $\mathcal{C}$ is countable, then it follows from
the previous corollary that there is a $\mathbf{c} \notin \mathcal C$
such that $\mathcal C \cup \{\mathbf{c}\}$ also has this
property. Thus $\mathcal C$ is not maximal. So any maximal family with
this property is uncountable. We can improve this result using
Mycielski's Theorem for category, which we stated as Theorem
\ref{Ramseyforcat} and will prove as Theorem \ref{myccatthm} below.

\begin{corollary}
\label{distance1}
There is a family $\mathcal{C}$ of degrees of size continuum such that
any two distinct degrees $\mathbf{a}, \mathbf{b} \in \mathcal{C}$
satisfy $H(\mathbf{a}, \mathbf{b}) = 1$.
\end{corollary}

We now consider analogous results for measure. We first examine
Hausdorff distances from the least degree $\mathbf{0}$.

\begin{proposition}
\label{M}
$(\aev B)[H(\emptyset,B)=1]$.
\end{proposition}

\begin{proof}
By a result of D. A. Martin (unpublished, see e.g.\ \cite[Theorem
8.21.1]{DH}), almost every set is of hyperimmune degree. As noted in
\cite{HJMS}, since every hyperimmune degree computes a weakly
$1$-generic, it follows from Theorem \ref{genthm} that
$\Gamma(\mathbf{b})=0$ for every hyperimmune degree $\mathbf{b}$, so
$H(\mathbf{0},\mathbf{b})=1$ for every such $\mathbf{b}$.
\end{proof}

One might hope to prove by relativizing the above proposition that for
every $A$, we have $H(A,B)=1$ for almost every $B$, which would be in
complete analogy with Corollary \ref{H3}. However, relativization
yields only the result that for every $A$, we have $H(A, A \oplus
B)=1$ for almost every $B$. In fact, we have the following result. We
state it in terms of the notion of Church stochasticity (see
\cite[Definition 7.4.1]{DH}), which is implied by $1$-randomness (and
even by computable randomness), but all we need is the fact that if
$B$ is Church stochastic relative to $A$, then for any infinite
$A$-computable set $C$, the density of $B$ within $C$ is $1/2$.

\begin{theorem}
If $B$ is Church stochastic relative to $A$ then $\Gamma_B(A) \geq
1/2$, so if $A$ and $B$ are Church stochastic relative to each other
then $H(A,B)=1/2$.
\end{theorem}

\begin{proof}
Suppose that $B$ is Church stochastic relative to $A$. Let $C
\leq\sub{T} A$ be infinite and coinfinite. Then the density of $B$
within $C$ and the density of $B$ within the complement of $C$ must
both be $1/2$. Thus $\rho(C \sd B) = 1/2$. That is, $B$ is a
$1/2$-description of every  infinite, coinfinite $A$-computable set,
and hence $\Gamma_B(A) \geq 1/2$.

If $A$ and $B$ are Church stochastic relative to each other then
$\Gamma_B(A) \geq 1/2$ and $\Gamma_A(B) \geq 1/2$, but $A$ and $B$ are
also Turing incomparable, so in fact $\Gamma_B(A) = \Gamma_A(B) =
1/2$, and hence $H(A,B)=1/2$.
\end{proof}

The following corollary holds for computable randomness as well, but
we state it for $1$-randomness as that is the version we will use
below.

\begin{corollary}
\label{randcor}
If $B$ is $1$-random relative to $A$ then $\Gamma_B(A) \geq 1/2$, so
if we also have $A \nleq\sub{T} B$, then $\Gamma_B(A)=1/2$. Thus, if
$A$ and $B$ are relatively $1$-random, then $H(A,B)=1/2$.
\end{corollary}

\begin{corollary}
\label{excor}
For every $A$ there exists a $B$ such that $H(A, B) = 1/2$ and  $B$ is $1$-random relative to $A$.
\end{corollary}

\begin{proof}
Let $B$ be $1$-random relative to $A$ and such that every
$B$-computable function is dominated by an $A$-computable
function. Such a set exists by the relativized version of the
hyperimmune-free basis theorem. Then $\Gamma_A(B)=1/2$ by the
relativized form of \cite[Corollary 1.13]{ACDJL}, while
$\Gamma_B(A)=1/2$ by Corollary \ref{randcor}. Thus $H(A,B)=1/2$.
\end{proof}

The class of $1$-randoms has measure $1$. Furthermore, if $A$ is
$1$-random then the class of sets that are $1$-random relative to $A$
also has measure $1$, and if $B$ is in this class then $A$ and $B$ are
relatively $1$-random, by van Lambalgen's Theorem. Thus we have the
following fact.

\begin{corollary}
$(\aev A)(\aev  B)[H(A,B) = 1/2]$.
\end{corollary}

It also follows that there is an uncountable subset $\mathcal{C}$ of
$2^\omega$ such that if $A$ and $B$ are any two distinct element of
$\mathcal{C}$, then $H(A, B) = 1/2$: Let $\mathcal E$ be the class of
all $A$ such that $H(A,B) = 1/2$ for almost every $B$. By the previous
corollary, $\mathcal E$ has measure $1$.  By Zorn's Lemma there is a
maximal family $\mathcal{C} \subseteq \mathcal{E}$ with the property
in the statement of the corollary. It follows from the previous
corollary and the countable additivity of Lebesgue measure that
$\mathcal{C}$ cannot be countable. We can improve on this result by
applying Corollary \ref{perfrandcor}.

\begin{corollary}
\label{perfhalf}
There is a subset $\mathcal{C}$ of $2^\omega$ of size continuum such that
if $A$ and $B$ are any two distinct element of $\mathcal{C}$, then $H(A,
B) = 1/2$.
\end{corollary}

Kolmogorov's $0$-$1$ Law implies that any measurable collection of
sets that is closed under Turing equivalence has measure $0$ or
$1$. In particular, for every $A$, the class of all $B$ such that
$H(A,B)=1$ always has measure $0$ or $1$. As we have seen, if $A$ is
$1$-random then this class has measure $0$. 

\begin{definition}
\label{attdef}
A set $A$, and the degree of $A$, are \emph{attractive} if the class
of all $B$ such that $H(A,B)=1/2$ has measure $1$, or equivalently, the
class of all $B$ such that $H(A,B)=1$ has measure $0$. Otherwise, $A$
and its degree are \emph{dispersive}.
\end{definition}

It follows from Proposition \ref{M} that $\emptyset$ is dispersive,
and it follows from Corollary \ref{randcor} and van Lambalgen's
Theorem that every $1$-random set is attractive.

\begin{proposition}
The class of attractive degrees is closed upwards. Equivalently, the
class of dispersive degrees is closed downwards.
\end{proposition}

\begin{proof}
Suppose that $A$ is dispersive and $C \leq\sub{T} A$. By Corollary
\ref{Hvalues}, we have $(\aev B) [\Gamma_A (B) = 0 \text{ or }
\Gamma_B(A) = 0]$.  By Corollary \ref{randcor}, we have $\Gamma_B(A)
\geq 1/2$ for almost every $B$, so $\Gamma_A(B)=0$ for almost every
$B$. Since $C \leq\sub{T} A$, we have $\Gamma_C (B) \leq \Gamma_A (B)$
for all $B$.  Hence $\Gamma_C (B) = 0$ for almost every $B$. It
follows from Corollary \ref{Hvalues} that $H(C, B) = 1$ for almost
every $B$, so $C$ is dispersive.
\end{proof}

It follows from the above proposition and the remark just before it
that if a set computes a $1$-random then it is attractive.  In
particular, $\emptyset'$ is attractive.  We will see in Observation
\ref{aedobs} that not every attractive set computes a $1$-random, but
let us first discuss the dispersive sets. The following proposition
follows at once from Theorem \ref{genthm} and will be used frequently
to show that sets are dispersive.

\begin{proposition}
If almost every set computes a set that is weakly $1$-generic relative
to $A$, then $A$ is dispersive.
\end{proposition}

Recall from Proposition \ref{M} that the empty set is dispersive. It
follows from known results that there are noncomputable dispersive
sets. A degree $\textbf{a}$ is \emph{low for weak $1$-genericity} if
every weakly $1$-generic is still weakly $1$-generic relative to
$\textbf{a}$. If $\textbf{a}$ is low for weak $1$-genericity, then
$\textbf{a}$ is dispersive by the above proposition, since almost
every degree is hyperimmune, and thus computes a set that is weakly
$1$-generic relative to $\textbf{a}$. Stephan and Yu \cite{SY} showed
that a degree is low for weak $1$-genericity if and only if it is
hyperimmune-free and not diagonally noncomputable. It follows that all
sets that are sufficiently generic for forcing with computable perfect
trees are dispersive. The following theorems give further examples. In
particular, we show that there is a high c.e.\ degree that is
dispersive and hence, as we will remark, that there is a computably
random set that is dispersive. We also show that every low
c.e.\ degree is dispersive, and that every weakly $2$-generic degree
is dispersive, so that the class of dispersive sets is
comeager. Indeed, we will show that if $A$ is weakly $2$-generic and
$B$ is $2$-random, then $H(A,B)=1$.

We begin with the following result. Although we will strengthen it
below, we include a proof because it will be helpful in explaining the
proofs of these stronger versions.

\begin{theorem}
\label{cethm}
There is a noncomputable c.e.\ set $A$ such that almost every set
computes a set that is weakly $1$-generic relative to $A$, and hence
$A$ is dispersive.
\end{theorem}

\begin{proof}
This proof is a finite-injury priority construction, based on Martin's
proof mentioned above that the hyperimmune degrees have measure
$1$. That proof, as presented for instance in \cite[Theorem
8.21.1]{DH}, can easily be adapted to give a direct proof that
almost every set computes a weakly $1$-generic (which is essentially
what we would get if we removed the $R$-requirements from the proof
below).

By Kolmogorov's $0$-$1$ Law, it is enough to build a Turing functional
$\Psi$ such that the set of $X$ for which $\Psi^X$ is weakly
$1$-generic relative to $A$ has positive measure. We will ensure that
the set of $X$ such that $\Psi^X$ is total has positive measure, while
satisfying requirements
\[
R_e : \Phi_e \neq A
\]
and
\[
Q_e : W_e^A \textrm{ dense and } \Psi^X \textrm{ total} \ \Rightarrow
\ \Psi^X \textrm{ meets } W_e^A.
\]
We arrange these in a priority ordering $Q_0,R_0,Q_1,R_1,\dots$. We
think of $\Psi$ as being defined in stages, where at stage $s$ we
define $\Psi^\tau$ for the strings $\tau$ of length $s$. During the
construction, certain strings will be \emph{claimed} by
$Q$-requirements. If $\tau$ extends such a string $\sigma$, then we do
not allow $\Psi^\tau$ to converge on any new values (i.e., we define
$\Psi^\tau=\Psi^{\tau \restr (s-1)}$), unless the strategy for the
requirement claiming $\sigma$ defines it otherwise, as discussed
below. Otherwise, we ensure that $\Psi^\tau(n)$ is defined for all $n
\leq s$. (The actual values do not matter in this case, so if
$\Psi^{\tau \restr (s-1)}(n)$ is not defined for such an $n$, then we
just let $\Psi^\tau(n)=0$.)

We satisfy $R_e$ in the usual way by choosing a witness $n$, waiting
until we see that $\Phi_e(n)=0$ (if ever), and then enumerating $n$
into $A$. Each time $R_e$ is initialized, it chooses a new witness
larger than any number previously mentioned in the construction, which
means that the value of $A$ on this number cannot affect any currently
existing computation.

To satisfy a single $Q_e$, we could proceed as follows. Let
$\sigma_0,\ldots,\sigma_{2^{e+2}-1}$ be the strings of length
$e+2$. We begin by claiming $\sigma_0$ at some stage $s > e+2$. Let
$\tau_0,\ldots,\tau_m$ be the extensions of $\sigma_0$ of length
$s-1$. For each such $\tau_i$, we have defined $\Psi^{\tau_i}=\mu_i$
for some $\mu_i$. We now wait until a stage $t \geq s$ such that for
each $i \leq m$, there is an extension $\nu_i$ of $\mu_i$ in
$W_e^A[t]$. If such a $t$ is never found, then $\sigma_0$ is
permanently claimed by $Q_e$, and $\Psi^X$ is not total for $X \succ
\sigma_0$, but the set of such $X$ has measure only $2^{-(e+2)}$. If
$t$ is found then we try to ensure that each $\nu_i$ is in $W_e^A$ by
initializing weaker priority $R$-requirements. (Of course, stronger
priority $R$-requirements might still act, but in that case we simply
restart our strategy for $Q_e$.) We then ensure that $\nu_i \prec
\Psi^X$ for each $i \leq m$ and each $X$ extending $\tau_i$, drop our
claim on $\sigma_0$, claim $\sigma_1$, and repeat our procedure.

In this way, we move through the strings of length $e+2$, with one of
two eventual outcomes. We might eventually permanently claim some
$\sigma_k$. If so, then $W_e^A$ is not dense, so $Q_e$ is satisfied,
and we have removed only $2^{-(e+2)}$ from the measure of $\{X : \Psi^X
\textrm{ total}\}$. Otherwise, we ensure that $\Psi^X$ extends some
element of $W_e^A$ for every $X$, again satisfying $Q_e$.

When considering all our $Q$-requirements at once, the only difference
is that we need to ensure that while $Q_e$ is claiming a string
$\sigma$, no $Q_i$ with $i>e$ can claim an extension of
$\sigma$. Notice that the total measure removed from $\{X : \Psi^X
\textrm{ total}\}$ by permanently claimed strings over the whole
construction is at most $\sum_e 2^{-(e+2)} = 1/2$.

We now proceed with the full construction. We think of $\Psi$ as a
computable function $2^{<\omega} \rightarrow 2^{<\omega}$, whose value
at $\sigma$ is denoted by $\Psi^\sigma$, such that if $\sigma \prec
\tau$ then $\Psi^\sigma \preccurlyeq \Psi^\tau$. Then $\Psi^X =
\bigcup_n \Psi^{X \restr n}$, so $\Psi^X$ is total if and only if
$\lim_n |\Psi^{X \restr n}| = \infty$. For two strings $\sigma$ and
$\tau$ of the same length, let $\sigma <\sub{l} \tau$ if $\sigma$
comes before $\tau$ in the lexicographic ordering.

When an $R$-strategy is initialized, its witness becomes
undefined. When a $Q$-strategy is initialized, it gives up any current
claims it might have, and is declared to be unsatisfied.

We begin with $\Psi^\lambda=\lambda$, where $\lambda$ is the empty
string.  At each stage $s>0$, we define $\Psi^\sigma$ for all $\sigma$
of length $s$, proceeding as follows.

First, for each $R_e$ with $e \leq s$ that is not yet satisfied and
does not currently have a witness, assign $R_e$ a witness larger than
any number appearing in the construction so far. Then, for each $R_e$
with $e \leq s$ that has a witness $n$ such that $\Phi_e(n)[s]=0$ and
$n$ is not yet in $A$, put $n$ into $A$, initialize all $Q_i$ with $i
> e$, and declare $R_e$ to be satisfied.

Now say that $Q_e$ with $e < s-2$ \emph{requires attention} if $Q_e$
is not currently declared to be satisfied, and either $Q_e$ is not
claiming any string or it is claiming a string $\sigma$, and for every
$\tau \succcurlyeq \sigma$ of length $s-1$, there is an extension
$\nu$ of $\Psi^\tau$ currently in $W_e^A$. For the least $e$ such that
$Q_e$ requires attention (if any), we act as follows.

If $Q_e$ is not claiming any string, then $Q_e$ claims the
$<\sub{l}$-least string $\tau$ of length $e+2$ such that no initial
segment of $\tau$ is currently being claimed by any $Q_i$ with
$i<e$. (Since each $Q_i$ can be claiming at most one string, and that
string must have length $i+2$, such a $\tau$ must exist.)

Otherwise, proceed as follows. For each $\tau \succcurlyeq \sigma$ of
length $s-1$, let $\nu$ be an extension of $\Psi^\tau$ currently in
$W_e^A$ and define $\Psi^{\tau^\frown0}=\Psi^{\tau^\frown1}=\nu$. Now
$Q_e$ drops its claim on $\sigma$ and claims the next $<\sub{l}$-least
string $\tau >\sub{l} \sigma$ of length $e+2$ such that no initial
segment of $\tau$ is currently being claimed by any $Q_i$ with
$i<e$. If there is no such $\tau$, then declare $Q_e$ to be satisfied.

In any case, initialize every $Q_i$ with $i>e$ and every $R_i$ with $i
\geq e$.

Finally, for each $\mu$ of length $s$ such that $\Psi^\mu$ is not yet
defined, if some initial segment of $\mu$ is currently claimed by a
$Q$-strategy then define $\Psi^{\mu}=\Psi^{\mu \restr (s-1)}$, and
otherwise define $\Psi^{\mu}=\Psi^{\mu \restr (s-1)}{}^\frown0$.

This completes the construction of $A$ and $\Psi$. We now verify its
correctness. Note that, for all $s$, no two compatible strings are
claimed by different requirements at the end of stage $s$. If no
initial segment of $X$ is ever permanently claimed, then there are
infinitely many stages $s$ such that no initial segment of $X$ is
claimed at the end of stage $s$, and hence $\Psi^X$ is total, so
$\mu(\{X : \Psi^X \textrm{ total}\}) \geq 1 - \sum_e 2^{-(e+2)} =
1/2$.

Whenever an $R$-strategy puts a number into $A$, it is permanently
satisfied. If $Q_e$ is not initialized, then it goes through the
strings of length $e+2$ in lexicographic order and thus eventually
stops requiring attention. Thus, by induction, each requirement is
initialized only finitely often.

Thus each $R_e$ has a final witness $n$, and we ensure that $\Phi_e(n)
\neq A(n)$.

If $Q_e$ ever permanently claims a string, then $W_e^A$ is not
dense. Otherwise, for each string $\sigma$ of length $e+2$ such that
no initial segment of $\sigma$ is eventually permanently claimed by
some $Q_i$, the strategy for $Q_e$ ensures that if $X$ extends
$\sigma$ then $\Psi^X$ has an initial segment in $W_e^A$. So if
$\Psi^X$ is total then it meets $W_e^A$.
\end{proof}

The above proof can be adapted to show that $A$ can be made to be
high, and also that it can be chosen to be any low set, as we show in
the next two theorems. (See Observation \ref{aedobs} for a remark on
how far these results could be extended.)

\begin{theorem}
\label{cethm2}
There is a high c.e.\ set $A$ such that almost every set computes a
set that is weakly $1$-generic relative to $A$, and hence $A$ is
dispersive.
\end{theorem}

\begin{proof}
For a set $X$, let $X^{[e]}=\{n : \langle e,n \rangle \in X\}$.  To
make $A$ high, we use the fact that there is a c.e.\ set $C$ such that
each $C^{[e]}$ is either finite or equal to $\omega^{[e]}$, and if
$A^{[e]} =^* C^{[e]}$ for all $e$ then $A$ is high. (See
e.g.\ \cite[Section 2.14.3]{DH}. Here $=^*$ is equality up to finitely
many elements.)

The basic idea of this proof is that we have the same requirements
\[
Q_e : W_e^A \textrm{ dense and } \Psi^X \textrm{ total} \ \Rightarrow
\ \Psi^X \textrm{ meets } W_e^A
\]
as in the previous proof, but the $R$-requirements in that proof are
replaced by
\[
R_e : A^{[e]} =^* C^{[e]}.
\]
If $C^{[e]}$ is finite, then the action we need to take to satisfy
$R_e$ is finitary, so its effect is similar to that of the
$R$-requirements in the previous proof. Otherwise, this action is
infinitary, but it is computable. A weaker priority strategy for
satisfying a $Q$-requirement can guess at an $m$ such that $n \in
A^{[e]}$ for all $n \geq m$. Then it does not believe an enumeration
into $W_e^A[s]$ with use $u$ unless every $n \in [m,u)$ is in
$A^{[e]}[s]$. Of course, we cannot know whether $C^{[e]}$ is finite or
not, so we proceed as usual and make this into an infinite-injury
priority construction, using a tree of strategies.

For each $e$ and $\alpha \in 2^e$, we have strategies $R_\alpha$ for
$R_e$ and $Q_\alpha$ for $Q_e$. The strategy $Q_\alpha$ works under
the assumption that $R_\beta$ is infinitary if $\beta^\frown0
\preccurlyeq \alpha$ and finitary if $\beta^\frown1 \preccurlyeq
\alpha$.

For binary strings $\alpha$ and $\beta$, write $\alpha <\sub{L} \beta$
if $\alpha$ is above or to the left of $\beta$, i.e., if $\alpha \prec
\beta$ or there is a $\gamma$ such that $\gamma^\frown0 \preccurlyeq
\alpha$ and $\gamma^\frown1 \preccurlyeq \beta$.

At stage $s$, we define a string $\gamma_s$ of length $s$ such that
$R_\alpha$ and $Q_\alpha$ are allowed to act at that stage if and only
if $\alpha \preccurlyeq \gamma_s$. We say these strategies are
\emph{accessible} at stage $s$, and allow them to act in order (i.e.,
$Q_{\gamma_s \restr n}$ acts before $R_{\gamma_s \restr n}$, which
acts before $Q_{\gamma_s \restr n+1}$).  We will describe the details
of these actions below, but we can already say how $\gamma_s$ is
defined. Suppose we have defined $\gamma_s \restr n$ for $n<s-1$. If
$R_{\gamma_s \restr n}$ enumerates any numbers into $A$ at stage $s$,
then $\gamma_s(n)=0$. Otherwise, $\gamma_s(n)=1$. As usual, the
\emph{true path} of the construction is $\liminf_s \gamma_s$, i.e.,
the leftmost path visited infinitely often.

The strategy $R_\alpha$ works as follows at stages at which it is
accessible. It chooses a number $m_\alpha$, which is picked to be a
fresh large number each time $R_\alpha$ is initialized. At any stage
$s$ at which $R_\alpha$ is accessible, it enumerates every $n \geq
m_\alpha$ in $C^{[e]}[s] \setminus A^{[e]}[s]$ into $A^{[e]}[s]$. If
there is at least one such $n$, then it initializes all $Q_\beta$ with
$\alpha <\sub{L} \beta$.

The strategy $Q_\alpha$ acts similarly to the strategies for
$Q$-requirements in the previous construction, but instead of working
with strings of a length fixed ahead of time, it has a parameter
$k_\alpha>1$, and works with strings of that length. Every time
$Q_\alpha$ is initialized, $k_\alpha$ is picked to be a fresh large
number (which ensures that it is larger than the lengths of the
strings currently being used by strategies above or to the left of
$Q_\alpha$ in the tree of strategies). As before, for strings $\sigma$
and $\tau$ of the same length, let $\sigma <\sub{l} \tau$ if $\sigma$
comes before $\tau$ in the lexicographic ordering.

Suppose that $Q_\alpha$ is accessible at stage $s$. If $Q_\alpha$ is
not claiming any string, is not currently satisfied, and was not
initialized at this stage, then $Q_\alpha$ claims the $<\sub{L}$-least
string $\tau$ of length $k_\alpha$ such that no initial segment of
$\tau$ is currently being claimed by any $Q_\beta$ with $\beta
<\sub{L} \alpha$.  To see that such a $\tau$ exists, note that
$k_\alpha$ is chosen to be larger than $k_\beta$ for all $\beta
<_L\alpha$, the $k_\beta$'s are distinct, and each strategy can claim
at most one string at a time. Hence there is at most one claimed
string of each length less than $k_\alpha$ that $Q_\alpha$ must
respect, so the total number of strings of length $k_\alpha$ that
cannot be chosen to be $\tau$ is at most the sum of $2^n$ for
$1<n<2^{k_\alpha}$. Thus not every $\tau$ of length $2^{k_\alpha}$ is
forbidden.

Now suppose that $Q_\alpha$ is claiming a string $\sigma$, and for
every $\tau \succcurlyeq \sigma$ of length $s-1$, there is an
extension $\nu$ of $\Psi^\tau$ currently in $W_e^A$ such that, for the
use $u$ of this enumeration and every $\beta$ such that $\beta^\frown0
\prec \alpha$, every number in $[m_\beta,u)$ is currently in
$A^{[|\beta|]}$. Then proceed as follows. For each $\tau \succcurlyeq
\sigma$ of length $s-1$, choose a $\nu$ as above and define
$\Psi^{\tau^\frown0}=\Psi^{\tau^\frown1}=\nu$. Now $Q_\alpha$ drops
its claim on $\sigma$ and claims the next $<\sub{l}$-least string
$\tau >\sub{l} \sigma$ of length $k_\alpha$ such that no initial
segment of $\tau$ is currently being claimed by any $Q_\beta$ with
$\beta <\sub{L} \alpha$. If there is no such $\tau$, then declare
$Q_\alpha$ to be satisfied.

In either of the above cases, initialize every $Q_\beta$ with $\alpha
<\sub{L} \beta$ and every $R_\beta$ with $\alpha \leq\sub{L} \beta$.

If neither of these cases holds, then $Q_\alpha$ does nothing at this
stage.

After all accessible strategies have acted at stage $s$, for each $\mu$
of length $s$ such that $\Psi^\mu$ is not yet defined, if some initial
segment of $\mu$ is currently claimed by a $Q$-strategy then define
$\Psi^{\mu}=\Psi^{\mu \restr (s-1)}$, and otherwise define
$\Psi^{\mu}=\Psi^{\mu \restr (s-1)}{}^\frown0$.

Now, as before, only one string of each length greater than $1$ can
ever be permanently claimed by any strategy, and $\Psi^X$ is total
unless some initial segment of it is eventually permanently claimed,
because whenever a strategy gives up its claim on a string, no other
strategy can claim an extension of that string at that stage. Thus
$\mu(\{X : \Psi^X \textrm{ total}\}) \geq 1 - \sum_n 2^{-(n+2)} =
1/2$.

An argument by induction, much as before, shows that any strategy on
the true path is initialized only finitely often. Thus, if $R_\alpha$
with $|\alpha|=e$ is on the true path then it puts all sufficiently
large elements of $C^{[e]}$ into $A^{[e]}$, and hence ensures that
$R_e$ is satisfied.

Now let $Q_\alpha$ with $|\alpha|=e$ be on the true path. Let
$\beta_0,\ldots,\beta_{k-1}$ be all of the strings such that
$\beta_i{}^\frown 0 \preccurlyeq \alpha$, and let $m_i$ be the final
value of $m_{\beta_i}$. If $Q_\alpha$ ever permanently claims a string
$\sigma$, then there is a $\tau \succcurlyeq \sigma$ such that for
every sufficiently large stage $s$ at which $Q_\alpha$ is accessible,
if there is an extension $\nu$ of $\tau$ in $W_e^A[s]$, then for the
use $u$ of this enumeration, there are an $i<k$ and an $n \in [m_i,u)$
such that $n \notin A^{[|\beta_i|]}[s]$. This $n$ will eventually be
put into $A^{[|\beta_i|]}$ by $R_{\beta_i}$, so $\tau$ has no
extension in $W_e^A$. Thus in this case $W_e^A$ is not dense.

Otherwise, for the final value of $k_\alpha$ and each string $\sigma$
of length $k_\alpha$ such that no initial segment of $\sigma$ is
eventually permanently claimed by some $Q_\beta$ with $\beta <\sub{L}
\alpha$, the strategy $Q_\alpha$ ensures that  if $X$ extends
$\sigma$ then $\Psi^X$ has an initial segment in $W_e^A$. So if
$\Psi^X$ is total then it meets $W_e^A$.
\end{proof}

An interesting consequence of the above theorem is that computable
randomness is not enough to ensure that a set is attractive, because
every high c.e.\ degree contains a computably random set, as shown by
Nies, Stephan, and Terwijn \cite{NST}. (A computably random dispersive
set gives an example of a set that is computably random,
but not computably random relative to $A$ for measure-$1$ many sets
$A$, which is a strong form of the failure of van Lambalgen's Theorem
for computable randomness, though there are easier means to obtain
this kind of example.)

\begin{theorem}
Let $A$ be a low c.e.\ set. Then almost every set computes a set that
is weakly $1$-generic relative to $A$, and hence $A$ is dispersive.
\end{theorem}

\begin{proof}
This proof is again based on that of Theorem \ref{cethm}, but now we
have no $R$-requirements, since $A$ is given to us. Thus we have only
the requirements
\[
Q_e : W_e^A \textrm{ dense and } \Psi^X \textrm{ total} \ \Rightarrow
\ \Psi^X \textrm{ meets } W_e^A.
\]
The basic idea is the following. Suppose that at a stage $s$, the
requirement $Q_e$ is currently claiming a string $\sigma$ and seems to
require attention, i.e., for every $\tau \succcurlyeq \sigma$ of
length $s-1$, there is an extension $\nu$ of $\Psi^\tau$ currently in
$W_e^A$. Then we can test the enumerations of these strings $\nu$ into
$W_e^A$ using a computable approximation to $A'$. Roughly speaking, we
can use this approximation to $A'$ to guess at whether these $\nu$ are
truly in $W_e^A$, by waiting until either the approximation says that
they are, or at least one of them leaves $W_e^A$. If the former
happens, then we think of $A'$ as certifying the enumerations of the
strings $\nu$, and declare that $Q_e$ does in fact require
attention. The idea is that by using the approximation to $A'$
carefully, certified enumerations will be incorrect only finitely
often.

More precisely, during the construction we build uniformly c.e.\ sets
of strings $D_{e,k,\sigma}$ for $e,k \in \omega$ and $\sigma \in
2^{<\omega}$. By the Recursion Theorem, we can assume that we have a
computable function $f : \omega \times \omega \times 2^{<\omega}
\rightarrow \omega$ such that $A$ has an initial segment in
$D_{e,k,\sigma}$ if and only if $A'(f(e,k,\sigma))=1$. Suppose that
$Q_e$ seems to require attention at stage $s$ as above. For each $\tau
\succcurlyeq \sigma$ of length $s-1$, pick the extension $\nu_\tau$
of $\Psi^\tau$ currently in $W_e^A$ that has been in that set the
longest, let $u_\tau$ be the use of the current enumeration of
$\nu_\tau$ into $W_e^A$, and let $u$ be the maximum of $u_\tau$ over
all $\tau \succcurlyeq \sigma$ of length $s-1$. Let $k$ be the number
of times $Q_e$ has been initialized. Put $A[s] \restr u$ into
$D_{e,k,\sigma}$, and search for a $t \geq s$ such that either
$A'(f(e,k,\sigma))[t]=1$ or $A[t] \restr u \neq A[s] \restr u$. Such a
$t$ must exist. (Notice that, in the second case, $A \restr u \neq
A[s] \restr u$, since $A$ is c.e.) In the first case, $Q_e$ actually
requires attention at stage $s$. As before, we choose the least $e$
such that $Q_e$ requires attention at stage $s$ and act for it as in
the proof of Theorem \ref{cethm}, making sure that we use the strings
$\nu_\tau$ mentioned above. If later we find that we were mistaken,
i.e., that $A'(f(e,k,\sigma))[u]=0$ for some $u>t$, then we restart
$Q_e$ as if it had been initialized, though we do not count this as an
initialization (so that we keep working with the same sets
$D_{e,k,\sigma}$).

This restarting process can happen only finitely often between
initializations, as each occurrence requires a change in the
approximation to $A'(f(e,k,\sigma))$ for a fixed $k$ and one of the
finitely many strings $\sigma$ of length $2^{-(e+2)}$. Thus the proof
that each strategy is initialized only finitely often remains the same
as before. 

We can then argue that each $Q_e$ is satisfied more or less as before:
Let $k$ be the total number of times that $Q_e$ is initialized. If
$Q_e$ ever permanently claims a string $\sigma$, then it eventually
stops requiring attention, so there is at least one extension $\tau$
of $\sigma$ such that whenever an extension of $\tau$ is in $W_e^A[s]$
with use $u$, there is a $t \geq s$ such that $A[t] \restr u \neq A[s]
\restr u$. Since $A$ is c.e., $\tau$ has no extension in $W_e^A$, so
$W_e^A$ is not dense. Otherwise, for each string $\sigma$ of length
$e+2$ such that no initial segment of $\sigma$ is eventually
permanently claimed by some $Q_i$, the strategy for $Q_e$ eventually
requires attention at a stage at which it is claiming $\sigma$, such
that it never gets initialized or restarted after that stage. Thus
this strategy ensures that if $X$ extends $\sigma$ then $\Psi^X$ has
an initial segment in $W_e^A$. So if $\Psi^X$ is total then it meets
$W_e^A$.
\end{proof}

Moving away from c.e.\ sets, we have the following result. Note that
it does not imply any of our previous results, because a weakly
$2$-generic set cannot compute any noncomputable c.e.\ sets.

\begin{theorem}
\label{w2gthm}
Let $A$ be weakly $2$-generic. Then almost every set computes a set
that is weakly $1$-generic relative to $A$, and hence $A$ is
dispersive.

In fact, almost every set computes a set that is weakly $1$-generic
relative to every weakly $2$-generic set.
\end{theorem}

\begin{proof}
By Kolmogorov's $0$-$1$ Law, it is enough to show that there are
positive-measure many sets that compute a set that is weakly
$1$-generic relative to every weakly $2$-generic set. We do this by
defining a Turing functional $\Psi$ such that $\Psi^X$ is total for
positive-measure many $X$, and such that if $\Psi^X$ is total then it
is weakly $1$-generic relative to every weakly $2$-generic set.

The construction of $\Psi$ is once again based on the one in the proof
of Theorem \ref{cethm}. We no longer have any $R$-requirements, but
now have requirements
\[
Q_e :  A \textrm{ weakly $2$-generic, } W_e^A \textrm{ dense, and }
\Psi^X \textrm{ total} \ \Rightarrow \ \Psi^X \textrm{ meets } W_e^A.
\]
For each $Q_e$, we will have  infinitely many strategies $Q_e^\alpha$,
one for each binary string $\alpha$. Associated with each $Q_e^\alpha$
will be a string $\beta_e^\alpha \succcurlyeq \alpha$, whose value might
change during the construction. This string will be such that, for the
final value of $\beta_e^\alpha$, we will have ensured that $Q_e$ holds
as long as $A$ extends $\beta_e^\alpha$. As this value will be
computably approximated, and the set of $\beta_e^\alpha$ will be dense,
if $A$ is weakly $2$-generic then it will extend some $\beta_e^\alpha$.

Let $g : \omega \times 2^{<\omega} \rightarrow \omega$ be a computable
injective function. We arrange the $Q_e^\alpha$'s into a priority list
using $g$, declaring that $Q_e^\alpha$ is stronger that $Q_i^\beta$ if
$g(e,\alpha)<g(i,\beta)$. The string $\beta_e^\alpha$ is initially
equal to $\alpha$, and is again set to $\alpha$ every time
$Q_e^\alpha$ is initialized. At each stage, $Q_e^\alpha$ might be
claiming a string. Initially, no $Q_e^\alpha$ claims a string. When
$Q_e^\alpha$ is initialized, it gives up its current claim if it has
one. As before, for two strings $\sigma$ and $\tau$ of the same
length, let $\sigma <\sub{l} \tau$ if $\sigma$ comes before $\tau$ in
the lexicographic ordering.

At stage $s$, say that $Q_e^\alpha$ with $g(e,\alpha) < s-2$
\emph{requires attention} if $Q_e^\alpha$ is not currently declared to
be satisfied, and either $Q_e^\alpha$ is not claiming any string or it
is claiming a string $\sigma$, and there is a $\gamma \succcurlyeq
\beta_e^\alpha$ of length $s$ such for every $\tau \succcurlyeq
\sigma$ of length $s-1$, there is an extension $\nu$ of $\Psi^\tau$ in
$W_e^\gamma[s]$. For the least value of $g(e,\alpha)$ such that
$Q_e^\alpha$ requires attention (if any), we act as follows.

If $Q_e^\alpha$ is not claiming any string, then it claims the
$<\sub{l}$-least string $\tau$ of length $g(e,\alpha)+2$ such that no
initial segment of $\tau$ is currently being claimed by any
$Q_i^\delta$ with $g(i,\delta)<g(e,\alpha)$. Such a $\tau$ exists by
the same argument used to show the existence of the analogous string
$\tau$ in the proof of Theorem \ref{cethm2}.

Otherwise, proceed as follows. For each $\tau \succcurlyeq \sigma$ of
length $s-1$, let $\nu$ be an extension of $\Psi^\tau$ in
$W_e^\gamma[s]$ and define
$\Psi^{\tau^\frown0}=\Psi^{\tau^\frown1}=\nu$. Now $Q_e^\alpha$ drops
its claim on $\sigma$ and claims the next $<\sub{l}$-least string
$\tau >\sub{l} \sigma$ of length $g(e,\alpha)+2$ such that no initial
segment of $\tau$ is currently being claimed by any $Q_i^\delta$ with
$g(i,\delta)<g(e,\alpha)$. If there is no such $\tau$, then declare
$Q_e^\alpha$ to be satisfied. In either case, redefine
$\beta_\alpha=\gamma$.

In any case, initialize every $Q_i^\delta$ with
$g(i,\delta)>g(e,\alpha)$.

Finally, for each $\mu$ of length $s$ such that $\Psi^\mu$ is not yet
defined, if some initial segment of $\mu$ is currently claimed by a
strategy then define $\Psi^{\mu}=\Psi^{\mu \restr (s-1)}$, and
otherwise define $\Psi^{\mu}=\Psi^{\mu \restr (s-1)}{}^\frown0$.

As before, if no initial segment of $X$ is ever permanently claimed,
then there are infinitely many stages $s$ such that no initial segment
of $X$ is claimed at the end of stage $s$, and hence $\Psi^X$ is
total, so $\mu(\{X : \Psi^X \textrm{ total}\}) \geq 1 - \sum_n
2^{-(n+2)} = 1/2$. Also as before, by induction, each strategy is
initialized only finitely often.

Fix $e$ and a weakly $2$-generic $A$. Writing $\beta_\alpha$ for the
final value of that string, $\{\beta_\alpha : \alpha \in
2^{<\omega}\}$ is dense and $\emptyset'$-c.e. (In fact, it is
$\emptyset'$-computable.) Thus there is an $\alpha$ such that
$\beta_\alpha \prec A$. If $Q_e^\alpha$ ever permanently claims a
string, then $W_e^X$ is not dense for any $X$ extending
$\beta_\alpha$, so in particular $W_e^A$ is not dense. Otherwise, the
construction ensures that for each string $\sigma$ of length
$g(e,\alpha)+2$ such that no initial segment of $\sigma$ is eventually
permanently claimed by some strategy, if $X$ extends
$\sigma$ then $\Psi^X$ has an initial segment in $W_e^{\beta_\alpha}$,
and hence in $W_e^A$. So if $\Psi^X$ is total then it meets $W_e^A$.
\end{proof}

We do not know whether every $1$-generic set is dispersive.

\begin{corollary}
The class of attractive sets is meager, i.e.,
\[
(\cm A)(\aev B)[H(A,B)=1].
\]
\end{corollary}

The proof of Theorem \ref{w2gthm} yields the following more precise
version of the above corollary.

\begin{theorem}
If $A$ is weakly $2$-generic and $B$ is $2$-random, then $B$ computes
a set that is weakly $1$-generic relative to $A$, and hence $H(A,B)=1$.
\end{theorem}

\begin{proof}
The proof of Theorem \ref{w2gthm} can easily be modified to show that
for each $i$ there is a Turing functional $\Psi_i$ that meets all the
requirements $Q_e$ with $\Psi_i$ replacing $\Psi$, and has the further
property that $\mu(\{X : \Psi_i^X \textrm{ total}\}) \geq 1 - 2^{-i}$.
Furthermore, an index for such a $\Psi_i$ can be obtained effectively
from $i$. Let $\mathcal{C}_i = \{X : \Psi_i^X \textrm{ not
  total}\}$. Then $\mathcal C_0,\mathcal C_1,\ldots$ are uniformly
$\Sigma^0_2$, and hence form a $\Sigma^0_2$ Martin-L\"of test. Since
$B$ is $2$-random, we can fix $i$ with $B \notin \mathcal{C}_i$.
Since all $Q_e$ are satisfied for $\Psi = \Psi_i$, the conclusion of
the theorem follows.
\end{proof}

Notice that we cannot improve the above theorem to all weakly
$2$-random $B$, because there are weakly $2$-randoms that have
hyperimmune-free degree (see e.g. \cite[Theorem 8.11.12]{DH}), and
hence do not compute any weakly $1$-generics. It would be interesting
to have an exact characterization of the attractive degrees that does
not mention Hausdorff distance. It would also be interesting to know
whether $A$ can be dispersive without it being the case that almost
every set computes a set that is weakly $1$-generic relative to
$A$. Finally, we do not know whether the above theorem can be improved
to hold of all $1$-generic $A$.

\begin{observation}
\label{aedobs}
A related notion that has been completely characterized is almost
everywhere domination. Dobrinen and Simpson \cite{DS} defined a set
$A$ to be \emph{almost everywhere dominating} if for almost every $B$,
every $B$-computable function is dominated by an $A$-computable
function. Binns, Kjos-Hanssen, Lerman, and Solomon \cite{BKLS} and
Kjos-Hanssen, Miller, and Solomon \cite{KMS} characterized the almost
everywhere dominating sets as those sets $A$ such that every set that
is $1$-random relative to $A$ is $2$-random. If almost every set
computes a set that is weakly $1$-generic relative to $A$, then $A$
cannot be almost everywhere dominating (because if a set is weakly
$1$-generic relative to $A$, then it computes a function that is not
dominated by any $A$-computable function).

We can strengthen this fact by noting that the proof in
\cite[Corollary 1.13(i)]{ACDJL} that if $B$ is $1$-random and has
hyperimmune-free degree then $\Gamma(B)=1/2$ relativizes to show that
if $B$ is $1$-random relative to $A$ and every $B$-computable function
is dominated by an $A$-computable function, then $\Gamma_A(B)=1/2$. Of
course, in this case we also have $\Gamma_B(A)=1/2$, since $B$ is
$1$-random relative to $A$, and thus $H(A,B)=1/2$. It follows that if
$A$ is almost everywhere dominating, then it is attractive. The other
direction does not hold, however, as there are $1$-random sets that
are not almost everywhere dominating, for instance any $1$-random set
of hyperimmune-free degree, since if $A$ has hyperimmune-free degree
and every $B$-computable function is dominated by an $A$-computable
function, then $B$ also has hyperimmune-free degree.  It is possible,
however, that for c.e.\ sets the two notions coincide. As explained
below Corollary 11.2.7 in \cite{DH}, there are incomplete c.e.\ sets
that are almost everywhere dominating. It follows that there are
attractive sets that do not compute any $1$-randoms.  We do not know
whether every attractive set of hyperimmune-free degree computes a
$1$-random.
\end{observation}

Much of the work in this section has been devoted to determining the
truth of sentences of the form
\[
(Q_1 A)(Q_2 B)[H(A, B) = r],
\]
where each $Q_i$ is $\cm$ or $\aev$, and $r$ is $1/2$ or $1$.  We
pause to summarize the results of this form.  For $r = 1/2$, the above
statement is true if and only if both $Q_1$ and $Q_2$ are $\aev$.  For
$r = 1$ the above statement is true if and only if either $Q_1$ or
$Q_2$ is $\cm$.

\section{Isometric embeddings into the Turing Degrees}
\label{isosec}

The motivation for the results in this section is the question of
which finite metric spaces with every distance equal to
$0$, $1/2$, or $1$ are isometrically embeddable in $(\mathcal
D,H)$. We begin with a partial answer to this question.

By a graph we will mean an undirected graph with no loops. For a
metric space $\mathcal M$ with every distance equal to $0$, $1/2$, or
$1$, let $G_{\mathcal M}$ be the graph whose vertices are the points
of $\mathcal M$, with an edge between $x$ and $y$ if and only if the
distance between $x$ and $y$ is $1$. We denote the complement of this
graph, where there is an edge between $x$ and $y$ if and only if the
distance between $x$ and $y$ is $1/2$, by $G^{\textup{c}}_{\mathcal
M}$. We write $G_{\mathcal D}$ for $G_{(\mathcal D,H)}$. Notice that
every graph is $G_{\mathcal M}$ for some $0,1/2,1$-valued metric space
$\mathcal M$.

A graph $(V,E)$ is a \emph{comparability graph} if there is a partial
order $(V,\preccurlyeq)$ such that $E(x,y)$ if and only if $x \prec y$
or $y \prec x$.

\begin{theorem}
Let $\mathcal M$ be a countable metric space with every distance equal
to $0$, $1/2$, or $1$, such that $G_{\mathcal M}$ is a comparability
graph. Then $\mathcal M$ is isometrically embeddable in $(\mathcal
D,H)$.
\end{theorem}

\begin{proof}
There is a computable partial ordering $(\omega,\preccurlyeq)$ that is
countably universal, i.e., every countable partial ordering is
order-isomorphic to a subordering of $(\omega,\preccurlyeq)$. (See
e.g.\ \cite[Exercise 6.15]{Snew}.) It is enough to show that there are
pairwise distinct degrees $\mathbf{a}_i$ such that
$H(\mathbf{a}_i,\mathbf{a}_j)=1$ if and only if $i \neq j$ and $i$ and
$j$ are $\preccurlyeq$-comparable.

Let $\bigoplus_{n \in \omega} B_n$ be a $\Delta^0_2$ $1$-random. By
van Lambalgen's Theorem, each $B_n$ is $1$-random relative to
$\bigoplus_{m \neq n} B_m$. Let $A_n = \bigoplus_{i \preccurlyeq n}
B_i$ and let $\mathbf{a}_n$ be the degree of $A_n$.

If $i \prec j$ then $A_i <\sub{T} A_j$. Furthermore, $A_j$ is
$\Delta^0_2$, so $A_j$ has hyperimmune degree relative to $A_i$, by
relativizing the result that nonzero $\Delta^0_2$ degrees are
hyperimmune.  Since $A_i <\sub{T} A_j$ and $A_j$ has hyperimmune
degree relative to $A_i$, it follows that $A_j$ computes a weakly
$1$-generic relative to $A_i$, and hence, by Theorem \ref{genthm},
$H(\mathbf{a}_i,\mathbf{a}_j)=1$.

Now suppose that $i$ and $j$ are $\preccurlyeq$-incomparable. Then
$B_i$ is $1$-random relative to $\bigoplus_{k \preccurlyeq j} B_k$,
since the latter set is computable from $\bigoplus_{k \neq i} B_k$. By
Corollary \ref{randcor}, $\Gamma_{\mathbf{a}_i}(\mathbf{a}_j) \geq
\Gamma_{B_i}(\mathbf{a}_j) = 1/2$, so in fact
$\Gamma_{\mathbf{a}_i}(\mathbf{a}_j) = 1/2$. The symmetric argument
shows that $\Gamma_{\mathbf{a}_j}(\mathbf{a}_i) = 1/2$. Thus
$H(\mathbf{a}_i,\mathbf{a}_j)=1/2$.
\end{proof}

We do not know what other countable $0,1/2,1$-valued metric spaces (if
any) are isometrically embeddable in $(\mathcal D,H)$. An interesting
test case is the finite $0,1/2,1$-valued metric space $\mathcal M$
such that $G_{\mathcal M}$ is a cycle of length $5$, which is the
simplest example of a graph that is not a comparability graph.
Answering this question will likely require better knowledge of the
possible distances between elements of sets of pairwise incomparable
degrees. A first step in that direction is to show that there are
pairwise incomparable degrees $\mathbf{a},\mathbf{b},\mathbf{c}$ such
that $H(\mathbf{a},\mathbf{b})=H(\mathbf{b},\mathbf{c})=1/2$ and
$H(\mathbf{a},\mathbf{c})=1$. The following result will do so, and
indeed ensure that the degrees $\mathbf{a},\mathbf{b},\mathbf{c}$ are
$1$-random, which might be useful in obtaining further embeddings of
metric spaces into $(\mathcal D,H)$.

\begin{theorem}
\label{incrandthm}
There are incomparable $1$-random degrees $\mathbf{a},\mathbf{c}$ such
that $H(\mathbf{a},\mathbf{c})=1$.
\end{theorem}

\begin{proof}
By a theorem of Ku{\v c}era \cite{Ku}, every degree $\geq \mathbf{0'}$
is $1$-random. Hence it suffices to show that there are incomparable
degrees $\mathbf{a},\mathbf{c} \geq \mathbf{0'}$ such that
$H(\mathbf{a},\mathbf{c}) = 1$. This can be done via what is
essentially a relativization of a proof that there exist
incomparable degrees $\mathbf{a},\mathbf{c}$ such that $H(\mathbf{a},
\mathbf{c}) = 1$. That is, let $G_0 \oplus G_1$ be $2$-generic, let
$\mathbf{a}$ be the degree of $\emptyset' \oplus G_0$, and let
$\mathbf{c}$ be the degree of $\emptyset' \oplus G_1$. Then
$\mathbf{a}$ and $\mathbf{c}$ are incomparable, since $G_0 \oplus G_1$
is $1$-generic relative to $\emptyset'$. We also have
$\Gamma_{\mathbf{a}}(\mathbf{c}) \leq \gamma_{\mathbf{a}}(G_1) = 0$,
by the relativization of Theorem \ref{genthm} to $\emptyset'$.  It
follows that $H(\mathbf{a}, \mathbf{c}) = 1$.
\end{proof}

\begin{corollary}
There are pairwise incomparable $1$-random degrees
$\mathbf{a},\mathbf{b},\mathbf{c}$ such that
$H(\mathbf{a},\mathbf{b})=H(\mathbf{b},\mathbf{c})=1/2$ and
$H(\mathbf{a},\mathbf{c})=1$.
\end{corollary}

\begin{proof}
Let $\mathbf{a}$ and $\mathbf{c}$ be as in the theorem, and let
$\mathbf{b}$ be $1$-random relative to $\mathbf{a \vee c}$, so that by
van Lambalgen's Theorem, $\mathbf{a}$ and $\mathbf{b}$ are relatively
$1$-random, as are $\mathbf{b}$ and $\mathbf{c}$. Then
$H(\mathbf{a},\mathbf{b})=H(\mathbf{b},\mathbf{c})=1/2$ by Corollary
\ref{randcor}.
\end{proof}

The proof of Theorem \ref{incrandthm} does not seem very flexible,
relying as it does on $1$-random degrees above $\mathbf{0'}$, which
are atypical in several ways. It also does not relativize to show that
$\mathbf{a},\mathbf{c}$ can be chosen to have higher levels of
algorithmic randomness. Thus we give the following alternate proof,
which establishes the result in relativized form.

\begin{theorem}
Let $X$ be any oracle. There are incomparable degrees
$\mathbf{a},\mathbf{c}$ that are $1$-random relative to $X$ and such
that $H(\mathbf{a},\mathbf{c})=1$.
\end{theorem}

\begin{proof}
Recall the notion of almost everywhere dominating sets from
Observation \ref{aedobs}. There are several ways to see that the class
of such sets has measure $0$, for instance because every such set is
high (see \cite[Section 10.6]{DH}), and the class of high sets has
measure $0$ (see the proof of \cite[Lemma 11.8.7]{DH})). Thus there is
a set $D$ that is $1$-random relative to $X$ and is not almost
everywhere dominating. Then there is a set $E$ that is $1$-random
relative to $X \oplus D$ and computes a function that is not dominated
by any $D$-computable function. Let $C = D \oplus E$. Then $C$ is
$1$-random relative to $X$ by van Lambalgen's Theorem relative to $X$.
Furthermore, $D \leq\sub{T} C$ and $C$ has hyperimmune degree relative
to $D$, so $C$ computes a set $G$ that is weakly $1$-generic relative
to $D$.

Let $\mathcal P$ be a nonempty $\Pi^{0,X}_1$ class of sets that are
$1$-random relative to $X$. We build a set $A$ by forcing with
nonempty $\Pi^{0,X}_1$ subclasses of $\mathcal P$. Since $C
\nleq\sub{T} X$, any set that is
sufficiently generic for this notion of forcing is Turing incomparable
with $C$. (See for instance \cite[Section 4]{DDS}.)

We claim that if $A$ is sufficiently generic for this notion of
forcing, then $G$ is weakly $1$-generic relative to $A$. Then we can
take $\mathbf{a}$ and $\mathbf{c}$ to be the degrees of $A$ and $C$,
respectively, and these will be incomparable degrees that are
$1$-random relative to $X$ and such
that $H(\mathbf{a},\mathbf{c})=1$, by Theorem \ref{genthm}. The key
here will be the result of Ku{\v c}era \cite{Ku} (in relativized form)
that if $Y$ is $1$-random relative to $X$ and $\mathcal Q$ is a
$\Pi^{0,X}_1$ class of positive measure, then $\mathcal Q$ contains an
element of the same degree as $Y$.

Thinking of c.e.\ operators as enumerating sets of binary strings, it
is enough to show that for each nonempty $\Pi^{0,X}_1$ subclass $\mathcal
Q$ of $\mathcal P$ and each $e$, there is a nonempty
$\Pi^{0,X}_1$ subclass $\mathcal R$ of $\mathcal Q$ such that either for
each $Z \in \mathcal R$, the set $W_e^Z$ is not dense, or for each $Z
\in \mathcal R$, there is an initial segment of $G$ in $W_e^Z$.

For each binary string $\sigma$, consider the $\Pi^{0,X}_1$ subclass
of $\mathcal Q$ consisting of all $Z \in \mathcal Q$ such that $W_e^Z$
does not contain an extension of $\sigma$. If any of these classes is
nonempty, we can take it to be $\mathcal R$. Otherwise, $W_e^Z$ is
dense for all $Z \in \mathcal Q$. Since $\mathcal Q$ is a nonempty
$\Pi^{0,X}_1$ class of sets that are $1$-random relative to $X$, it
has positive measure, and hence contains a set $B$ of the same degree
as $D$. Then $W_e^B$ is dense and $D$-c.e., so it contains an initial
segment $\rho$ of $G$. Let $\tau$ be an initial segment of $B$ such
that $\rho \in W_e^\tau$. Then we can take $\mathcal R$ to be the
restriction of $\mathcal Q$ to extensions of $\tau$.
\end{proof}

\begin{corollary}
For any $n$, there are pairwise incomparable $n$-random degrees
$\mathbf{a},\mathbf{b},\mathbf{c}$ such that
$H(\mathbf{a},\mathbf{b})=H(\mathbf{b},\mathbf{c})=1/2$ and
$H(\mathbf{a},\mathbf{c})=1$.
\end{corollary}

The graph $G_{\mathcal D}$ is connected and has diameter $2$, since
for any degrees $\mathbf{a}$ and $\mathbf{c}$, there is a degree
$\mathbf{b}$ that is weakly $1$-generic relative to $\mathbf{a \vee
c}$, and then
$H(\mathbf{a},\mathbf{b})=H(\mathbf{b},\mathbf{c})=1$. The graph
$G^{\textup{c}}_{\mathcal D}$ is also connected, as we now show, but
its diameter is more difficult to determine.

\begin{theorem}
\label{diam4}
The graph $G^{\textup{c}}_{\mathcal D}$ is connected and has diameter
at most $4$.
\end{theorem}

\begin{proof}
Let $\mathbf{a}$ and $\mathbf{b}$ be any two degrees. By Corollary
\ref{excor}, there are $1$-random degrees $\mathbf{c}$ and
$\mathbf{d}$ such that $H(\mathbf{a},\mathbf{c})=1/2$ and
$H(\mathbf{d},\mathbf{b})=1/2$. Now let $\mathbf{e}$ be $1$-random
relative to $\mathbf{a} \vee \mathbf{b} \vee \mathbf{c} \vee
\mathbf{d}$. By van Lambalgen's Theorem, $\mathbf{c}$ and $\mathbf{e}$
are relatively $1$-random, as are $\mathbf{d}$ and $\mathbf{e}$. So by
Corollary \ref{randcor},
$H(\mathbf{c},\mathbf{e})=H(\mathbf{e},\mathbf{d})=1/2$. Thus we
conclude that
$H(\mathbf{a},\mathbf{c})=H(\mathbf{c},\mathbf{e})=H(\mathbf{e},\mathbf{d})=H(\mathbf{d},\mathbf{b})=1/2$.
\end{proof}

We now show that the diameter of $G_{\mathcal D}$ is at least $3$, via
a couple of results of independent interest. To prove them we use a
lemma due to Ng, Stephan, Yang, and Yu \cite{NSYY}. We include a proof
since one does not appear in \cite{NSYY}.

\begin{lemma}[Ng, Stephan, Yang, and Yu \cite{NSYY}]
\label{nsyylem}
If $\mathcal P$ is a $\Pi^{0,\emptyset'}_1$ class with a member $A$ of
hyperimmune-free degree, then $\mathcal P$ has a $\Pi^0_1$ subclass
containing $A$.
\end{lemma}

\begin{proof}
Let $T$ be a $\Delta^0_2$ binary tree whose paths are exactly the
elements of $\mathcal P$, and let $T[0],T[1],\ldots$ be uniformly
computable binary trees approximating $T$. Let $f(n)$ be the least $k
\geq n$ such that $A \upharpoonright n$ is in $T[k]$. Then $f
\leq\sub{T} A$, so $f$ is majorized by some computable function
$g$. Let $Q$ consist of all binary strings $\sigma$ such that every
predecessor of $\sigma$ is in $Q$ and $\sigma \in T[k]$ for some $k
\in [|\sigma|,g(|\sigma|)]$. Then $Q$ is a computable
tree. Furthermore, every $A \upharpoonright n$ is in $Q$, since $f(n)
\in [n,g(n)]$, so $A$ is a path on $Q$. Finally, if $X$ is not a path
on $T$, then there is an $n$ such that $X \upharpoonright n \notin
T[k]$ for all $k \geq n$. Then $X \upharpoonright n \notin Q$, so $X$
is not a path on $Q$. Thus the class of paths on $Q$ is our desired
$\Pi^0_1$ subclass of $\mathcal P$.
\end{proof}

\begin{theorem}
\label{pathm}
If $\mathbf{a}$ is hyperimmune-free and $\mathbf{b}$ is a $\Delta^0_2$
PA degree, then $\Gamma_{\mathbf{a}}(\mathbf{b})=0$, and hence
$H(\mathbf{a},\mathbf{b})=1$.
\end{theorem}

\begin{proof}
We will need two properties of PA degrees. One is the relativized
version of the theorem due to Jockusch \cite[Proposition 4]{J72} that
if $\mathbf{b}$ is PA then there is a uniformly
$\mathbf{b}$-computable sequence of sets that includes all computable
sets. The second, due to Simpson \cite[Theorem 6.5]{Ssurv}, is that if
$\mathbf{b}$ is PA then there is another PA degree $\mathbf{c}$ such
that $\mathbf{b}$ is PA relative to $\mathbf{c}$.

Let $\mathbf{a}$ be hyperimmune-free and let $\mathbf{b}$ be a
$\Delta^0_2$ PA degree.  We begin with the following claim. Let $g$ be
a computable function. Then there are uniformly
$\mathbf{b}$-computable sets of strings $S_0,S_1,\ldots$ such that
$|S_n| = n$ for all $n$, every element of $S_n$ has length $g(n)$ for
all $n$, and for every $A \leq\sub{T} \mathbf{a}$, we have $A
\upharpoonright g(n) \in S_n$ for infinitely many $n$.

To establish the claim, suppose not. Let $\mathbf{c}$ be a PA degree
such that $\mathbf{b}$ is PA relative to $\mathbf{c}$. Let
$E_0,E_1,\ldots$ be uniformly $\mathbf{b}$-computable sets such that
every $\mathbf{c}$-computable set is on this list. For each $n$, let
$S_n=\{E_i \upharpoonright g(n) : i<n\}$. Then $S_0,S_1,\ldots$ are
uniformly $\mathbf{b}$-computable, and $|S_n| = n$ for all $n$, so
there is an $A \leq\sub{T} \mathbf{a}$ and an $m$ such that if $n>m$
then $A \upharpoonright g(n) \notin S_n$. Let $\mathcal P=\{X :
(\forall n>m) [X \upharpoonright g(n) \notin S_n]\}$. Then, since
$\mathbf{b} \leq \mathbf{0'}$, we have that $\mathcal P$ is a
$\Pi^{0,\mathbf{0'}}_1$ class containing $A$. By Lemma \ref{nsyylem},
$\mathcal P$ has a nonempty $\Pi^0_1$ subclass $\mathcal Q$. Since
$\mathbf{c}$ is PA, $\mathcal Q$ contains a $\mathbf{c}$-computable
set, which is equal to $E_i$ for some $i$. Let $n>i,m$. Then $E_i
\upharpoonright g(n) \in S_n$, contradicting the fact that $E_i \in
\mathcal Q \subseteq \mathcal P$. Thus we have established the claim.

The idea now is that by choosing $g$ to be sufficiently fast growing,
we can build a set $B \leq\sub{T} \mathbf{b}$ to diagonalize against
each element of each $S_n$ on a large segment, and thus in particular
to diagonalize against every $A \leq\sub{T} \mathbf{a}$ on infinitely
many large segments, ensuring that, for every such $A$, the upper
density of $A \sd B$ is equal to $1$, so that $\gamma_{\mathbf{a}}(B)
= 0$.  It then follows that $\Gamma_{\mathbf{a}}(\mathbf{b})=0$, and
hence $H(\mathbf{a},\mathbf{b})=1$.

Let $F_0,F_1,\ldots$ be consecutive segments of $\omega$ with
$|F_n|=n+1$. Let $I_k=[k! , (k+1)!)$. Let $g(n)= (k+1)!$ for the
largest $k \in F_n$. Apply the claim to this $g$ to obtain sets
$S_0,S_1,\ldots$ as above. Note that these sets are pairwise disjoint
since $g$ is injective. Assign each $\sigma \in S_n$ to a $k_\sigma
\in F_n$, so that $k_\sigma \neq k_\tau$ for $\sigma \neq \tau \in
S_n$. This is possible since $|F_n|=|S_n|+1$ and the sets $S_n$ are
pairwise disjoint.  Furthermore, this assignment can be made
computably in $\mathbf{b}$.  Define $B \leq\sub{T} \mathbf{b}$ as
follows.  For $i \in I_{k_\sigma}$, let $B(i)=1-\sigma(i)$. Now for
each of the infinitely many $n$ such that $A \restr g(n) \in S_n$,
letting $k = k_{A \restr g(n)}$, we have that $k \in F_n$, and by
definition $B(i)=1-A(i)$ for all $i \in I_k$. Hence $\rho_{(k+1)!} (A
\sa B) \leq k! / (k+1)! = 1 / (k + 1)$. Since there are arbitrarily
large such $k$, it follows that $\urho(A \sa B) = 0$.  Since $A$ was
an arbitrary $\mathbf{a}$-computable set, we have that
$\gamma_{\mathbf{a}}(B) = 0$.  It then follows as above that
$H(\mathbf{a},\mathbf{b})=1$.  
\end{proof}

Notice that this theorem gives us yet another proof of Theorem
\ref{incrandthm}, by considering a hyperimmune-free $1$-random degree
and $\mathbf{0}'$ (which is PA and $1$-random).

It would be interesting to know how far the above theorem can be
extended, and in particular whether it holds for all hyperimmune
PA degrees.

\begin{corollary}
\label{pacor}
There is a degree $\mathbf{b}$ such that for all degrees $\mathbf{a}$,
if $H(\mathbf{0},\mathbf{a})=1/2$ then $H(\mathbf{a},\mathbf{b})=1$.
\end{corollary}

\begin{proof}
Let $\mathbf{b}$ be a $\Delta^0_2$  PA degree. By \cite[Theorem
2.2]{HJMS}, if $H(\mathbf{0},\mathbf{a})=1/2$, then $\mathbf{a}$ is
hyperimmune-free, so by Theorem \ref{pathm},
$H(\mathbf{a},\mathbf{b})=1$. Thus there is no degree $\mathbf{a}$
such that $H(\mathbf{0},\mathbf{a})=H(\mathbf{a},\mathbf{b})=1/2$.
\end{proof}

This result has implications for the issue of extensions of isometric
embeddings of metric spaces into $(\mathcal D,H)$, which we will not
pursue further here.  Let $\mathcal M$ be the metric space with two
points $x$ and $y$ at distance $1$ from each other, and let $\mathcal
M'$ be the extension of $\mathcal M$ obtained by adding a point $z$
such that the distances between $x$ and $z$ and between $y$ and $z$
are both $1/2$. Notice that $\mathcal M'$ is isometrically embeddable
into $(\mathcal D,H)$. Let $\mathbf{b}$ be as in the proof of
Corollary \ref{pacor}. Then the isometric embedding of $\mathcal M$
into $(\mathcal D,H)$ obtained by mapping $x$ to $\mathbf{0}$ and $y$
to $\mathbf{b}$ cannot be extended to an isometric embedding of
$\mathcal M'$ into $(\mathcal D,H)$.

\begin{corollary}
\label{diam3}
The diameter of $G^{\textup{c}}_{\mathcal D}$ is at least $3$.
\end{corollary}

We do not know whether the diameter of $G^{\textup{c}}_{\mathcal D}$
is $3$ or $4$.

\section{Mycielski's Theorem, computability, and reverse mathematics}
\label{mycsec}

In this section, we analyze Mycielski's Theorems \ref{Ramsey} and
\ref{Ramseyforcat} and their consequences from the points of view of
computability theory and reverse mathematics. To talk about perfect
sets in this context, we use perfect trees, as defined in Definition
\ref{ptdef}.  We can think of a perfect tree $T$ as a $1$-$1$ function
from $2^\omega$ to $2^\omega$, so there are continuum many paths
through $T$. Indeed, the paths through $T$ form a perfect set. Notice
that for every $A$ we have $A \oplus T \equiv\sub{T} T(A) \oplus
T$. An equivalent way to think of a perfect tree is as a binary tree
(in the usual sense) that has no dead ends and no isolated paths.

\subsection{Mycielski's Theorem for measure and computability theory}
\label{mycsec1}

We begin by effectivizing Corollary \ref{perfmutcor}.  The following
notions will be useful.

\begin{definition}
For a binary string $\sigma$, let $[\sigma] = \{X \in 2^\omega :
\sigma \prec X\}$. For a measurable class $\mathcal C$ and $X \in
\mathcal C$, let
\[
d(X \mid \mathcal C) = \textstyle{\liminf_n}\, 2^n\mu([X \uhr n] \cap
\mathcal C)
\]
be the \emph{density of  $\mathcal C$ near $X$}. 

A set $X$ is a \emph{density-one point} if $d(X \mid \mathcal P)=1$
for all $\Pi^0_1$ classes $\mathcal P$ containing $X$. 

If $\sigma$ is a string, the \emph{relative measure of $\mathcal{C}$ above
$\sigma$} is $2^{|\sigma|} \mu([\sigma] \cap \mathcal{C})$.
\end{definition}

For example, every $1$-generic set is a density-one point, since it
lies in the interior of every $\Pi^0_1$ class it belongs to.

The following lemma will be useful below.

\begin{lemma}
\label{measlem}
Suppose that $d(X \mid \mathcal C)=1$. Let $c \in \omega$ and
$\delta>0$. For all sufficiently large $t$, if $\tau \succ X \uhr t$
has length $t+c$, then $\mu([\tau] \cap \mathcal C) >
(1-\delta)2^{-|\tau|}$.
\end{lemma}

\begin{proof}
Let $\mathcal D = 2^\omega \setminus \mathcal C$.  Since $[\tau] \cap
\mathcal D \subseteq [X \uhr t] \cap \mathcal D$, the relative measure of
$\mathcal D$ above $\tau$ is at most $2^c$ times the relative measure
of $\mathcal D$ above $X \uhr t$, and the latter relative measure goes
to $0$ as $t$ increases.
\end{proof}

We write $\bigoplus_{i  \leq n} X_i$ for the set given by the sequence
\[
X_0(0) \cdots X_n(0) X_0(1)\cdots X_n(1) X_0(2)\cdots
X_n(2) \cdots,
\] 
and for strings $\sigma_0,\ldots,\sigma_n$ of the same length $k$, we
write $\bigoplus_{i \leq n} \sigma_i$ for the string
\[
\sigma_0(0) \cdots \sigma_n(0) \sigma_0(1) \cdots \sigma_n(1) \cdots
\sigma_0(k-1) \cdots \sigma_n(k-1).
\]

The following basic properties are easy to check.

\begin{lemma}
\label{baslem}
Let $\bigoplus_{i < m} X_i$ be a density-one point.
\begin{enumerate}

\item For any pairwise distinct $i_0,\ldots,i_k < m$, the set
$\bigoplus_{j<k} X_{i_j}$ is a density-one point.

\item Let $n>m$. If $\mathcal P$ is a $\Pi^0_1$ class containing
$\bigoplus_{i < m} X_i$ and $\mathcal C = \{\bigoplus_{i < n}
Y_i : \bigoplus_{i < m} Y_i \in \mathcal P\}$ then
$d(\bigoplus_{i < n} X_i \mid \mathcal C)=1$ for all
$X_{m+1},\ldots,X_n$.

\end{enumerate}
\end{lemma}

By the Lebesgue Density Theorem (see e.g.\ \cite[Theorem 1.2.3]{DH}),
for each measurable class $\mathcal C$, for almost every $X$, if $X$
is in $\mathcal C$ then $d(X \mid \mathcal C)=1$. Since there are only
countably many $\Pi^0_1$ classes, we have the following.

\begin{lemma}
\label{leblem}
There are measure-$1$ many density-one points, so there are
measure-$1$ many density-one $1$-random points.
\end{lemma}

We can also relativize this notion by saying that $X$ is a
\emph{density-one point relative to $A$} if $d(X \mid \mathcal P)=1$
for all $\Pi^{0,A}_1$ classes $\mathcal P$ containing $X$. The analogs
of the above properties continue to hold in this case.

\begin{theorem}
\label{perfeff}
For any $A$ there is an $A'$-computable perfect tree $T$ such that for
any nonempty finite collection $\mathcal F$ of paths through $T$, the
set $\bigoplus_{Y \in \mathcal F} Y$ is $1$-random relative to $A$,
and hence the joins of any two finite, disjoint, nonempty collections
of paths through $T$ are mutually $1$-random relative to $A$.
\end{theorem}

\begin{proof}
We do the proof for $A=\emptyset$, as the full proof is a
straightforward relativization. We begin by discussing the intuition
behind the proof, which is an effectivization of a proof of
Mycielski's Theorem along the lines discussed in \cite{Tod}.

We want to build a perfect tree $T$. Let us ignore for now the
complexity of $T$, as showing that $\emptyset'$ is sufficient to build
$T$ will not be difficult. For ease of exposition, let us first
discuss only how to make joins of two paths $1$-random, the general
case below being similar. The idea is to define for each $\sigma$ a
set $X_\sigma$ and for each $n$ a number $k_n$ so that for each
$\sigma \in 2^n$, we have $X_{\sigma^\frown i} \uhr k_n = X_\sigma
\uhr k_n$ for $i=0,1$ and $X_{\sigma^\frown 0} \uhr {k_{n+1}} \neq
X_{\sigma^\frown 1} \uhr {k_{n+1}}$. We will then define $T(\sigma) =
X_\sigma \uhr k_n$, so that each path through $T$ will be a limit of
$X_\sigma$'s (with respect to the usual metric on $2^\omega$).

Suppose we were just trying to make each individual path
$1$-random. The first idea might be to pick each $X_\sigma$ to be
$1$-random, but that is not enough because the limit of $1$-randoms
might not be $1$-random. So instead we can fix a $\Pi^0_1$ class of
$1$-randoms $\mathcal P_1$ and choose each $X_\sigma$ to be in
$\mathcal P_1$, which ensures that so is their limit.  Notice that if
$X_\sigma \in \mathcal P_1$ then $\mathcal P_1 \cap [X_\sigma \uhr k]$
has positive measure for any $k$, so $X_{\sigma^\frown 0}$ and
$X_{\sigma^\frown 1}$ can be defined.

To make joins of pairs of distinct paths $1$-random, we want to keep
all $X_\sigma \oplus X_\tau$, where $\sigma$ and $\tau$ are distinct
strings of the same length, inside some $\Pi^0_1$ class $\mathcal
P_2$. This class cannot consist entirely of $1$-randoms, because
$X_\sigma$ and $X_\tau$ can be arbitrarily close, so to be closed,
$\mathcal P_2$ must include elements of the form $X \oplus X$. But we
can define $\mathcal P_2$ so that every element $X \oplus Y$ is
$1$-random unless $X = Y$, by letting it consist of all sets of the
form $X \oplus X$, together with all sets of the form $X \oplus Y$
such that, for the least $m$ such that $X(m) \neq Y(m)$, we have $X =
(X \uhr m+1)Z_0$ and $Y = (X \uhr m+1)Z_1$ for some $Z_0 \oplus Z_1
\in \mathcal R_m$, where the $\mathcal R_m$ are uniformly $\Pi^0_1$
classes of $1$-randoms. For reasons addressed below, it will be
important to choose these classes so that $\mu(\mathcal R_m)$
approaches $1$ as $m$ increases.

Say that a sequence of sets is \emph{acceptable} if the sets are
pairwise distinct and for every pair of distinct sets $X$ and $Y$ in
the sequence, $X \oplus Y \in \mathcal P_2$. We aim to make
$\{X_\sigma : \sigma \in 2^n\}$ acceptable for all $n$. Then for any
two distinct paths $X$ and $Y$ through $T$, we will have that $X
\oplus Y$ is the limit of elements of $\mathcal P_2$, and hence is in
$\mathcal P_2$. Since $X \neq Y$, this will ensure that $X \oplus Y$
is $1$-random. We proceed recursively. Suppose that we have defined an
acceptable family $\{X_\sigma : \sigma \in 2^n\}$ and want to do the
same for $n+1$.

We can do so in a step-by-step fashion as long as we can establish a
lemma stating that if $Z_0,\ldots,Z_{n-1}$ is acceptable and $k \in
\omega$, then there is an acceptable $Y_0,\ldots,Y_n$ such that $Y_i
\uhr k = Z_i \uhr k$ for all $i < n$ and $Y_n \uhr k = Z_{n-1} \uhr
k$. To give an example, suppose we can do this, and we have $X_0$ and
$X_1$ and want to build $X_{00}, X_{01}, X_{10}, X_{11}$ as
above. Recall that we also have a parameter $k_1$. Then we can find an
acceptable sequence $Y_0,Y_1,Y_2$ such that $Y_0 \uhr k_1 = X_0 \uhr
k_1$ and $Y_1 \uhr k_1 = Y_2 \uhr k_1 = X_1 \uhr k_1$. Then we repeat
this procedure with $Y_1,Y_2,Y_0$ to get an acceptable sequence
$Z_0,Z_1,Z_2,Z_3$ such that $Z_0 \uhr k_1 = Z_3 \uhr k_1 = Y_0 \uhr
k_1 = X_0 \uhr k_1$, while $Z_1 \uhr k_1 = Y_1 \uhr k_1 = X_1 \uhr
k_1$ and $Z_2 \uhr k_1 = Y_2 \uhr k_1 = X_1 \uhr k_1$. Finally, we let
$X_{00}=Z_0$, $X_{01}=Z_3$, $X_{10}=Z_1$, and $X_{11}=Z_2$.

The necessary lemma can be proved using a measure argument. First, we
can argue using Lemma \ref{leblem} that we can assume that
$\bigoplus_{i<n} Z_i$ is a density-one point. Let $Z_n=Z_{n-1}$. If we
consider a pair of distinct numbers $i,j \leq n$ other than $n-1,n$,
the fact that $Z_i \oplus Z_j$ is a density-one point implies that the
relative measure of $\mathcal P_2$ above $(Z_i \oplus Z_j) \uhr m$
goes to $1$ as $m$ increases. Thus if $m$ is large enough, all such
relative measures will be close to $1$. 

For the pair $n-1,n$, the relative measure of $\mathcal P_2$ above
$(Z_{n-1} \oplus Z_n) \uhr m$ also has a positive lim inf by the
definition of $\mathcal P_2$.  To see that this is the case, let
$\mathcal{E}_m = \mathcal P_2 \cap [(Z_{n-1} \oplus Z_n) \uhr m]
$. Note that if $U \oplus V$ extends $(Z_{n-1} \oplus Z_n) \uhr 2m$
and $U(m) \neq V(m)$, and $U \oplus V$ has the form $((U \oplus V)
\uhr (2m + 2))^\frown R$ for some $R \in \mathcal R_m$, then $U \oplus
V \in \mathcal E_{2m}$. As the three events just mentioned are
mutually independent, we have that
\[
\mu(\mathcal E_{2m}) \geq 2^{-2m} \cdot 1/2 \cdot \mu(\mathcal R_m),
\]
since the three factors on the right-hand side are the respective
probabilities of the three events just mentioned. Since $\mathcal
E_{2m} \subseteq \mathcal E_{2m-1}$, we also have
\[ 
\mu(\mathcal E_{2m-1}) \geq  2^{-2m} \cdot 1/2 \cdot \mu(\mathcal R_m)
= 2^{-(2m-1)} \cdot 1/4 \cdot \mu(\mathcal R_m)
\] 
for $m>0$. It follows that
\[
d(Z_{n-1} \oplus Z_n \mid \mathcal{P}_2) = \textstyle{\liminf_k 2^k}
\mu(\mathcal E_k) \geq 1/4,
\]
since $\lim_m \mu(\mathcal R_m) = 1$.

Thus we see that if $m$ is sufficiently large, then the classes
$\mathcal C_{i,j} = \mathcal P_2 \cap [(Z_i \oplus Z_j) \uhr m]$ have
large enough measure to ensure that the class of all $\bigoplus_{i
\leq n} Y_i$ such that each $Y_i \oplus Y_j$ for $i \neq j$ is in
$\mathcal C_{i,j}$ has positive measure, and in particular contains an
element such that $Y_i \neq Y_j$ for all $i<j \leq n$, since the class
of all $\bigoplus_{i \leq n} Y_i$ such that $Y_i \neq Y_j$ for all
$i<j \leq n$ has measure $1$.

We now proceed with the full construction. Let $\mathcal R_0,\mathcal
R_1,\ldots$ be uniformly $\Pi^0_1$ classes of $1$-randoms such that
$\mu(\mathcal R_m)$ goes to $1$ as $m$ increases (for instance, the
complements of the levels of a universal Martin-L\"of test). Let
$\mathcal P_1 = \mathcal R_0$. For $n>1$, let $\mathcal P_n$ consist
of all $\bigoplus_{k < n} X_k$ such that either $X_{n-2} = X_{n-1}$ or
for the least $m$ such that $X_{n-2}(m) \neq X_{n-1}(m)$, there is a
$Y \in \mathcal R_m$ such that $\bigoplus_{k< n} X_k=(\bigoplus_{k<n}
(X_k \uhr m+1))^\frown Y$.  Note that the $\mathcal P_n$ are uniformly
$\Pi^0_1$ classes. Note also that if $\bigoplus_{k<n} X_k \in \mathcal
P_n$ and $X_{n-1} \neq X_{n-2}$, then $\bigoplus_{k < n} X_k$ is
$1$-random.

Let $n \geq 1$. For $X=\bigoplus_{i < n} X_i$ and $0 < m \leq n$, let
$\langle X \rangle^m$ be the set of all $\bigoplus_{j < m} X_{i_j}$
such that the $i_j$ are distinct numbers less than $n$. We say that
$X$ is \emph{$n$-acceptable} if $X_i \neq X_j$ for every $i<j<n$, and
for every $m \leq n$, every element of $\langle X \rangle^m$ is in
$\mathcal P_m$. Note that if $X$ is $n$-acceptable, then $X \in
\langle X \rangle^n \subset\mathcal P_n$, so $X$ is $1$-random. Note
also that if $\bigoplus_{i < n} X_i$ is $n$-acceptable and $\pi$ is a
permutation of $0,1,\ldots,n-1$, then $\bigoplus_{i < n} X_{\pi(i)}$
is also $n$-acceptable.

\begin{lemma}
\label{auxlem}
Let $\bigoplus_{i<n} X_i$ be $n$-acceptable, and let $k \in
\omega$. Then there is an $n$-acceptable density-one point
$\bigoplus_{i<n} Z_i$ such that $Z_i \uhr k = X_i \uhr k$ for all $i <
n$.
\end{lemma}

\begin{proof}
Let $\mathcal P$ be the class of all $Z=\bigoplus_{i<n} Z_i$ such that
for every $0< m \leq n$, every element of $\langle Z \rangle^m$ is in
$\mathcal P_m$, and $Z_i \uhr k = X_i \uhr k$ for all $i < n$. Then
$\mathcal P$ is a $\Pi^0_1$ class containing the $1$-random set
$\bigoplus_{i<n} X_i$, and hence $\mathcal P$ has positive measure, so
it contains a density-one $1$-random point $Z=\bigoplus_{i<n} Z_i$, by
Lemma \ref{leblem}. Since $Z$ is $1$-random, $Z_i \neq Z_j$ for every
$i<j<n$, so $Z$ is $n$-acceptable.
\end{proof}

\begin{lemma}
\label{mainlem}
Let $\bigoplus_{i<n} X_i$ be $n$-acceptable, and let $k \in
\omega$. Let $X_n=X_{n-1}$. Then there is an $n+1$-acceptable
$\bigoplus_{i \leq n} Y_i$ such that $Y_i \uhr k = X_i \uhr k$ for all
$i \leq n$.
\end{lemma}

\begin{proof}
By Lemma \ref{auxlem}, we can assume that $\bigoplus_{i<n} X_i$ is a
density-one point.

Let $S$ be the set of all nonempty sequences of distinct numbers less
than or equal to $n$. For each $s=(i_0,\ldots,i_{m-1}) \in S$, let
$\mathcal C_s$ be the class of all $Y = \bigoplus_{i \leq n} Y_i$ such
that $\bigoplus_{j < m} Y_{i_j} \in \mathcal P_m$. Our goal is to show
that there is an element $\bigoplus_{i \leq n} Y_i$ of the
intersection of all of these classes such that $Y_i \neq Y_j$ for all
$i<j \leq n$ and $Y_i \uhr k = X_i \uhr k$ for all $i \leq n$. Our
strategy is first to define a sequence of strings $\{\sigma_l\}_{l \in
\omega}$ in such a way that we can show that, for all $s \in S$, the
relative measure of $\mathcal C_s$ above $\sigma_l$ approaches $1$ as
$l$ grows. Let
\[
\sigma_l = (X_0 \uhr l+1) \oplus \cdots \oplus (X_{n-2} \uhr l+1)
\oplus ((X_{n-1} \uhr l)^\frown 0) \oplus ((X_n \uhr l)^\frown 1).
\]

Let $s=(i_0,\ldots,i_{m-1}) \in S$ and suppose that $s$ does not
contain both $n - 1$ and $n$. Then $\bigoplus_{j < m} X_{i_j}$ is a
density-one point, by Lemma \ref{baslem}. Furthermore, since
$\bigoplus_{i<n} X_i$ is $n$-acceptable, $\bigoplus_{j < m} X_{i_j}
\in \mathcal P_m$. Thus $\bigoplus_{i \leq n} X_i \in \mathcal C_s$,
and hence, by Lemma \ref{baslem},
\[
d\left(\bigoplus_{i \leq n} X_i \mid \mathcal C_s\right) = 1.
\]
We claim that it now follows from Lemma \ref{measlem} that the
relative measure of $\mathcal C_s$ above $\sigma_l$ approaches $1$ as
$l$ increases. To justify this claim, let $X = \bigoplus_{i \leq n}
X_i$, so that $d\left(X \mid \mathcal C_s\right) = 1$. Then
$\bigoplus_{i \leq n} (X_i \uhr l) = X \uhr l(n+1)$.  For all $l$, the
string $\sigma_l$ is obtained from the latter string by adding $n + 1$
new bits at the end. So applying Lemma \ref{measlem} with $c = n+1$
and values of $t$ of the form $l(n+1)$ shows that the relative measure
of $\mathcal C_s$ above $\sigma_l$ approaches $1$ as $l$ increases.

If $s = (i_0,\ldots,i_{m-1}) \in S$ contains both $n-1$ and $n$, then
let $a<b$ be such that $\{i_a,i_b\}=\{n-1,n\}$. For any $l$, any
$\bigoplus_{j<m} Z_j \in \mathcal R_l$, and any $c_0,\ldots,c_{m-1}
\in \{0,1\}$ such that $c_a \neq c_b$, we have $\bigoplus_{j<m}
((X_{i_j} \uhr l)^\frown c_j^\frown Z_j) \in \mathcal P_m$.
Thus
\[
\mu([\sigma_l] \cap \mathcal C_s) \geq 2^{- \mid \sigma_l
\mid}\mu(\mathcal R_l).
\]
Since $\mu(\mathcal R_l)$ goes to $1$ as $l$ increases, it follows
that the relative measure of $\mathcal C_s$ above $\sigma_l$
approaches $1$ as $l$ increases.

It follows from the previous two paragraphs that for \emph{all} $s \in
\mathcal{S}$ the relative measure of $\mathcal C_s$ above $\sigma_l$
approaches $1$ as $l$ increases. Hence the relative measure of
$\bigcap_{s \in \mathcal S} \mathcal C_s$ above $\sigma_l$ approaches
$1$ as $l$ increases.

Let $k$ be as in the hypothesis of the lemma, and let $\beta =
\bigoplus_{i \leq n}(X_i \uhr k)$. If $l \geq k$, then $\sigma_l$
extends $\beta$. Fix $l$ sufficiently large so that $l \geq k$ and
$\mu([\sigma_l] \cap \bigcap_{s \in \mathcal S} \mathcal C_s) >
0$. Then
\[
\mu\left(\left[\beta\right] \cap \bigcap_{s \in S} \mathcal
C_s\right)  > 0.
\]
Let $\mathcal C$ be the class of all $n+1$-acceptable $\bigoplus_{i
\leq n} Y_i$ such that $Y_i \uhr k = X_i \uhr k$ for all
$i \leq n$, and let $\mathcal D = \{\bigoplus_{i \leq n} Y_i :
(\forall i<j \leq n)[Y_i \neq Y_j]\}$. Then $\mathcal D$ has measure
$1$, and 
\[
\mathcal C = [\beta] \cap \bigcap_{s \in S} \mathcal C_s \cap \mathcal
D,
\] 
so $\mathcal C$ has positive measure, and in particular is nonempty.
\end{proof}

We now build a $\emptyset'$-computable perfect tree as follows.  For
each $\sigma \in 2^n$, we will define a set $X_\sigma$, and for each
$n$ we will define a number $k_n$ so that
\begin{enumerate}

\item $\bigoplus_{\sigma \in 2^n} X_\sigma$ is $2^n$-acceptable,

\item $X_{\sigma^\frown i} \uhr k_n = X_\sigma \uhr k_n$ for $i=0,1$,
and 

\item $X_\sigma \uhr k_n \neq X_\tau \uhr k_n$ for all distinct
$\sigma,\tau \in 2^n$.

\end{enumerate}
We then take $T(\sigma) = X_\sigma \uhr k_n$. Note that, by the
closure of $n$-acceptability under permutations mentioned above Lemma
\ref{auxlem}, it does not matter in item (1) how we order $2^n$.

For the empty string $\lambda$, let $X_\lambda \in \mathcal P_1$ and
let $k_n=0$. Given $k_n$ and $X_\sigma$ for each $\sigma \in 2^n$,
apply Lemma \ref{mainlem} repeatedly to obtain sets $X_{\tau}$ for
each $\tau \in 2^{n+1}$ so that $\bigoplus_{\tau \in 2^{n+1}} X_\tau$
is $2^{n+1}$-acceptable, and each $X_{\sigma^\frown i}$ extends
$X_\sigma \uhr k_n$. (Here we again use the fact that
$n$-acceptability is closed under permutations.) We can do this
$\emptyset'$-computably because the class $\mathcal C$ in the proof of
Lemma \ref{mainlem} is nonempty and the class $\mathcal D$ in that
proof is $\Sigma^0_1$, so using $\emptyset'$ we can find a $\sigma$
such that $[\sigma] \subset \mathcal D$ and $[\sigma] \cap \mathcal C$
is nonempty, and then $[\sigma] \cap \mathcal C$ is a $\Pi^0_1$ class,
so $\emptyset'$ can find a path on it. Let $k_{n+1}$ be sufficiently
large so that $X_\sigma \uhr k_{n+1} \neq X_\tau \uhr k_{n+1}$ for all
distinct $\sigma,\tau \in 2^{n+1}$.

Then $T$ is a perfect tree. Let $Y_0,\ldots,Y_n$ be distinct paths
through $T$. Then each $Y_i$ is the limit of sets $X_{\tau_{i,k}}$
with $|\tau_{i,k}|=k$. If $k$ is sufficiently large, then
$\tau_{0,k},\ldots,\tau_{n,k}$ are pairwise distinct, so $\bigoplus_{i
\leq n} X_{\tau_{i,k}} \in \mathcal P_{n+1}$. Since $\mathcal
P_{n+1}$ is closed, it follows that the limit $\bigoplus_{i \leq n}
Y_i$ is also in $\mathcal P_{n+1}$, and hence is $1$-random.
\end{proof}

\begin{corollary}
For each $n \geq 1$, there is a $\emptyset^{(n)}$-computable perfect
tree $T$ such that for any nonempty finite collection $\mathcal F$ of
pairwise distinct paths through $T$, the set $\bigoplus_{Y \in
\mathcal F} Y$ is $n$-random, and hence any two finite, disjoint,
nonempty collections of pairwise distinct paths through $T$ are
mutually $n$-random.
\end{corollary}

As another consequence, we get a proof of Mycielski's Theorem
\ref{Ramsey}, as noticed earlier by Miller and Yu \cite{MilYu}, who
gave their own direct proof of Corollary \ref{perfmutcor}, though
without a bound on the complexity of the perfect tree: Let $\mathcal
M_0,\mathcal M_1,\ldots$ be such that each $\mathcal M_i$ is a
measure-$0$ subset of $(2^{\omega})^{n_i}$ for some $n_i \geq 1$. Then
each $\mathcal N_i = \{\bigoplus_{j < n_i} X_i :
(X_0,\ldots,X_{n_i-1}) \in \mathcal M_i\}$ is a measure-$0$ subset of
$2^\omega$, and so is contained in a measure-$0$ $G_\delta$ subset
$\mathcal C_i$ of $2^\omega$. Each $\mathcal C_i$ is the intersection
of an $A_i$-Martin-L\"of test for some $A_i$. Let $A = \bigoplus_i
A_i$, and let $T$ be as in the theorem. If $X_0,\ldots,X_{n_i-1}$ are
distinct paths through $T$, then $\bigoplus_{j < n_i} X_i$ is
$1$-random relative to $A$, and hence relative to $A_i$, and so is not
in $\mathcal C_i$, and hence $(X_0,\ldots,X_{n_i-1})$ is not in
$\mathcal M_i$.

Indeed, this proof shows that Theorem \ref{perfeff} is basically a
``quantitative version'' of Mycielski's Theorem.

Notice that in the proof of Theorem \ref{perfeff}, $\mathcal
P_1=\mathcal R_0$ can be any nonempty $\Pi^0_1$ class of
$1$-randoms. So if we are given a $\Pi^0_1$ class $\mathcal C$ of
positive measure, then we can intersect $\mathcal C$ with a $\Pi^0_1$
class of $1$-randoms of sufficiently large measure, then take that
intersection as $\mathcal P_1$. Thus the theorem still holds if we
require that the paths through $T$ be in some given $\Pi^0_1$ class of
positive measure. This fact implies for instance that $T$ can be
chosen to be \emph{pathwise-random}, as defined by Barmpalias and Wang
\cite{BW}, which means that there is a $c$ such that every path $X$
through $T$ has randomness deficiency at most $c$ (i.e., $K(X \uhr n)
\geq n - c$ for all $n$, where $K$ is prefix-free Kolmogorov
complexity). (They call trees all of whose paths are $1$-random
\emph{weakly pathwise-random}.)

Ku{\v c}era \cite{Ku} showed that if $\mathcal C$ is a $\Pi^0_1$ class
of positive measure, then every $1$-random has a tail in $\mathcal C$.
The fact that Theorem \ref{perfeff} still holds if we require that the
paths through $T$ be in some given $\Pi^0_1$ class of positive measure
also follows from the following extension of Ku{\v c}era's result,
which has also been noted by Barmpalias and Wang \cite{BW}. For a tree
$T$ and $\sigma \in T$, we write $T_\sigma$ for the tree consisting of
all $\tau$ such that $\sigma^\frown \tau \in T$. We say that $\sigma
\in T$ is \emph{extendible} if $T_\sigma$ is infinite.

\begin{proposition}
\label{classprop}
Let $T$ be an infinite binary tree such that each path through $T$ is
$1$-random and let $\mathcal C$ be a $\Pi^0_1$ class of positive
measure. Then there is an extendible $\sigma \in T$ such that every
path through $T_\sigma$ is in $\mathcal C$.
\end{proposition}

\begin{proof}
This proof is a minor variation on that of the aforementioned
result of Ku{\v c}era \cite{Ku}. Suppose that no such $\sigma$
exists. We can assume that $\mathcal C \neq 2^\omega$.  Let $W$ be a
prefix-free set of strings generating the complement of $\mathcal C$
(i.e., this complement is $\bigcup_{\sigma \in W} [\sigma]$). Let $S_0
= W$ and $S_{n+1} = \{\sigma^\frown \tau : \sigma \in S_n
\,\mathbin{\&}\, \tau \in W\}$. Let $\mathcal U_n$ be the $\Sigma^0_1$
class generated by $S_n$. Then the $\mathcal U_n$ are uniformly
$\Sigma^0_1$ classes, and $\mu(\mathcal U_{n+1})=\mu(\mathcal
U_n)\mu(\mathcal U_0)=\mu(\mathcal U_0)^{n+2}$, so we can find a
Martin-L\"of test $\mathcal U_{n_0},\mathcal U_{n_1},\dots$, since
$\mu(U_0) = 1 - \mu(\mathcal{C}) < 1$.

Since not every path through $T$ is in $\mathcal C$, there is an
extendible $\sigma_0 \in T$ that is in $S_0$. Since not every path
through $T_{\sigma_0}$ is in $\mathcal C$, there is a $\sigma_1$ such
that $\sigma_0^\frown\sigma_1$ is an extendible element of $T$ and is
in $S_1$. Proceeding in this way, we build $\sigma_0,\sigma_1,\ldots$
such that $\sigma_0^\frown\cdots^\frown\sigma_n$ is an extendible
element of $T$ and $T_{\sigma_0^\frown\cdots^\frown\sigma_n} \subseteq
\mathcal U_n$. Then $\sigma_0^\frown\sigma_1^\frown\cdots$ is a path
through $T$, but is also in every $\mathcal U_n$, and hence is not
$1$-random.
\end{proof}

It is interesting to consider whether $A'$ is the best we can do in
Theorem \ref{perfeff}. While we do not know the answer to this
question, we can give a lower bound using the following fact. A degree
$\mathbf{x}$ is a \emph{strong minimal cover} of a degree $\mathbf{a}$
if $\mathbf{a} < \mathbf{x}$ and every degree strictly below
$\mathbf{x}$ is below $\mathbf{a}$.

\begin{theorem}
\label{negthm}
If a degree has a strong minimal cover then it does not compute any
perfect tree all of whose paths are $1$-random.
\end{theorem}

\begin{proof}
Suppose the degree $\mathbf{a}$ has a strong minimal cover
$\mathbf{x}$. If $X \in \mathbf{x}$ then $X$ cannot be $1$-random, as
otherwise if we write $X = X_0 \oplus X_1$ then both $X_0$ and $X_1$
have degree strictly below $\mathbf{x}$, and hence are
$\mathbf{a}$-computable, whence so is $X$.

Suppose that there is an $\mathbf{a}$-computable perfect tree $T$ all
of whose paths are $1$-random. Let $X \in \mathbf{x}$ and let $B =
T(X)$. Then $B \leq\sub{T} T \oplus X \equiv\sub{T} X$, but $B$ is a
path through $T$, and hence is $1$-random, so in fact $B <\sub{T} X$,
and hence $B$ is $\mathbf{a}$-computable. From $B$ and $T$ we can
compute $X$, however, so $X$ is $\mathbf{a}$-computable, which is a
contradiction.
\end{proof}

Lewis \cite{Lew} showed that there is a $1$-random degree with a
strong minimal cover. Indeed, Barmpalias and Lewis \cite{BL} showed
that every $2$-random degree has a strong minimal cover. Thus we have
the following corollary, which can be seen as an analog to the fact
that in the set-theoretic context, adding a random real does not
necessarily add a perfect set of random reals. (See for instance
Bartoszynski and Judah \cite{BJ}.)

\begin{corollary}
\label{negcor}
There is a $1$-random that does not compute any perfect tree all of
whose paths are $1$-random. Indeed, every $2$-random has this
property.
\end{corollary}

Barmpalias and Wang \cite{BW} have independently proved a stronger
version of this result, showing that $2$-randomness can be replaced by
the weaker notion of difference randomness, shown by Franklin and Ng
\cite{FN} to be equivalent to being $1$-random and not computing
$\emptyset'$. Notice that if $T$ is a perfect tree such that every
path through $T$ is $1$-random, then since $\mathcal C=\{X : K(X \uhr
n) \geq n-c\}$ is a $\Pi^0_1$ class of positive measure for all
sufficiently large $c$, it follows from Proposition \ref{classprop}
that there is an extendible $\sigma \in T$ such that every path
through $T_\sigma$ is in $\mathcal C$. Then $T_\sigma$ is a
$T$-computable pathwise-random tree that is perfect, and hence has
infinitely many paths.

\begin{theorem}[Barmpalias and Wang \cite{BW}]
\label{bwthm}
Let $X \ngeq\sub{T} \emptyset'$ be $1$-random. Then $X$ does not
compute a pathwise-random tree with infinitely many paths. Thus $X$
does not compute a perfect tree all of whose paths are $1$-random.
\end{theorem}

Chong, Li, Wang, and Yang \cite[Question 4.2]{CLWY} asked whether
there is a (computable or not) tree $T$ such that the set of paths
through $T$ has positive measure but there are only measure-$0$ many
oracles that compute a perfect subtree of $T$. Corollary \ref{negcor}
and Theorem \ref{bwthm} give strong positive answers to this question,
as they show that no $2$-random, and even no difference random, can
compute a perfect subtree of a tree all of whose paths are $1$-random.

If we weaken Theorem \ref{perfeff} to say only that each individual
path on $T$ is $1$-random, then it has an obvious proof, since we can
take a computable binary tree $R$ all of whose paths are $1$-random,
and use $\emptyset'$ to find a perfect subtree of $R$, using the fact
that $R$ has no isolated paths. But in this case we can improve the
theorem from $A'$ to any set that has PA degree relative to $A$, by
work of Greenberg, Miller, and Nies \cite{GMN} and of Chong, Li, Wang,
and Yang \cite{CLWY}. Say that a binary tree has positive measure if
the set of paths through $T$ does. The former group of authors showed
that if $T$ is a tree of positive measure and $B$ has PA degree
relative to $T$, then there is a $B$-computable nonempty subtree $S$
of $T$ such that for every $\sigma \in S$, the tree $S_\sigma$ has
positive measure, and such an $S$ is perfect (i.e., it has no dead
ends and no isolated paths). The latter group of authors showed
directly and independently that if $T$ is a tree of positive measure
and $B$ has PA degree relative to $T$, then $T$ has a $B$-computable
perfect subtree (which can be chosen to have positive
measure). Greenberg, Miller, and Nies showed that the PA degrees are
not a lower bound for their result, and Chong, Li, Wang, and Yang
noted that Patey did the same in their context. We will return to
these results in the setting of reverse mathematics below.

Another way we could weaken Theorem \ref{perfeff} is to replace
$1$-randomness by a weaker notion of randomness. Of course, notions of
computability-theoretic randomness (as opposed to, say,
complexity-theoretic randomness) generally do not admit computable
instances, so for such notions we can never have a fully effective
version of the theorem. However, for the weak notion of independence
in Theorem \ref{perfhalfcor}, which was the original motivation for
the work in this section, we have the following result.

\begin{theorem}
\label{compperfthm}
There is a computable perfect tree $T$ such that every path through
$T$ has density $1/2$, and the symmetric difference of any two
distinct paths through $T$ also has density $1/2$.
\end{theorem}

\begin{proof}
Call $\sigma \in 2^{< \omega}$ \emph{balanced} if $|\sigma^{-1} (0)| =
|\sigma^{-1}(1)|$, i.e.\ $\sigma$ has the same number of $0$'s as
$1$'s.  Call a pair $(\sigma , \tau)$ of strings \emph{balanced} if
$\sigma$ and $\tau$ have the same length and agree on exactly half
of their arguments.  If $\sigma$ is a string and $n \in \omega$, let
$\sigma^n$ denote the concatenation of $n$ copies of $\sigma$, and
call $\sigma^n$ a \emph{power} of $\sigma$.  Note that any power of a
balanced string is balanced.

Let $\mu_i = 0^{2^i} 1^{2^i}$.  Then $\mu_i$ is balanced and has
length $2^{i+1}$.  Our tree $T$ will have the following properties:
\begin{enumerate}

\item $T(\sigma)$ is balanced for every $\sigma$.

\item The length of $T(\sigma)$ depends only on the length of
$\sigma$.  Let $l_n$ denote the length of $T(\sigma)$ for all
$\sigma$ of length $n$.

\item For all strings $\sigma $ and $i \leq 1$, the string
$T(\sigma)^\frown i$ will be of the form $T(\sigma)^\frown \mu_j^k$
for some $j$ and $k$.
  
\end{enumerate}

We now give the definition of $T$.  Let $T(\lambda) = \lambda$, where
$\lambda$ is the empty string.  Assume inductively that $T(\sigma)$ is
defined for every string $\sigma$ of length $n$, and that all of these
strings are balanced and have the same
length $l_n$.  We will choose $l_{n+1} = l_n + 2^{2^{n+4}}$.  (The
reason for this outlandish choice will become clear later.)

Let the strings of length $n$ be $\sigma_1, \sigma_2, \dots,
\sigma_{2^n}$.  For each $k \leq 2^n$ and $i \leq 1$, define
$T(\sigma_k^\frown i)$ to be $T(\sigma_k)^\frown \gamma$, where
$\gamma$ is the unique power of $\mu_{2k + i - 1}$ of length $l_{n+1}
- l_n$.  (There is such a power because $l_{n+1} - l_n = 2^{2^{n+4}}$
is a multiple of $|\mu_{2k + i - 1}| = 2^{2k + i}$.)

It follows immediately from the construction by induction that if
$|\sigma| = n$ then $|T(\sigma)| = l_n$.  It is also easy to see that
$T(\sigma)$ is balanced for all $\sigma$.
 
Let $b_n = |\mu_{2^n}| = 2^{2^n + 1}$.  Call a number $l$ \emph{good}
if every string on $T$ of length $l$ is balanced.  We have already
remarked that $l_n$ is good for every $n$.  The next lemma gives
further examples of good numbers.

\begin{lemma}
\label{cl1}
If $l_n \leq j \leq l_{n+1}$ and $j - l_n$ is divisible by $b_n$, then
$j$ is good.
\end{lemma}

\begin{proof}
To prove the lemma, assume that $j$ is as in its hypothesis, and let
$\gamma$ be a string on $T$ of length $j$.  We must show that $\gamma$
is balanced.  Write $\gamma$ as $T(\sigma) ^\frown \nu$, where
$\sigma$ has length $n$.  Since $T(\sigma)$ is balanced, it suffices
to show that $\nu$ is balanced.  Let $i \leq 1$ be such that $\gamma
\preccurlyeq T(\sigma ^\frown i)$.  Then by construction, $\nu$ is
extended by a power of $\mu_{2k + i - 1}$, where $\sigma = \sigma_k$.
Since $|\mu_{2k + i - 1}|$ divides $b_n$, and $b_n$ divides $k - l_n =
|\gamma|$, it follows that $|\mu_{2k + i - 1}|$ divides $|\nu|$.
Since $\nu$ is extended by a power of $\mu_{2k + i - 1}$ and $|\mu_{2k
+ i - 1}|$ divides $|\nu|$, it follows that $\nu$ is a power of
$\mu_{2k + i - 1}$.  Since $\mu_{2k + i - 1}$ is balanced, and powers
of balanced strings are balanced, $\nu$ is balanced.
\end{proof}

\begin{lemma}
\label{cl2}
Every path through $T$ has density $1/2$.
\end{lemma}

\begin{proof}
To prove the lemma, let $C$ be a path through $T$, so $C \restr k$ is
on $T$ for every $k$.  Let $l_n \leq k < l_{n+1}$.  Let $j$ be maximal
such that $j \leq k$ and $j - l_n$ is divisible by $b_n$.  Then $k - j
\leq b_n$.  Also, $j$ is good by Lemma \ref{cl1}, so $|C \restr j| = j
/ 2$.  It follows that
\[
|C \restr k| \leq |C \restr j| + b_n.
\]
Dividing through by $k$, we obtain that
\[
\rho_k(C) \leq \frac{|C \restr j|}{k} + \frac{b_n}{k} \leq \frac{|C
  \restr j|}{ j} + \frac{b_n}{l_n} = \frac{1}{2} + \frac{b_n}{l_n}.
\]
It follows from the definitions of $b_n$ and $l_n$ and a
straightforward computation that $\lim_n \frac{b_n}{l_n} = 0$, and so
$\overline{\rho}(C) \leq 1/2$.  Replacing $C$ by its complement $\neg
C$ in the above argument we obtain that $\overline{\rho}(\neg C) \leq
1/2$, and so $\underline{\rho}(C) \geq 1/2$.  It follows that $\rho(C)
= 1/2$.
\end{proof}

It remains to be shown that if $A$ and $B$ are distinct branches of
$T$, then $\rho(A \sd B) = 1/2$. The following fact will be useful.

\begin{lemma}
\label{cl3}
Let $i \neq j$, and let $\sigma$ and $\tau$ be such that $|\sigma| =
|\tau|$. If $\sigma$ is a power of $\mu_i$ and $\tau$ is a power of
$\mu_j$ then $(\sigma, \tau)$ is balanced.
\end{lemma}

\begin{proof}
Suppose that $i < j$.  Break up the numbers less than $|\tau|$ into
consecutive intervals of length $2^j$.  This is possible because
$\tau$ is a power of $\mu_j$, and hence $2^j$ divides $|\tau|$.  On
each such interval, $\tau$ has the form $0^{2^j}$ or $1^{2^j}$, and
$\sigma$ is a power of $0^{2^i} 1^{2^i}$.  Hence $\sigma$ and $\tau$
agree on exactly half of each interval, so $(\sigma, \tau)$ is
balanced.
\end{proof}
 
\begin{lemma}
\label{cl4}
Let $A$ and $B$ be distinct paths through $T$. Then $\rho(A \sd B) =
1/2$.
\end{lemma}

\begin{proof}
For every sufficiently large $n$, there are distinct strings $\sigma$
and $\tau$ of length $n$ such that $A$ extends $T(\sigma)$ and $B$
extends $T(\tau)$. We assume that this fact holds for every $n > 0$,
since the general case is essentially the same. Then, by induction and
the previous lemma, $(A \restr l_n, B \restr l_n)$ is balanced for
every $n > 0$. If $l_n \leq j \leq l_{n+1}$ and $j - l_n$ is divisible
by $b_n$, then as in the proof of Lemma \ref{cl1}, $A \restr j$ is of
the form $(A \restr l_n)^\frown \nu_0$ for a power $\nu_0$ of some
$\mu_{i_0}$, and $B \restr j$ is of the form $(B \restr l_n)^\frown \nu_1$
for a power $\nu_1$ of some $\mu_{i_1}$, and it follows from the
definition of $T$ that $i_0 \neq i_1$. So by the previous lemma, $(A
\restr j, B \restr j)$ is balanced. Now essentially the same argument
as that in the proof of Lemma \ref{cl2} shows that $\rho(A \sd B) =
1/2$.
\end{proof}
 
The theorem follows from Lemmas \ref{cl2} and \ref{cl4}.
\end{proof}

We can also consider other variants on Theorem \ref{perfeff} where we
relax the conditions on $T$, possibly leading to better upper bounds
on its complexity. One possibility is to require only that pairs of
distinct paths be mutually $1$-random. Another is to replace the
condition that $T$ be perfect by the condition that $T$ be infinite
and have no isolated paths (which still ensures that the set of paths
through $T$ is perfect), and of course we can do both at the same
time. We do not know whether any of these variants give rise to better
upper bounds. For infinite trees with no isolated paths, we also no
longer have the same lower bound, but notice that if $T$ is a
computable infinite tree all of whose paths are $1$-random, then there
are paths through $T$ that are not mutually $1$-random: Take
incompatible strings $\sigma$ and $\tau$ in $T$ such that $T$ is
infinite above both $\sigma$ and $\tau$. Then the leftmost paths of
$T$ above $\sigma$ and above $\tau$ are both left-c.e.\ $1$-randoms,
so they have the same degree, namely $\mathbf{0'}$, and hence cannot
be mutually $1$-random.

If we go even further and require only that $T$ be an infinite tree,
or an infinite tree with no dead ends, then we do get a better upper
bound for fairly trivial reasons: If $X$ is $1$-random relative to $A$
then the set of all strings $X \uhr n$ is the desired tree. (It does
not change things if we add the condition that $T$ have infinitely
many paths, because if we write $X = \bigoplus_i X_i$ then we can form
a tree by taking $X_0$ and appending a copy of $X_{i+1}$ with root $(X
\uhr i)^\frown(1-X(i))$ for each $i$.) If we do require that $T$ have
no dead ends, then $1$-randomness is also a lower bound.

\subsection{Mycielski's Theorem for measure and reverse mathematics}
\label{mycsec2}

We now turn to reverse mathematics. From now on, all implications and
equivalences we mention are over the usual weak base system RCA$_0$. A
binary tree $T$ is \emph{positive} if there is a $q>0$ such that
$\frac{|T \cap 2^n|}{2^n} > q$ for all $n$. The system WWKL$_0$
consists of RCA$_0$ together with WWKL, the principle that every
positive tree has an infinite path. This principle is equivalent to
the one stating that for each $A$, there is a set that is $1$-random
relative to $A$, where the latter notion is formalized using
Martin-L\"of tests. (See Avigad, Dean, and Rute \cite{ADR} for
details.) We will implicitly use van Lambalgen's Theorem (for instance
in not having to distinguish between saying that $A$ and $B$ are
relatively $1$-random and saying that $A \oplus B$ is $1$-random), so
it is worth noting that a standard proof of this theorem, for instance
the one given in \cite[Section 6.9]{DH}, can be carried out in RCA$_0$.

Corresponding to their computability-theoretic work mentioned above,
Greenberg, Miller, and Nies \cite{GMN} defined the following
principles, whose strength they showed to be strictly intermediate
between WWKL$_0$ and WKL$_0$.

\medskip

\noindent\textbf{WSWWKL}: Every positive tree has a positive subtree
with no dead ends.

\medskip

\noindent\textbf{SWWKL}: Every positive tree has a positive subtree $T$
such that for every $\sigma \in T$, the restriction of $T$ to strings
compatible with $\sigma$ is positive.

\medskip

At about the same time, Chong, Li, Wang, and Yang \cite{CLWY} studied
the following principle, which has also been studied by Barmpalias and
Wang \cite{BW}, who denoted it by P.

\medskip

\noindent \textbf{PSUB}: Every positive tree has a perfect subtree.

\medskip

Chong, Li, Wang, and Yang \cite{CLWY} showed that PSUB is provable in
WKL$_0$ and noted that Patey gave an argument showing that it does not
imply WKL. (The fact that PSUB is provable in WKL$_0$ is also implicit
in an earlier proof by Barmpalias, Lewis, and Ng \cite{BLN}; see also
\cite[proof of Theorem 8.8.8]{DH}.) These facts also follow from the
results of Greenberg, Miller, and Nies \cite{GMN}, since WSWWKL
clearly implies PSUB. Indeed, it implies the following principle,
which was shown by Barmpalias and Wang \cite{BW} (who denote it by
P$^+$) to be strictly stronger than PSUB.

\medskip

\noindent \textbf{SPSUB}: Every positive tree has a positive perfect
subtree.

\medskip

Clearly, PSUB implies WWKL, since from a perfect subtree of a positive
tree all of whose paths are $1$-random relative to $A$, we can compute
a set that is $1$-random relative to $A$. Chong, Li, Wang, and Yang
\cite[Question 4.1]{CLWY} asked whether PSUB is provable in
WWKL$_0$. We can answer this question with the following result, which
has also been obtained by Barmpalias and Wang \cite{BW}. (Here an
\emph{$\omega$-model} is a structure in the language of second-order
arithmetic whose first-order part is standard.)

\begin{theorem}
\label{psubthm}
There is an $\omega$-model of $\textup{WWKL}_0$ that is not a model of
\textup{PSUB}.
\end{theorem}

\begin{proof}
As noted above, WWKL is equivalent over RCA$_0$ to the principle that
for each $X$, there is a set that is $1$-random relative to $X$. If
$Z$ is $1$-random then let $Z_i = \{n : \langle i,n \rangle \in
Z\}$, and let $\mathcal S = \{Y : (\exists k)[Y \leq\sub{T}
\bigoplus_{i \leq k} Z_i]\}$. Then $\mathcal S$ is a Turing ideal, and
hence is an $\omega$-model of RCA$_0$. Furthermore, by van Lambalgen's
Theorem, each $Z_i$ is $1$-random relative to $\bigoplus_{j \neq i}
Z_j$, so $\mathcal S$ is in fact a model of WWKL$_0$.
  
Let $A$ be a $1$-random as in Corollary \ref{negcor}. By the previous
paragraph, there is an $\omega$-model of WWKL$_0$ all of whose members
are $A$-computable. If PSUB holds in this model then, applying it to a
computable positive tree all of whose paths are $1$-random, we obtain
an $A$-computable perfect tree all of whose paths are $1$-random,
contradicting the choice of $A$.
\end{proof}

\begin{corollary}
$\textup{RCA}_0 + \textup{PSUB}$ is strictly intermediate between
\textup{WWKL}$_0$ and \textup{WKL}$_0$.
\end{corollary}

Notice that, by Corollary \ref{negcor}, the set $A$ in the proof of
Theorem \ref{psubthm} can be chosen to be any $2$-random, and hence to
have any level of randomness we would like. So, for example, PSUB is
not implied by the statement that for each $A$, there is a set
that is arithmetically random relative to $A$.

We can begin to connect these principles with Mycielski's Theorem by
showing that PSUB is equivalent to the following statement.

\medskip

\textbf{PTR}: For each $A$, there is a perfect tree $T$ such that
every path through $T$ is $1$-random relative to $A$.

\medskip

To see that PTR and PSUB are equivalent, we use the following version
of Proposition \ref{classprop}. The restriction to trees with no dead
ends is necessary because in the absence of WKL, an infinite tree can
satisfy the condition that all of its paths be $1$-random simply by
not having any paths, an issue we will return to below.

\begin{proposition}
\label{classrca}
$\textup{RCA}_0$ proves that if $S$ is a positive tree, and $T$ is an
infinite binary tree with no dead ends such that each path through $T$
is $1$-random relative to $S$, then there is an extendible $\sigma \in
T$ such that $T_\sigma$ is a subtree of $S$.
\end{proposition}

\begin{proof}
The proof of Proposition \ref{classprop} can be carried out in
RCA$_0$. The only thing to note is that, since we do not have to worry
about extendibility when $T$ has no dead ends, the existence of the
strings $\sigma_i$ in that proof follows by $\Sigma^0_1$-induction,
and then the sequence $\sigma_0,\sigma_1,\ldots$ can be obtained
computably in $T$.
\end{proof}

Thus PTR is equivalent to the statement that for each $A$,
each positive tree $S$ has a perfect subtree $T$ such that every path
through $T$ is $1$-random relative to $A$. Since (provably in RCA$_0$)
for each $A$ there is a positive tree whose paths are $1$-random
relative to $A$, PTR and PSUB are equivalent. Thus, however we
formalize Mycielski's Theorem for measure using perfect trees, it
should imply PSUB.

Theorem \ref{perfeff} and the proof of Mycielski's Theorem for measure
using this theorem suggest that the following is a reasonable way to
formulate Mycielski's Theorem for measure as a statement of
second-order arithmetic.

\medskip

\noindent \textbf{MYC-M}: For each $A$, there is a perfect tree $T$
such that for any nonempty finite collection $\mathcal F$ of pairwise
distinct paths through $T$, the set $\bigoplus_{Y \in \mathcal F} Y$
is $1$-random relative to $A$.

\medskip

By Proposition \ref{classrca}, MYC-M is equivalent to the statement
that for each $A$, each positive tree $S$ has a perfect subtree $T$
such that for any nonempty finite collection $\mathcal F$ of pairwise
distinct paths through $T$, the set $\bigoplus_{Y \in \mathcal F} Y$
is $1$-random relative to $A$.

We can also consider the following version of MYC-M restricted to
joins of pairs of paths, which corresponds to Corollary
\ref{perfrandcor}.

\medskip

\noindent\textbf{RMYC-M}: For each $A$, there is a perfect tree $T$
such that any two distinct paths through $T$ are mutually $1$-random
relative to $A$.

\medskip

RMYC-M implies PSUB, since it clearly implies PTR, and hence it
properly implies WWKL. The proof of Theorem \ref{perfeff} shows that
MYC-M is provable in ACA$_0$: Showing that the class $\mathcal C$ in
the proof of Lemma \ref{mainlem} has positive measure, under the
assumption that $\bigoplus_{i<n} X_i$ is a density-one point, requires
only (the relativized version of) the fact that if $d(Z \mid \mathcal
C_i)=1$ for each of finitely many $\Pi^0_1$ classes $\mathcal
C_0,\ldots, \mathcal C_n$, then $d(Z \mid \bigcap_{i \leq n} \mathcal
C_i)=1$, for which arithmetical induction is more than enough. Then
building $T$ arithmetically is straightforward. So the only possible
issue is Lemma \ref{auxlem}, but all the proof of that lemma requires
is (the relativized version of) the fact that every $\Pi^0_1$ class
$\mathcal P$ of positive measure contains a density-one $1$-random
point $X$. This fact is provable in ACA$_0$, for instance by proving
the Lebesgue Density Theorem as in \cite[Proposition 3.3]{CLWY},
applied to the intersection $\mathcal C$ of $\mathcal P$ with a
sufficiently large $\Pi^0_1$ class of $1$-randoms, then taking
$\sigma_0 \prec \sigma_1 \prec \cdots$ such that the relative measure
of $\mathcal C$ above $\sigma_n$ goes to $1$, and defining $X =
\bigcup_n \sigma_n$.

We thus have the following diagram, where double arrows represent
implications that are known not to reverse, single arrows represent
implications for which we do not know whether the reversal holds, and
all missing arrows not implied by transitivity represent open
questions.

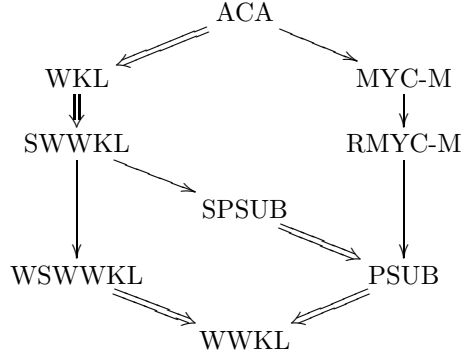
\begin{equation}
\label{mdiag}
\xymatrix@C=15pt@R=12pt{
 & {\textup{ACA}}\ar@{=>}[dl]\ar[dr]\\
{\textup{WKL}}\ar@{=>}[d]  & & {\textup{MYC-M}}\ar[d]\\
{\textup{SWWKL}}\ar[dd]\ar[dr] & & {\textup{RMYC-M}}\ar[dd]\\
 & {\textup{SPSUB}}\ar@{=>}[dr]\\
{\textup{WSWWKL}}\ar@{=>}[dr] & & {\textup{PSUB}}\ar@{=>}[dl]\\
 & {\textup{WWKL}}
}
\end{equation}

\bigskip

Although PTR and PSUB are equivalent, there is a significant formal
difference between them: the former statement is $\Pi^1_3$, while the
latter is $\Pi^1_2$. We often think of a $\Pi^1_2$ statement of the
form $(\forall X)[\Phi(X) \, \rightarrow \, (\exists Y)\Psi(X,Y)]$
with $\Phi$ and $\Psi$ arithmetic as a \emph{problem}. An
\emph{instance} is an $X$ such that $\Phi(X)$ holds, and a
\emph{solution} to this instance is a $Y$ such that $\Psi(X,Y)$
holds. (For the sake of this discussion, let us assume we are working
over the standard natural numbers.) We can do the same with a
statement of this form where $\Psi$ is now $\Pi^1_1$, but we then need
to be careful what we mean by saying that $\Psi(X,Y)$ holds, since
that might be model-dependent. That is, for an instance $X$ in an
$\omega$-model $\mathcal M$, there might be a $Y \in \mathcal M$ such
that $\mathcal M \vDash \Psi(X,Y)$ but $\mathcal N \nvDash \Psi(X,Y)$
for some $\mathcal N \supset \mathcal M$. This fact leads to some
peculiar situations.

For example, let us consider reverse-mathematical analogs of some of
the variants of MYC-M discussed following the proof of Theorem
\ref{compperfthm}. The one where $T$ is required to be a tree with no
dead ends is equivalent to WWKL$_0$, by the same argument as in the
computability-theoretic case, but the one where $T$ is required only
to be an infinite tree is provable in RCA$_0$: It is provable in
WWKL$_0$, but also in RCA$_0$ plus the negation of WWKL, or even just
of WKL, because if $T$ is an infinite tree with no infinite paths,
then it satisfies any universal condition on paths, an issue already
mentioned above Proposition \ref{classrca}. This proof, which is made
possible by the $\Pi^1_3$ form of this statement, feels like something
of a cheat, particularly as it allows us to prove the statement in
RCA$_0$ without giving us any more idea of its computability-theoretic
complexity than we previously had.

A similar kind of cheating allows us to show that even the version
where $T$ is required to be an infinite tree with no isolated paths is
provable in RCA$_0$ (despite the fact mentioned above that every
computable infinite tree all of whose paths are $1$-random has paths
that are not mutually $1$-random). One can define the notion of having
no isolated paths by quantifying over paths, in which case trees with
no paths automatically have no isolated paths. A first-order way to
define this notion is to say that a tree $T$ has an isolated path if
there is a $\sigma \in T$ such that $T$ is infinite above $\sigma$,
and for each $\tau \succcurlyeq \sigma$, at most one of $T$ above
$\tau^\frown 0$ and $T$ above $\tau^\frown 1$ is infinite. However,
even with this definition we can still prove in RCA$_0$ that if $T$ is
infinite but has no infinite paths, then it cannot have an isolated
path: Assume $T$ has an isolated path in this sense, and let $\sigma$
be as above. Then we can produce a path on $T$ computably in $T$ by
beginning with $\sigma$, and given $\tau$, waiting for $T$ to
terminate above some $\tau^\frown i$, then continuing on to
$\tau^\frown (1-i)$. It follows by $\Pi^0_1$ induction that this
process produces an infinite path. Thus for the following proposition
it does not matter which definition of having no isolated paths we
take.

\begin{proposition}
\textup{RCA}$_0$ proves that for each $A$, there is an infinite tree
$T$ with no isolated paths such that for any nonempty finite
collection $\mathcal F$ of pairwise distinct paths through $T$, the
set $\bigoplus_{Y \in \mathcal F} Y$ is $1$-random relative to $A$.
\end{proposition}

\begin{proof}
If ACA holds then so does MYC-M, so we are done. If WKL fails then
there is an infinite tree $T$ with no paths. This tree cannot have an
isolated path, as otherwise it would be able to compute such a path,
and it vacuously satisfies the condition on randomness of joins of
paths. So we can assume that WKL (or even just WWKL) holds but ACA
fails. Then there is a $B$ such that $B'$ does not exist.

Now fix $A$, and let $X = \bigoplus_i X_i$ be $1$-random relative to
$A$. We define a tree $S$ as follows. Fix a $B$-computable enumeration
of $B'$. Also fix an ordering of $2^{<\omega}$. For a tree $R$, say
that $\sigma$ is an \emph{off-node} of $R$ if $\sigma$ is not in $R$
but its immediate predecessor is. Let $S_0$ consist of every initial
segment of $X_{\langle 0,0 \rangle}$. Given $S_{n-1}$, let
$\sigma_0,\sigma_1,\ldots$ be the off-nodes of $S_{n-1}$ in
order. Form $S_n$ by appending to $S_{n-1}$ a copy of $X_{\langle n,k
\rangle}$ with root $\sigma_k$ for each $k$. Let $S = \bigcup_n
S_n$. It is not difficult to see that $S$ can be built computably from
$X$.

Now define $T$ as follows. Start with $S$, and whenever we find at a
stage $s$ that $n$ is in $B'$, truncate $S$ so that if $k<s$ then $T$
is finite above the $k$th off-node of $S_n$. Again, it is easy to see
that such a $T$ can be built $X$-computably. It is also easy to see
that $T$ has no isolated paths.

Let $Y_0,\ldots,Y_m$ be distinct paths on $T$. For each $i \leq m$, it
cannot be the case that $Y_i$ goes through an off-node of $S_n$ for
every $n$, since then we could compute $B'$ from knowing which
off-nodes $Y_i$ goes through, so each $Y_i$ is equal to $\sigma^\frown
X_{j_i}$ for some $\sigma$ and $j_i$. By construction, the $j_i$'s are
pairwise distinct, so $\bigoplus_{i \leq m} X_{j_i}$ is $1$-random,
and hence so is $\bigoplus_{i \leq m} Y_i$.
\end{proof}

Natural $\Pi^1_3$ statements can exhibit even stranger
behavior. Consider for instance the following statement. (See
\cite{ADR} for a formalization of the notion of $2$-randomness.)

\medskip

\textbf{T$\mathbf{2}$R}: For each $A$, there is an infinite tree $T$
such that every path through $T$ is $2$-random relative to $A$.

\medskip

For the same reason as above, T$2$R is implied by the negation of WKL
(or equivalently, RCA$_0$ is the infimum of T$2$R and WKL). It is also
follows from $2$-RAND, the principle stating that for every $A$, there
is a set that is $2$-random relative to $A$, since if $X$ is
$2$-random relative to $A$, then the initial segments of $X$ form the
required $T$. Indeed, T$2$R is equivalent to $\neg \textup{WKL} \vee
2\textup{-RAND}$, since T$2$R and WKL together imply $2$-RAND. Unlike
the previous principles for $1$-randomness, however, T$2$R is not
provable in RCA$_0$, or even in WKL$_0$, because if WKL$_0$ proves
T$2$R then it proves the existence of $2$-randoms, and there are
$\omega$-models of WKL$_0$ not containing any $2$-randoms (e.g., any
$\omega$-model of WKL$_0$ below a $\Delta^0_2$
 PA degree). Nevertheless, suppose we take any model $\mathcal N$ of
$\Sigma^0_1$-PA (the first-order part of RCA$_0$) and consider the
model $\mathcal M$ of RCA$_0$ with first-order part $\mathcal N$ and
second-order part consisting of all $\Delta^0_1$-definable subsets of
$\mathcal N$. Then $\mathcal M$ is not a model of WKL, so it is a
model of T$2$R, and hence any principle that does not hold in
$\mathcal M$ is not implied by T$2$R. Almost all natural principles
not provable in RCA$_0$ that have been studied have this property for
some $\mathcal M$ of this form, so the position of T$2$R in the
reverse-mathematical universe is rather strange. We will discuss
genericity below, but note here that the principle that for each $A$,
there is an infinite tree $T$ such that every path through $T$ is
$1$-generic relative to $A$ has the same properties. These are not the
first examples of these kinds of ``monsters in the reverse mathematics
zoo'': Belanger \cite{Bel1,Bel2} found several natural model-theoretic
principles equivalent to $\neg \textup{WKL} \vee \textup{ACA}$.

With this discussion in mind, it is worth seeking a $\Pi^1_2$ version
of MYC-M. We can be guided here by the proof of Theorem \ref{perfeff},
as well as by the formulation of PSUB. Assume that the
$\Pi^0_1$ classes $\mathcal R_0,\mathcal R_1,\ldots$ in that proof are
nested, i.e., $\mathcal R_0 \subseteq \mathcal R_i \subseteq
\cdots$. The construction of $T$ in that proof ensures that if
$\sigma_0,\ldots,\sigma_n$ are pairwise distinct strings of the same
length, then there is an $m$ such that for all paths
$T(\sigma_0)^\frown X_0,\ldots,T(\sigma_n)^\frown X_n$ through $T$, we
have $\bigoplus_{i \leq n} X_i \in \mathcal R_m$. By taking the
$\mathcal R_m$'s to be the complements of the levels of a universal
Martin-L\"of test, this fact gives us a way to talk about the
$1$-randomness of joins of paths through $T$ without quantifying over
the paths themselves. There are several ways in which we could
formalize this idea as a reverse-mathematical statement, and it is not
clear how much of a difference the exact choice makes to the strength
of the resulting principle. The following seems like a reasonable way
to do it, which is also adaptable to the variants of MYC-M discussed
above, and is weak enough to be as close to MYC-M as possible. Here we
think of a perfect tree as a binary tree with no dead ends and no
isolated paths. Let $T$ be any binary tree. Recall that for $\sigma
\in T$, we write $T_\sigma$ for the tree consisting of all $\tau$ such
that $\sigma^\frown \tau \in T$. For $\sigma_0,\ldots,\sigma_n \in T$,
we write $T_{(\sigma_0,\ldots,\sigma_n)}$ for the tree consisting of
the closure under initial segments of the set of all strings of the
form $\bigoplus_{i \leq n} \tau_n$ where $\tau_0 \in T_{\sigma_0},
\ldots, \tau_n \in T_{\sigma_n}$ are strings of the same length. By a
\emph{bar} for $T$ we mean a finite collection $F$ of pairwise
incompatible elements of $T$ such that every element of $T$ is
compatible with some element of $F$.

\medskip

\textbf{MYC-M$^+$}: Let $A$ be a set and let $S_0,S_1,\ldots$ be
positive trees such that for each $k$, we have $\frac{|S_k \cap
2^m|}{2^m} > 2^{-(k+1)}$ for all $m$. Then there is a perfect tree $T$
with the following property. For any pairwise distinct
$\sigma_0,\ldots,\sigma_n \in T$ of the same length, there are bars
$F_0,\ldots,F_n$ for $T_{\sigma_0},\ldots,T_{\sigma_n}$, respectively,
such that for each $\tau_0,\ldots,\tau_n$ with $\tau_i \in F_i$ for $i
\leq n$, there is a $k$ for which $T_{\tau_0,\ldots,\tau_n}$ is a
subtree of $S_k$.

\medskip

We can similarly define RMYC-M$^+$ by restricting the above to
$n=1$. Clearly MYC-M$^+$ and RMYC-M$^+$ imply MYC-M and RMYC-M,
respectively, and Diagram (\ref{mdiag}) remains unchanged if we
replace MYC-M and RMYC-M by MYC-M$^+$ and RMYC-M$^+$, respectively. We
can also define versions of the weakenings of MYC-M and RMYC-M to
infinite trees with no isolated paths, or just to infinite trees. For
example, we can consider the following principle.

\medskip

\textbf{WMYC-M$^+$}: Let $A$ be a set and let $S_0,S_1,\ldots$ be
positive trees such that for each $k$, we have $\frac{|S_k \cap
2^m|}{2^m} > 2^{-(k+1)}$ for all $m$. Then there is an infinite tree $T$
with no isolated paths satisfying the following property. For any pairwise
distinct $\sigma_0,\ldots,\sigma_n \in T$ of the same length, there
are bars $F_0,\ldots,F_n$ for $T_{\sigma_0},\ldots,T_{\sigma_n}$,
respectively, such that for each $\tau_0,\ldots,\tau_n$ with $\tau_i
\in F_i$ for $i \leq n$, there is a $k$ for which
$T_{\tau_0,\ldots,\tau_n}$ is either finite or is a subtree of $S_k$.

\medskip

We do not know what relationships hold between these various
principles beyond the obvious ones, and the fact that even the
weakening of WMYC-M$^+$ obtained by removing the requirement that $T$
have no isolated paths is not provable in RCA$_0$, by the fact
mentioned above that every computable infinite tree all of whose paths
are $1$-random has paths that are not mutually $1$-random. It might
also be interesting to consider the computability-theoretic and
reverse-mathematical content of versions of Corollary \ref{perfhalf},
given that mutual $1$-randomness is sufficient but not necessary to
ensure Hausdorff distance $1/2$.

\subsection{Mycielski's Theorem for category}
\label{mycsec3}

As noted in Theorem \ref{Ramseyforcat}, Mycielski's Theorem has an
analog for category in place of measure that is much easier to
prove. As in the measure case, we can do so via the effective case, by
considering $1$-genericity and then relativizing. A useful fact here
is the analog of van Lambalgen's Theorem for $1$-genericity in place
of $1$-randomness proved by Yu \cite{Yuvl}. It is straightforward to
check that the proof given in that paper can be carried out in
RCA$_0$.

\begin{theorem}[Mycielski \cite{My}]
\label{myccatthm}
Let $\mathcal C_0 \subseteq (2^\omega)^{n_0}, \mathcal C_1 \subseteq
(2^\omega)^{n_1}, \ldots$ be comeager. Then there is a perfect subset
$P$ of $2^\omega$ such that $P^{n_i} \in \mathcal C_i$ for all $i$.
\end{theorem}

\begin{proof}
By a \emph{finite partial perfect tree} we mean the restriction of a
perfect tree to $\sigma \in 2^{<n}$ for some $n$. It is
straightforward to encode finite partial perfect trees as binary
strings so that if $\sigma$ and $\tau$ encode $S$ and $T$
respectively, then $\tau \succ \sigma$ if and only if $T$ properly
extends $S$. Let $W_0,W_1,\ldots$ list the c.e.\ sets of binary
strings. Let $D_e$ be the set of all $\sigma$ such that $\sigma$
encodes a finite partial perfect tree $T$ and for every tuple
$\tau_0,\ldots,\tau_n$ of distinct leaves of $T$, we have that $\tau =
\bigoplus_{i \leq n} \tau_i$ meets or avoids $W_e$ (i.e., either
$\tau$ has an initial segment in $W_e$, or no extension of $\tau$ is
in $W_e$).

We claim that each $D_e$ is dense, so that if $X$ is sufficiently
generic and $T$ is the union of the finite partial perfect trees
encoded by initial segments of $X$, then $T$ is a perfect tree such
that any join of finitely many pairwise distinct paths through $T$ is
$1$-generic. To prove the claim, fix a string $\sigma$. By extending
$\sigma$ if needed, we can assume that $\sigma$ encodes a tree
$S$. Let $S(\tau_0),\ldots,S(\tau_n)$ be the leaves of $S$. Let $P$ be
the set of all tuples of pairwise distinct pairs $(i,j)$ with $i \leq
n$ and $j \in \{0,1\}$. Let $F_0,\ldots,F_{m-1} \in P$ be pairwise
distinct, and consider the following procedure. For $i \leq n$ and $j
\in \{0,1\}$, let $\mu_{i,j,0}=S(\tau_j)^\frown i$. If we have defined
all the strings $\mu_{i,j,k}$ for some $k < m$, search for strings
$\nu_{i,j}$ for $j \leq n$ and $i \in \{0,1\}$ such that each
$\nu_{i,j}$ extends $\mu_{i,j,k}$, and $\bigoplus_{(i,j) \in F_k}
\nu_{i,j} \in W_e$ (where this join is ordered as in $F_k$). If such
strings are found then let $\mu_{i,j,k+1} = \nu_{i,j}$. Let $Q$ be the
set of all pairwise distinct $F_0,\ldots,F_{m-1} \in P$ such that this
procedure ends up defining the strings $\mu_{i,j,m}$. Take an element
of $Q$ of maximal size, use that element to define the strings
$\mu_{i,j,m}$, extend $S$ by defining $S(\tau_i^\frown
j)=\mu_{i,j,m+1}$, and let $\rho$ be a string encoding this new
tree. It is easy to see that $\rho$ is an extension of $\sigma$ in
$D_e$.

Now, each $\mathcal C_i$ contains the intersection of dense open sets
$\mathcal U_i^0,\mathcal U_i^1,\dots$. Let $V_i^j$ be a set of strings
generating the open set $\{\bigoplus_{k<n_i} Y_k :
(Y_0,\ldots,Y_{n_i-1}) \in \mathcal U_i^j\}$. Each $V_i^j$ is
$\Sigma^{0,A_i^j}_1$ for some $A_i^j$. Let $A=\bigoplus_{i,j}
A_i^j$. Then relativizing the above argument to $A$ produces a perfect
tree $T$ such that for each $i$ and each sequence
$Y_0,\ldots,Y_{n_i-1}$ of distinct paths through $T$, we have
$(Y_0,\ldots,Y_{n_i-1}) \in \mathcal U_i^j$ for all $j$, and hence
$(Y_0,\ldots,Y_{n_i-1}) \in \mathcal C_i$.
\end{proof}

The proof of Theorem \ref{myccatthm} shows that, in contrast with the
second part of Theorem \ref{negcor}, if $X$ is sufficiently generic,
then it computes a perfect tree such that the join of any nonempty
finite collection of distinct paths is $1$-generic. (This distinction
mirrors the set-theoretic one, as, unlike in the case of random reals,
adding a single Cohen generic real does add a perfect set of mutually
Cohen generic reals (see e.g.\ \cite{Ham}).)  A natural question now
is how generic such an $X$ needs to be.  It is not quite enough to
have $X$ be $1$-generic, because Kumabe \cite{Kumabe} showed that
there is a $1$-generic degree with a strong minimal cover, so the
proof of Theorem \ref{negthm} with ``$1$-random'' replaced by
``$1$-generic'' throughout yields the following result.

\begin{theorem}
\label{ptgthm}
There is a $1$-generic that does not compute any perfect tree all of
whose paths are $1$-generic.
\end{theorem}

The sets $D_e$ above are uniformly $\Delta^0_2$, and indeed
can be replaced by uniformly $\Pi^0_1$ dense sets, because if a
$\Delta^0_2$ set $D$ is dense then so is the $\Pi^0_1$ set $C=\{\tau :
(\forall s > |\tau|)(\exists \sigma \preccurlyeq \tau)[\sigma \in
D[s]]\}$, and clearly a set meets $D$ if and only if it meets
$C$. Thus the notion of $\Pi^0_1$-genericity becomes relevant
here. (We will return to this notion in the reverse-mathematical
context below.) We can analyze these sets a bit further, though.

A set of strings $D$ is \emph{pb-dense} if there is a function $f$
that is computable from $\emptyset'$ with a primitive recursive bound
on the use function, such that for each $\sigma$, we have that
$f(\sigma) \in D$ and $f(\sigma) \succcurlyeq \sigma$. A set is
\emph{pb-generic} if it meets every pb-dense set of strings. Downey,
Jockusch, and Stob \cite[Theorem 3.2]{DoJoSto} showed that a degree
$\mathbf{a}$ computes a pb-generic if and only if it is array
noncomputable, which means that for each $g \leq\sub{wtt} \emptyset'$,
there is an $\mathbf{a}$-computable function $h$ such that $h(n) \geq
g(n)$ for infinitely many $n$.

The argument in the proof of Theorem \ref{myccatthm} that the sets
$D_e$ are dense shows that given $\sigma$, we can find an extension
$\tau$ of $\sigma$ in $W_e$ by asking enough existential questions to
determine the elements of $Q$, since once we have an element of $Q$ of
maximal length, we can find such a $\tau$ computably. There clearly is
a primitive recursive bound on the numbers $n$ for which we need to
query $\emptyset'(n)$ to answer all of these questions. Thus each
$D_e$ is pb-dense, and hence any array noncomputable degree can compute a
set that meets all of them. Relativizing this argument and that of
Downey, Jockusch, and Stob \cite{DoJoSto}, we have the following
result.

\begin{theorem}
\label{ancthm}
For any $A$ and any $B$ that is array noncomputable relative to $A$,
there is a $B$-computable perfect tree $T$ such that for any nonempty
finite collection $\mathcal F$ of paths through $T$, the set
$\bigoplus_{Y \in \mathcal F} Y$ is $1$-generic relative to $A$, and
hence the joins of any two finite, disjoint, nonempty collections of
paths through $T$ are mutually $1$-generic relative to $A$.
\end{theorem}

On the reverse mathematics side, the following principles have been
well-studied. (See e.g.\ \cite[Section 9.3]{Hbook} for more on the
reverse mathematics of model theory.)

\medskip

\noindent \textbf{$\mathbf{\Pi^0_1}$G}: For any family of uniformly
$\Pi^0_1$ dense predicates $P_0,P_1,\ldots$ on $2^{<\omega}$, there is
a set $G$ such that each $P_i$ holds of some initial segment of $G$.

\medskip

\noindent \textbf{AMT}: Every complete atomic theory has an atomic
model.

\medskip

It follows from computability-theoretic work of Csima, Hirschfeldt,
Knight, and Soare \cite{CHKS} and Conidis \cite{Con} that $\Pi^0_1$G
implies AMT, and that the two are equivalent over RCA$_0$ together
with $\Sigma^0_2$ induction. Hirschfeldt, Shore, and Slaman
\cite{HSS} showed that AMT does not imply $\Pi^0_1$G over RCA$_0$
alone. They also showed that both principles, while provable in
ACA$_0$, are incomparable with WKL and WWKL. The existence of
pb-generics does not seem to have been studied from this point of
view. It is not difficult to show that it follows from $\Pi^0_1$G, but
we do not know whether it is strictly weaker.

Corresponding to the principles for measure in the previous
subsection, we have the following principles.

\medskip

\noindent \textbf{MYC-C}: For each $A$, there is a perfect tree $T$
such that for any nonempty finite collection $\mathcal F$ of pairwise
distinct paths through $T$, the set $\bigoplus_{Y \in \mathcal F} Y$
is $1$-generic relative to $A$.

\medskip

\noindent\textbf{RMYC-C}: For each $A$, there is a perfect tree $T$
such that any two distinct paths through $T$ are mutually $1$-generic
relative to $A$.

\medskip

\noindent\textbf{PTG}: For each $A$, there is a perfect tree $T$ such
that every path through $T$ is $1$-generic relative to $A$.

\medskip

These also have natural versions that avoid quantification over
paths. (Here a string $\sigma$ \emph{meets or avoids} a predicate $P$
if either $P(\sigma)$ holds or $P(\tau)$ fails to hold for every $\tau
\succcurlyeq \sigma$.)

\medskip

\noindent \textbf{MYC-C$^+$}: For each $A$, there is a perfect tree
$T$ such that for every $\Sigma^0_1$ predicate $P$ on binary strings,
every $n$, and every $k$, there is an $m>k$ such that for every
pairwise distinct $\sigma_0,\ldots,\sigma_n \in T \cap 2^m$, the
string $\bigoplus_{i \leq n} \sigma_n$ meets or avoids $P$.

\medskip

\noindent\textbf{RMYC-C$^+$}: For each $A$, there is a perfect tree
$T$ such that for every $\Sigma^0_1$ predicate $P$ on binary strings
and every $k$, there is an $m>k$ such that for every distinct
$\sigma_0,\sigma_1 \in T \cap 2^m$, the string $\sigma_0 \oplus
\sigma_1$ meets or avoids $P$.

\medskip

\noindent\textbf{PTG$^+$}: For each $A$, there is a perfect tree $T$
such that for every $\Sigma^0_1$ predicate $P$ on binary strings and
every $k$, there is an $m>k$ such that every $\sigma \in T \cap 2^m$
meets or avoids $P$.

\medskip

The argument we gave above that $\Pi^0_1$-genericity is enough for the
proof of Theorem \ref{myccatthm} can be carried out over RCA$_0$ to
show that $\Pi^0_1$G implies MYC-C$^+$. (In the proof of Theorem
\ref{myccatthm}, the existence of the set $Q$ is ensured by bounded
$\Sigma^0_1$-comprehension, which holds in RCA$_0$.) PTG clearly
implies $1$-GEN, the principle stating that for every $A$ there is a
set that is $1$-generic relative to $A$. The latter principle has been
shown by Cholak, Downey, and Igusa \cite{CDI} to be equivalent to the
Finite Intersection Principle (FIP) studied by Dzhafarov and Mummert
\cite{DM}. Day, Dzhafarov, and Miller [unpublished] and Greenberg and
Hirschfeldt [unpublished] showed that AMT implies $1$-GEN.  We have
mentioned that $\Pi^0_1$G is incomparable with WKL and WWKL. The same
is true of $1$-GEN, since it follows from $\Pi^0_1$G, and hence cannot
imply WWKL, and is not implied by WKL because there are
hyperimmune-free  PA degrees, which cannot compute $1$-generics. Thus
we are in a different section of the reverse-mathematical universe as
in the measure case.

For each $1$-generic $A$, there is an $\omega$-model of $1$-GEN all
of whose elements are computable in $A$, by the same argument as in
the first paragraph of the proof of Theorem \ref{psubthm}, with
randomness replaced by genericity. Thus Theorem \ref{ptgthm} has the
following consequence.

\begin{corollary}
There is an $\omega$-model of $\textup{RCA}_0+1\textup{-GEN}$ that is
not a model of \textup{PTG$^+$}.
\end{corollary}

Thus we have the following picture, where the arrows and missing
arrows have the same meaning as in Diagram (\ref{mdiag}).

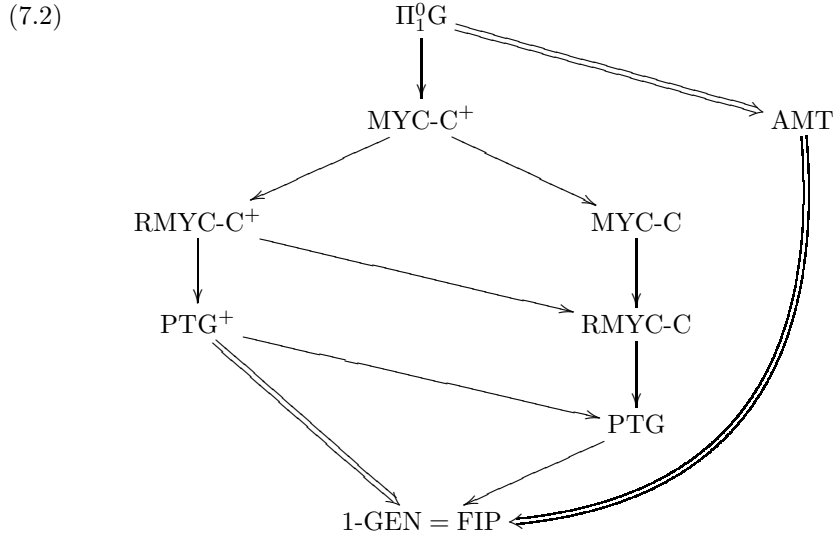
\begin{equation}
\xymatrix{
 & {\Pi^0_1\textup{G}}\ar[d]\ar@{=>}[drr] \\
 & {\textup{MYC-C}^+}\ar[dl]\ar[dr] & & {\textup{AMT}}\ar@{=>}@/^5pc/[ddddll]\\
{\textup{RMYC-C}^+}\ar[d]\ar[drr] & & {\textup{MYC-C}}\ar[d] \\
{\textup{PTG}^+}\ar@{=>}[ddr]\ar[drr] & & {\textup{RMYC-C}}\ar[d] \\
 & & {\textup{PTG}}\ar[dl]\\
 & {1\textup{-GEN}=\textup{FIP}}
}
\end{equation}

\bigskip

As in the measure case, we can also define versions of the above
principles where we replace perfect trees by infinite trees without
isolated paths, or just by infinite trees. We will not study these
principles further here. It might also be interesting to
consider the computability-theoretic and reverse-mathematical content
of versions of Corollary \ref{distance1}, given that relative
$1$-genericity is sufficient but not necessary to ensure Hausdorff
distance $1$.

\section{Cauchy sequences in $(\mathcal S,\delta)$}
\label{compsec}

In this section we pursue the question raised following Theorem
\ref{complthm} regarding the effectiveness of the completeness of
$(\mathcal S,\delta)$. We need to work with particular representatives
of coarse equivalence classes, so, keeping in mind that $\delta$ is
defined (as a pseudo-metric) on sets as well as on elements of
$\mathcal S$, we say that $A_0,A_1,\ldots$ form a
\emph{$\delta$-Cauchy sequence} if for every $k$ there is an $m$ such
that $\delta(A_m,A_n) \leq 2^{-k}$ for all $n>m$. We say that
$A_0,A_1,\ldots$ form a \emph{strongly $\delta$-Cauchy sequence} if
$\delta(A_m,A_n) \leq 2^{-m}$ for all $m<n$. We say that $A$ is a
\emph{limit} of the $\delta$-Cauchy sequence $A_0,A_1,\ldots$ if for
every $k$ there is an $m$ such that $\delta(A_n,A) \leq 2^{-k}$ for
all $n>m$. By Theorem \ref{complthm}, every $\delta$-Cauchy sequence
has a limit. While this limit is not unique, all the limits of a given
$\delta$-Cauchy sequence are coarsely equivalent (i.e., are coarse
descriptions of each other).

For the purposes of reverse mathematics, these definitions need to be
rephrased a bit, because the existence of the $\delta$ function cannot
be proven in RCA$_0$. (Indeed, because it involves the existence of
the upper density, and there is a computable set whose density
computes $\emptyset'$, as shown in \cite[Theorem 2.21]{JS1}, the
existence of the $\delta$ function can be shown to be equivalent to
ACA$_0$.) However, the definitions above do not actually require the
$\delta$ function to exist (i.e., they can be written in classically
equivalent forms that do not mention $\delta$). So when working in
RCA$_0$, we say that $A_0,A_1,\ldots$ form a \emph{$\delta$-Cauchy
sequence} if
\[
(\forall k)(\exists m)(\forall n>m)(\exists i)(\forall j>i)
[\rho_j(A_m \sd A_n) \leq 2^{-k}]
\]
and that it is a \emph{strongly $\delta$-Cauchy sequence} if
\[
(\forall m)(\forall n>m)(\exists i)(\forall j>i)
[\rho_j(A_m \sd A_n) \leq 2^{-m}].
\]
We say that $A$ is a
\emph{limit} of the $\delta$-Cauchy sequence $A_0,A_1,\ldots$ if
\[
(\forall k)(\exists m)(\forall n>m)(\exists i)(\forall
j>i)[\rho_j(A_m \sd A) \leq 2^{-k}].
\]

The reason we consider strongly $\delta$-Cauchy sequences is the
following theorem, which was stated in a different form in
\cite[Theorem 5.9]{HJMS}.

\begin{theorem}[Miller, see {\cite[Theorem 5.9]{HJMS}}]
\label{milthm}
Every strongly $\delta$-Cauchy sequence has a limit computable from the
sequence.
\end{theorem}

The proof of this theorem (as given in \cite{HJMS}) consists of taking
a strongly $\delta$-Cauchy $C_0,C_1,\ldots$, defining the notion of
$C_m$ \emph{trusting} $C_n$ on $J_k$ in terms of the density function
$d_k$ (where $J_k$ and $d_k$ are as in Definition \ref{jddef}), and
then defining a limit $C$ of this sequence by letting $C \restr J_k =
C_n \restr J_k$ for the largest $n \leq k$ such that $C_n$ is trusted
on $J_k$ by all $C_m$ with $m<n$. Since the definition of trusting is
computable, this definition of $C$ can be carried out in RCA$_0$. The
rest of the proof is a verification that $C$ is indeed a limit of the
sequence, using Lemma \ref{factor2}, which can also be carried out in
RCA$_0$. Thus we have the following fact.

\begin{theorem}
\label{milrcathm}
\textup{RCA}$_0$ proves that every strongly $\delta$-Cauchy sequence has a
limit.
\end{theorem}

Theorem \ref{milthm} cannot be strengthened by dropping the
requirement that the $\delta$-Cauchy sequence be strongly
$\delta$-Cauchy, as we now show.

\begin{theorem}
\label{highthm}
There is a computable $\delta$-Cauchy sequence such that every limit
is high.
\end{theorem}

\begin{proof}
Recall the sets $\mathcal R(A)$ from Definition \ref{rdef}. By Theorem
2.19 of \cite{JS1}, if $A \leq\sub{T} \emptyset'$ then $\mathcal R(A)$
is coarsely computable. In fact, the proof of that theorem is uniform,
and hence shows that if $A_0,A_1,\ldots$ are uniformly
$\emptyset'$-computable, then there are uniformly computable sets
$C_0,C_1,\ldots$ such that $C_i$ is a coarse description of $\mathcal
R(A_i)$.

Let $A_0,A_1,\ldots$ be a $\emptyset'$-computable approximation to
$\emptyset''$. Then $\mathcal R(A_0),\mathcal R(A_1),\ldots$ is a
$\delta$-Cauchy sequence with limit $\mathcal R(\emptyset'')$. Let
$C_0,C_1,\ldots$ be as above. Then $C_0,C_1,\ldots$ is a computable
$\delta$-Cauchy sequence that also has $\mathcal R(\emptyset'')$ as a
limit. By Lemma \ref{Rlem2}, any coarse description of $\mathcal
R(\emptyset'')$ is high.
\end{proof}

One way to get an upper bound on the minimal complexity of limits of
computable $\delta$-Cauchy sequences is to consider the complexity of
passing from a $\delta$-Cauchy sequence to a strongly $\delta$-Cauchy
subsequence. If $A_0,A_1,\ldots$ is a computable $\delta$-Cauchy
sequence then, using $\emptyset'''$ as an oracle, for each $k$ we can
find an $m_k$ such that $(\forall n>m_k)(\exists i)(\forall
j>i)[\rho_j(A_{m_k} \sd A_n) \leq 2^{-k}]$. Thus $A_0,A_1,\ldots$ has a
$\emptyset'''$-computable strongly $\delta$-Cauchy subsequence. On the
other hand, it is straightforward to define a computable
$\delta$-Cauchy sequence such that any strongly $\delta$-Cauchy
subsequence computes $\emptyset'$, and to use this construction to show
that the statement that every $\delta$-Cauchy sequence has a strongly
$\delta$-Cauchy subsequence is equivalent to ACA$_0$. We can do
better, however, by also allowing ourselves to replace the elements of
our $\delta$-Cauchy sequence by coarsely equivalent ones.

\begin{theorem}
Let $A_0,A_1,\ldots$ be a $\delta$-Cauchy sequence, and let
$(\bigoplus_i A_i)''' \leq\sub{T} B'$. Then there is a $B$-computable
strongly $\delta$-Cauchy sequence $C_0,C_1,\ldots$ such that for some
sequence $i_0<i_1<\ldots$, each $C_j$ is a coarse description of
$A_{i_j}$.
\end{theorem}

\begin{proof}
We can $B$-computably approximate the function $k \mapsto m_k$ defined
above. We can assume that this function is strictly increasing. Let
$m_k[s]$ be the stage $s$ approximation to $m_k$. Define $C_k$ by
letting $C_k(n)=A_{m_k[s]}(n)$. Then $C_k =^* A_{m_k}$ for all $k$, so
it is easy to check that $C_0,C_1,\ldots$ has the desired
properties. (Recall that $=^*$ is equality up to finitely many
elements.)
\end{proof}

\begin{corollary}
\label{cauchcor}
If $A_0,A_1,\ldots$ is a $\delta$-Cauchy sequence and $(\bigoplus_i
A_i)''' \leq\sub{T} B'$, then $A_0,A_1,\ldots$ has a $B$-computable
limit.
\end{corollary}

\begin{proof}
Let $C_0,C_1,\ldots$ be as in the theorem. Then $C_0,C_1,\ldots$ has
the same limits as $A_0,A_1,\ldots$, so we can apply Theorem
\ref{milthm}.
\end{proof}

There is a gap between Theorem \ref{highthm} and Corollary
\ref{cauchcor}, which we can close as follows.

\begin{theorem}
\label{veryhighthm}
There is a computable $\delta$-Cauchy sequence such that for any limit
$C$, we have $\emptyset''' \leq\sub{T} C'$.
\end{theorem}

\begin{proof}
We first outline the proof and then fill in the details. To begin, we
construct a certain function $f$ such that every function that
majorizes $f$ computes $0'''$. We then partition the natural numbers
into uniformly computable sets $U_0, U_1, \dots$ and define a
computable $\delta$-Cauchy sequence $C_0, C_1, \dots$ with the
following properties for all $n$ and $k$:
\begin{enumerate}

\item[(i)] If $k < f(n)$ or $f(n) = 0$, then $C_k \cap U_n$ is
finite.

\item[(ii)]  If $f(n) > 0$, then  $\orho(C_j \cap U_n) \geq 2^{-n -
1}$ for all sufficiently large $j$.

\end{enumerate}

These properties suffice to prove the theorem.  To see that this is
the case, let $C$ be a limit of $C_0,C_1,\dots$.  We must show that
$\emptyset''' \leq\sub{T} C'$.  Let $n$ be given.  Using a
$C'$-oracle, find a $k$ such that $\delta(C, C_k) < 2^{-n - 2}$.  Such
$k$ exist by the choice of $C$, and one can be computed from $C'$
since the predicate $\delta(C, C_k) 2^{-n-2}$ is c.e.\ in $C'$.  We
claim that $k \geq f(n)$.  This fact is obvious if $f(n) = 0$, so
assume that $f(n) > 0$, and also assume for the sake of a
contradiction that $k < f(n)$.  Hence by (i), $C_k \cap U_n$ is
finite.  Now choose $j$ so large that $\orho(C_j \cap U_n) \geq 2^{-n
- 1}$ and $\delta(C, C_j) < 2^{-n-2}$.  We have that
\[
\delta(C_k, C_j) \geq \delta(C_k \cap U_n, C_j \cap U_n)  =
\delta(\emptyset, C_j \cap U_n) = \orho(C_j \cap U_n) \geq 2^{-n-1}.
\]
However, by the triangle inequality,
\[
\delta(C_k, C_j) \leq \delta(C_k, C) + \delta(C , C_j) < 2^{-n - 2} +
2^{-n-2} = 2^{-n-1}.
\]    
This contradiction proves the claim.  Thus, if we let $g(n)$ be the
first such $k$ that is found, then $g$ majorizes $f$, and we have
$0''' \leq_T g \leq\sub{T} C'$, which completes the proof of the
theorem from the above properties.

It remains to construct $f$, and sets $U_n$ and $C_j$ as above.      

Fix $e_1,e_2,e_3$ such that $W_{e_1}=\emptyset'$,
$W_{e_2}^{\emptyset'}=\emptyset''$, and
$W_{e_3}^{\emptyset''}=\emptyset'''$. Define a function $f$ as
follows. Let $f(3n)=0$ if $n \notin \emptyset'$ and otherwise let
$f(3n)$ be the least $s+1$ such that $n \in W_{e_1,s}$. Let
$f(3n+1)=0$ if $n \notin \emptyset''$ and otherwise let $f(3n+1)$ be
the least $s+1$ such that $n \in W_{e_2,s}^{\emptyset'}$. Let
$f(3n+2)=0$ if $n \notin \emptyset'''$ and otherwise let $f(3n+2)$ be
the least $s+1$ such that $n \in W_{e_3,s}^{\emptyset''}$.  It is not
difficult to see that if $f(k) \leq g(k)$ for all $k$, then
$\emptyset''' \leq\sub{T} g$.

Note that the predicate $f(n) = s > 0$ is a $\emptyset''$-computable predicate
of $n$ and $s$, and hence is $\Delta^0_3$. Thus there is a $\Pi^0_2$
predicate $P(n, s, x)$ such that for all $n$ and $s$,
\[
f(n) = s > 0 \iff (\exists x)P(n, s, x).
\]

We will need the above $x$ to be unique when it exists to prove that
$C_0,C_1,\ldots$ is a $\delta$-Cauchy sequence, which we can achieve
by modifying $P$.  Since $f(n) = s > 0$ is a $\emptyset''$-computable
predicate of $n$ and $s$, we can apply the limit lemma relative to
$\emptyset'$ to obtain a $\emptyset'$-computable function $g_0(n, s,
t)$ such that, for all $n$ and $s$, if $f(n) = s > 0$ then $\lim_t
g_0(n, s, t) = 1$, and otherwise $\lim_t g_0(n, s, t) = 0$.  Now
define $P_0 (n, s, x)$ to hold if $(\forall t \geq x) [g_0(n, s, t) =
1]$ and either $x=0$ or $g_0(n,s,x-1)=0$. Then $P_0$ is a $\Pi^0_2$
predicate, and $P_0 (n, s, x)$ holds if and only if $x$ is minimal
with the property that $(\forall t \geq x) [g(n, s, x) = 1]$.  Hence,
for all $n$ and $s$,
\[
f(n) = s > 0 \iff (\exists x)P_0(n, s, x) \iff (\exists ! x) P_0 (n,
s, x).
\]
Since $\{e : W_e \text{ is infinite} \}$ is $\Pi^0_2$ complete (with
respect to $1$-reducibility), there is a computable function $h(n, s,
x)$ such that for all $n$ and $s$,
\[
f(n) = s > 0 \iff (\exists x)[W_{h(n, s, x)} \text{ is infinite}]
\iff (\exists ! x) [W_{h(n, s, x)} \text{ is infinite}].
\]

Recall that $J_s$ is the interval $[2^z - 1 , 2^{z+1} - 1)$. Let $U_n =
\bigcup_t J_{\langle n , t \rangle}$. We assume that our pairing
function is bijective, so to define $C_j$ it suffices to define $C_j
\restr U_n$ for each $n$. The various values of $n$ will be treated
independently of each other.  To define $C_j \restr U_n$ one might
attempt to make the conclusion of condition (ii) hold whenever $j >
f(n) > 0$.  This is difficult to achieve since the condition $j > f(n)
> 0$ is only $\Delta^0_3$.  To overcome this problem, we require that
$j$ not only bound $f(n)$ but also bound an $x$ witnessing that $j >
f(n) = s > 0$ as in the above displayed formula.  Doing so replaces
the $\Delta^0_3$ condition above by a $\Pi^0_2$ condition and is
crucial for the proof.  This $\Pi^0_2$ condition is true if and only
if there are infinitely many stages $t$ at which we think it is true,
and at each such stage we can add a finite set to $C_j \cap U_n$,
chosen to make progress towards satisfying the conclusion of condition
(ii).

For a nonempty finite set $F$ and a real number $r$ with $0 \leq r
\leq 1$, let $F[r]$ be the shortest initial segment $G$ of $F$ such
that $|G| \geq r |F|$.  Note that $2^{-r} \leq d_w (J_w[r]) \leq
2^{-r+1}$ for all $w$ and all $r \in [0,1]$, where $d$ is as in
Definition \ref{jddef}.

We can now define $C_j \restr U_n$, for which it suffices to define
$C_j \restr J_{\langle n , t \rangle}$ for each $t$.  To do this,
check whether there exist $s$ and $x$ less than $j$ such that $W_{h(n,
s, x), t+1} \neq W_{h(n, s, x), t}$.  If so, let $C_j \restr
J_{\langle n, t \rangle} = J_{\langle n, t \rangle}[2^{-n}]$.
Otherwise, let $C_j \restr J_{\langle n, t \rangle} = \emptyset$.
Clearly the sets $C_0, C_1, \dots$ are uniformly computable.

We now check that conditions (i) and (ii) from the beginning of the
proof hold.  To prove (i), assume that $k < f(n)$.  Then we must prove
that $C_k \cap U_n$ is finite, i.e., $C_k \cap J_{\langle n, t
\rangle}$ is empty for all sufficiently large $t$.  This fact is
clear since there are only finitely many pairs $(s , x)$ with $s$ and
$x$ each less than $j$, and for each such pair the set $W_{h(n, s,
x)}$ is finite.  The case in (i) where $f(n) = 0$ is similar.

To prove (ii), assume that $f(n) = s > 0$.  Fix $x$ such that $W_{h(n,
s, x)}$ is infinite.  Then if $j$ is greater than both $x$ and $s$,
there are infinitely many $t$ such that $C_j \restr C_{\langle n, t
\rangle} = J_{\langle n, t \rangle}[2^{-n}]$.  For such $t$ we have
$d_{\langle n, t \rangle}(C_j \cap U_n) \geq 2^{-n}$.  It follows that
$\overline{d}(C_j \cap U_n) \geq 2^{-n}$.  By Lemma 2.5 it follows
that $\orho({C_j \cap U_n}) \geq 2^{-n-1}$, which completes the proof
of (ii).

It remains to be shown that $C_0,C_1,\ldots$ is a $\delta$-Cauchy
sequence.  Fix $n$.  We claim first that, if $j$ and $k$ are
sufficiently large, then $(C_j \sd C_k) \cap U_n$ is finite.  This is
clear from condition (i) if $f(n) = 0$.  Suppose now that $f(n) = s >
0$.  Let $x$ be the unique number such that $W_{h(n, s, x)}$ is
infinite.  Then if $j$ is greater than both $s$ and $x$, the set $C_j
\cap U_n$ differs only finitely from $\bigcup \{J_{\langle n, i
\rangle} : W_{h(n, s, x), i + 1} \neq W_{h(n, s, x), i} \}$.  Since
the latter set does not depend on $j$, the claim follows. By applying
it to all $m < n$, we see that if $j$ and $k$ are sufficiently large,
then $\bigcup_{m < n} ((C_j \sd C_k) \cap U_m)$ is finite.  Hence, if
$j$ and $k$ are sufficiently large, then $\delta(C_j, C_k) = \orho(C_j
\sd C_k) = \orho( (C_j \sd C_k) \cap \bigcup_{m \geq n} U_m)$.  Note
that $C_j \cap \bigcup_{m \geq n} U_m$ is a union of finite sets $F$
with $d_{\langle m , t \rangle} (F) \leq 2^{-m+1}$ for some $m \geq
n$.  It follows that $\overline{d}(C_j \cap \bigcup_{m \geq n} U_m)
\leq 2^{-n+1}$.  By Lemma 2.5, we have that $\orho(C_j \cap \bigcup_{m
\geq n} U_m) \leq 2^{-n+2}$.  The same fact holds with $j$ replaced
by $k$.  The symmetric difference of two sets is a subset of their
union, and the upper density of their union is at most the sum of
their upper densities.  Hence, for all $n$, if $j$ and $k$ are
sufficiently large, then $\delta(C_j, C_k) \leq 2^{-n+3}$. It follows
that $C_0,C_1,\ldots$ is a $\delta$-Cauchy sequence.
\end{proof}

For the purposes of reverse mathematics, let us take the completeness
of $(\mathcal S, \delta)$ to be the statement that every
$\delta$-Cauchy sequence has a limit. Let HIGH be the statement that
for every $A$ there is a $B$ such that $A'' \leq\sub{T} B'$. More
precisely, we can take this statement as saying that for every $A$
there is a $0,1$-valued binary function $f$ such that for each $e$, we
have that $\lim_s f(e,s)$ exists, and is equal to $1$ if and only if
$\Phi_e^A$ is total. (Basic computability theory, including the notion
of a (partial) $X$-computable function, an enumeration
$\Phi_0^X,\Phi_1^X,\ldots$ of all partial $X$-computable functions,
and so on, can be developed in RCA$_0$ using universal $\Sigma^0_1$
formulas.)

H\"olzl, Raghavan, Stephan, and Zhang \cite{HRSZ} and H\"olzl, Jain,
and Stephan \cite{HJS} studied the principle DOM, which they stated in
terms of the notion of weakly represented families of functions, but
can be equivalently stated as saying that for every $X$, there is a
function that dominates all $X$-computable functions. Martin \cite{mart}
showed that a set is high if and only if it computes a set that
dominates all computable functions. His proof, as given for instance
in \cite[Theorem 2.23.7]{DH}, can be relativized and carried out in
RCA$_0$ to show that DOM and HIGH are equivalent over RCA$_0$. We
state our results in terms of DOM since that is the name used in the
aforementioned papers, in which the authors obtain several results on
the reverse-mathematical strength of this principle. For instance,
they show that DOM implies the cohesive set principle COH over
$\textup{RCA}_0 + \textup{B}\Sigma^0_2$, but does not imply SRT$^2_2$
(stable Ramsey's Theorem for pairs), even over $\omega$-models. They
also show that DOM is restricted $\Pi^1_2$-conservative over RCA$_0$
(see \cite{HRSZ} for the definition), but implies full arithmetic
induction over B$\Sigma^0_2$, and that it is equivalent to several
results in the theory of inductive inference.

Recall that two statements of second-order arithmetic are
\emph{$\omega$-equivalent} if they hold in the same $\omega$-models of
RCA$_0$.

\begin{theorem}
The completeness of $(\mathcal S,\delta)$ is $\omega$-equivalent to
\textup{DOM}.
\end{theorem}

\begin{proof}
Let $\mathcal M$ be a Turing ideal. If every $\delta$-Cauchy sequence
in $\mathcal M$ has a limit in $\mathcal M$, then relativizing the
proof of Theorem \ref{highthm} shows that $\mathcal M$ is a model of
HIGH, and hence of DOM.

Conversely, suppose that $\mathcal M$ is a model of DOM, and hence of
HIGH, and let $A_0,A_1,\ldots$ be a $\delta$-Cauchy sequence in
$\mathcal M$. Then there is a $D \in \mathcal M$ such that $(\bigoplus
A_i)'' \leq\sub{T} D'$. Applying HIGH again, there is a $B \in
\mathcal M$ such that $(\bigoplus A_i)''' \leq\sub{T} D'' \leq\sub{T}
B'$. By Corollary \ref{cauchcor}, $A_0,A_1,\ldots$ has a limit in
$\mathcal M$.
\end{proof}

By Theorem \ref{milrcathm}, the second half of the above proof carries
through in RCA$_0$.

\begin{theorem}
\label{domrcathm}
\textup{DOM} implies the completeness of $(\mathcal S,\delta)$ over
\textup{RCA}$_0$.
\end{theorem}

In the proof of Theorem \ref{highthm}, the sequence $\mathcal
R(A_0),\mathcal R(A_1),\ldots$ is $\delta$-Cauchy because the
approximation to $\emptyset''$ settles on initial segments, i.e., for
each $n$, there is an $s$ such that $A_t(m)=\emptyset''(m)$ for all
$t>s$ and $m<n$. This will not generally be the case over nonstandard
models, however, so that proof does not immediately give us an
implication over RCA$_0$. We do have the following, however.

\begin{theorem}
\label{dombthm}
The completeness of $(\mathcal S,\delta)$ implies \textup{DOM} over
$\textup{RCA}_0 + \textup{I}\Sigma^0_2$.
\end{theorem}

\begin{proof}
Fix a set $A$. Working in $\textup{RCA}_0 + \textup{I}\Sigma^0_2$
together with the completeness of $(\mathcal S,\delta)$, we show that
there is a function dominating all $A$-computable functions. We use
the notation in Observation \ref{hilbert}.

Let $B_{e,n} = \{s : (\forall m<n)[\Phi_e^A(n)[s]\converges]\}$ and let
$C_n = \bigoplus^{\mathcal R}_e B_{e,n}$. Given $k$, we can use
bounded $\Sigma^0_2$ comprehension (which is provable in
$\textup{RCA}_0 + \textup{I}\Sigma^0_2$), to obtain the set $F$ of all
$e \leq k$ such that $\Phi_e^A(m)[s]\diverges$ for some $m$. We can then
use B$\Sigma^0_2$ to obtain a $b$ such that for each $e \in F$, we
have $\Phi_e^A(m)[s]\diverges$ for some $m<b$. Let $n > b$. Then
$B_{e,n}=B_{e,b}=\emptyset$ for all $e \in F$. Furthermore, there is
a $t$ such that $s \in B_{e,n}$ and $s \in B_{e,b}$ for all $s \geq t$
and $e \leq k$ with $e \notin F$, so $\delta(C_b,C_n) \leq 2^{-k}$.

Thus $C_0,C_1,\ldots$ is a $\delta$-Cauchy sequence. Let $C$ be a
limit of this sequence. If $\Phi_e^A$ is total then $B_{e,n} =^* \mathbb
N$ for all $n$, so $C \uhr R_n$ has density $1$. Otherwise, $B_{e,n} =
\emptyset$ for all sufficiently large $n$, so $C \uhr R_n$ has density
$0$. Thus we can define the function $f$ as follows. Given $n$, search
for an $s \geq n$ such that for each $e \leq n$, either
$\Phi_e^A(n)[s]\converges$ or $\rho_s(C \uhr R_n) \leq 1/2$. Such an
$s$ must exist. Let $f(n) = \max\{\Phi_e^A(n) :
\Phi_e^A(n)[s]\converges\}+1$.

If $\Phi_e^A$ is total, then $f(n)>\Phi_e^A(n)$ for all sufficiently large
$n$, since $\rho_s(C \uhr R_n) > 1/2$ for all sufficiently large $s$.
\end{proof}

Notice that by Theorem \ref{domrcathm} and the conservativity of DOM,
the completeness of $(\mathcal S,\delta)$ does not imply
I$\Sigma^0_2$. We do not know whether it implies DOM over RCA$_0$, or
over $\textup{RCA}_0 + \textup{B}\Sigma^0_2$.

It follows from the results on DOM mentioned above and Theorem
\ref{dombthm} that the completeness of $(\mathcal S,\delta)$ implies
both arithmetic induction and COH over $\textup{RCA}_0 +
\textup{I}\Sigma^0_2$. We do not know whether these implications hold
over $\textup{RCA}_0 + \textup{B}\Sigma^0_2$ or, in the case of COH,
over RCA$_0$.

\section{Open Questions}

In this section, we gather some open questions discussed above.

\begin{question}
Without assuming CH, is there a nonempty $\mathcal U \subsetneq
\mathcal S$ that is generated by a Turing invariant set $U$ and is
closed in $(\mathcal S,\delta)$? What if $U$ is required to be an
ideal? (See the discussion at the end of Section \ref{invsec}).
\end{question}

Recall the notions of attractive and dispersive sets from Definition
\ref{attdef}.

\begin{question}
Is there a natural characterization of the attractive degrees that
does not mention Hausdorff distance? Do the notions of being
attractive and being almost everywhere dominating coincide for
c.e.\ degrees? [Raised by T. A. Slaman:] Is every attractive
c.e.\ degree high? Can $A$ be dispersive without it being the case
that almost every set computes a set that is weakly $1$-generic
relative to $A$? Is every $1$-generic set dispersive? Does every
attractive set of hyperimmune-free degree compute a $1$-random set?
\end{question}

\begin{question}
Which finite (or countable) metric spaces can be isometrically
embedded in $(\mathcal{D},H)$?
\end{question}

It seems conceivable that every finite metric space with every
distance equal to $0$, $1/2$, or $1$ is isometrically embeddable in
$(\mathcal{D},H)$. As mentioned above, an interesting test case is the
$0,1/2,1$-valued metric space $\mathcal M$ such that $G_{\mathcal M}$
is a cycle of length $5$.

One way to answer this question would be to show that if $\mathcal A$
and $\mathcal B$ are disjoint finite sets of degrees, then there is a
degree $\mathbf{c}$ such that $H(\mathbf{a},\mathbf{c})=1/2$ for all
$\mathbf{a} \in \mathcal A$ and $H(\mathbf{b},\mathbf{c})=1$ for all
$\mathbf{b} \in \mathcal B$. Corollary \ref{pacor} shows that this is
not the case, but it might still be true within some class of
degrees. The $1$-random degrees seem potentially promising in this
regard.

\begin{question}
Let $\mathcal A$ and $\mathcal B$ be disjoint finite sets of
$1$-random degrees. Must there be a $1$-random degree $\mathbf{c}$
such that $H(\mathbf{a},\mathbf{c})=1/2$ for all $\mathbf{a} \in
\mathcal A$ and $H(\mathbf{b},\mathbf{c})=1$ for all $\mathbf{b} \in
\mathcal B$?
\end{question}

However, even the following basic question remains open.

\begin{question}
If $\mathbf{a}$ is $1$-random, must there be a $1$-random $\mathbf{b}$
that is incomparable with $\mathbf{a}$ such that
$H(\mathbf{a},\mathbf{b})=1$? 
\end{question}

\begin{question}
If $\textbf{a}$ is hyperimmune-free and $\textbf{b}$ is a hyperimmune
PA degree, must $H(\textbf{a},\textbf{b})=1$?
\end{question}

\begin{question}
What is the diameter of $G^{\textup{c}}_{\mathcal D}$? We know by
Corollary \ref{diam3} and Theorem \ref{diam4} that it is  $3$ or $4$.
\end{question}

\begin{question}
Can we improve on the $A'$ bound in Theorem \ref{perfeff}? In
particular, can we replace $A'$ by any set that has  PA degree relative
to A? Relatedly, is MYC-M provable in WKL$_0$? Does it imply WKL$_0$?
Does it imply ACA$_0$?
\end{question}

\begin{question}
Can we improve the bound in Theorem \ref{ancthm}?
\end{question}

\begin{question}
What else can we say about the computability-theoretic and
reverse-mathematical strength of the principles discussed in Section
\ref{mycsec}?
\end{question}

\begin{question}
Does the completeness of $(\mathcal S,\delta)$ imply DOM over RCA$_0$,
or over $\textrm{RCA}_0+\textrm{B}\Sigma^0_2$? Does it imply COH over
either of these systems? Does it imply arithmetic induction over over
$\textrm{RCA}_0+\textrm{B}\Sigma^0_2$?
\end{question}

\end{document}